\documentclass[a4paper, 11pt]{amsart}

\usepackage[backend=bibtex,maxnames=9,hyperref=auto,style=alphabetic]{biblatex}
\usepackage{latex_base}
\usepackage{macros}

\usepackage{graphicx}

\definecolor{JuliaGray}{HTML}{444444}
\definecolor{JuliaGreen}{HTML}{389826}
\definecolor{JuliaRed}{HTML}{D73A49}
\definecolor{JuliaPurple}{HTML}{9558B2}
\definecolor{JuliaBlue}{HTML}{4063D8}

\lstdefinelanguage{Julia}%
{keywords={Framework,BodyHinge,BodyBar,Polytope,AngularFramework,FrameworkOnSurface,DeformationPath,animate,plot,DeformationPath_EdgeContraction,SphericalDiskPacking,SpherePacking,VolumeHypergraph,Dodecahedron},
	morekeywords={abstract,break,case,catch,const,continue,do,else,elseif,end,export,false,for,function,immutable,import,importall,if,in,macro,module,otherwise,quote,return,switch,true,try,type,typealias,%
		using,while},%
	sensitive=true,%
	alsoother={$},%
	morecomment=[l]\#,%
	morecomment=[n]{\#=}{=\#},%
	morestring=[s]{"}{"},%
	morestring=[m]{'}{'},%
}[keywords,comments,strings]%

\usepackage{listings}
\newcommand\juliastyle{\lstset{
		language=Julia,
		basicstyle=\fontsize{10.5}{9}\ttfamily\selectfont,
		tabsize=5,
		aboveskip=10pt,
		belowskip=10pt,   
		morekeywords={self},
		keywordstyle=\bfseries\ttfamily\color{JuliaPurple},
		emph={},
		emphstyle=\ttfamily\bfseries\color{JuliaRed},
		stringstyle=\ttfamily\color{JuliaGreen},
		commentstyle=\ttfamily\color{JuliaBlue},
		frame=tb,
		showstringspaces=false,
		columns=flexible
}}
\lstnewenvironment{julia}[1][frame=none,xleftmargin=.15in]
{
	\juliastyle
	\lstset{#1}
}
{}

\lstnewenvironment{juliafig}[1][]
{
	\juliastyle
	\lstset{#1}
}
{}

\newcommand\juliainline[1]{{\juliastyle\lstinline!#1!}}

\author[Matthias Adrian-Himmelmann]{Matthias Adrian-Himmelmann$^*$}
\thanks{$^*$Institute for Analysis and Algebra, Technische Universit\"at Braunschweig, Germany.}
\address{Matthias Adrian-Himmelmann, Technische Universit\"at Braunschweig, Institut f\"ur Analysis und Algebra, AG Algebra, Universit\"atsplatz 2, 38106 Braunschweig,
	Germany\medskip}

\email{matthias.himmelmann@tu-braunschweig.de}

\date{}



\title[Approximating Continuous Motions]{Approximating Continuous Motions of Geometric Constraint Systems}

\begin{document}
\maketitle

\begin{abstract}
The realization space of geometric constraint systems is given by the vanishing locus of polynomials corresponding to natural geometric constraints. Such geometric constraint systems arise in many real-world scenarios such as structural engineering and soft matter physics. When a geometric constraint system is flexible, it admits continuous motions. The ability to explicitly compute such continuous motions is essential for analyzing the constraint system's quasistatic or elastic properties. However, this task is computationally challenging, even for comparatively simple geometric constraint systems, making numerical strategies attractive.
    
In this article, we present a general numerical framework for approximating continuous motions of geometric constraint systems given by quadratic polynomials. Our approach combines Riemannian optimization with numerical algebraic geometry to construct continuous motions via the metric projection onto the constraint set. By using homotopy continuation, we ensure that the computed motions correspond to genuine solutions of the constraint system and avoid numerical artifacts such as path-jumping. To handle singularities and over-determined systems, we introduce theoretical enhancements including randomization, adaptive step size control and a second-order analysis. These methods are implemented in the Julia package \juliainline{DeformationPaths.jl}, which supports a broad class of geometric constraint systems and demonstrates its robust and effective performance across a wide range of test cases.
\end{abstract}

\section{Introduction}
Geometric constraint systems are used to describe geometric objects that come with natural constraints which can often be represented as polynomials. Such geometric constraints can, for instance, come in the form of distances, angles, coplanarity, or tangency and restrict the possible deformations of the object. A geometric constraint system is called \struc{rigid} if all admissible deformations of a given realization are ambient isometries. Otherwise, they are \struc{flexible}, signifying the existence of a \struc{continuous motion} which deforms the structure and satisfies the underlying constraints. The dual concepts of rigidity and flexibility are two sides of the same coin, both studied within the mathematical field of rigidity theory. Even for simple bar-joint frameworks given by embedded graphs with edge-length constraints, determining their rigidity in $\mathbb{R}^d$ is a coNP-hard problem for $d\geq 2$ \cite{generalizationkempeuniversality}, so deciding their flexibility is NP-hard. There is thus little hope that deciding the rigidity of a geometric constraint system with more general constraints is any easier. 

One of the most extensively studied geometric constraint systems are polytopes. These structures already fascinated ancient Greek philosophers such as Plato and Pythagoras. Accordingly, rigidity theory has its roots in the study of polytopes, reaching back to 1813, when Cauchy proved the rigidity of all three-dimensional simplicial convex polytopes \cite{cauchytheorem}. In 1864, Maxwell established a condition for infinitesimal rigidity based on the kernel of the corresponding rigidity matrix \cite{maxwellcount} that was used by Dehn to extend Cauchy's Theorem to polytopes with rigid faces \cite{dehnstheorem}. However, not all triangulated polytopes are rigid and Bricard showed in 1897 that there are flexible, non-convex polyhedra with infinitesimally rigid faces \cite{raoulbricard}. A realization of \struc{Bricard's octahedron} is depicted in Figure \ref{fig:framework-introductory-pic}(l.). 

\begin{figure}[h!]
	\centering
	\begin{minipage}{0.34\linewidth}
	\includegraphics[width=\linewidth]{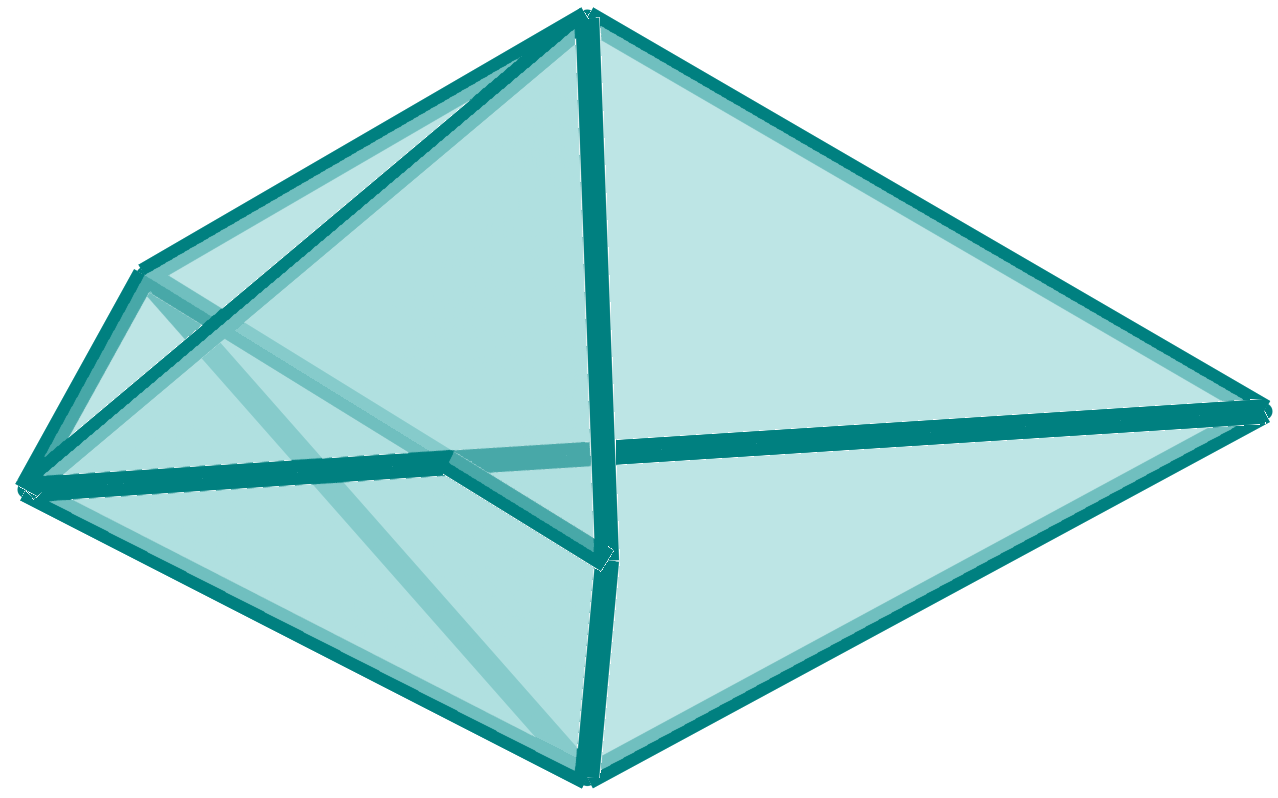}\\[6mm]
	\includegraphics[width=0.96\linewidth]{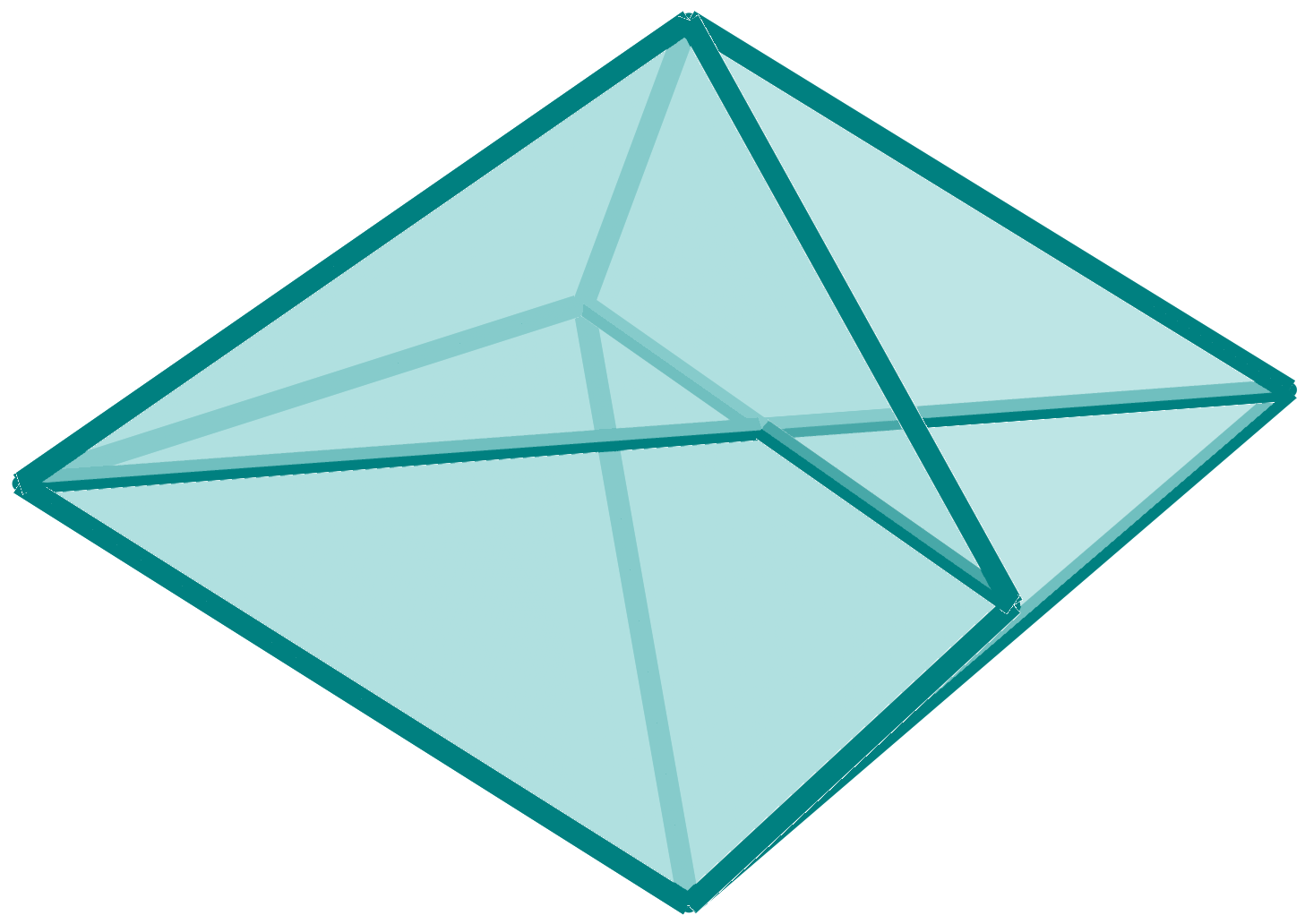}
\end{minipage}\hspace*{13mm}
\begin{minipage}{0.36\linewidth}
\includegraphics[width=\linewidth]{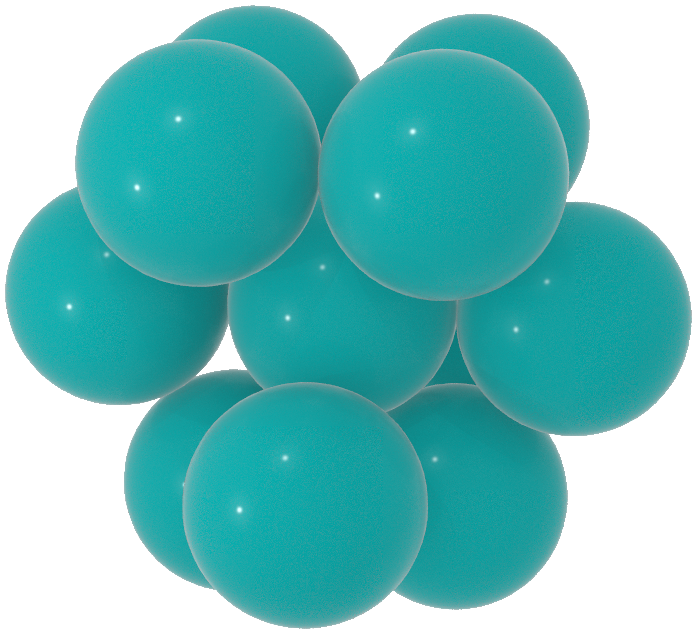}
\end{minipage}
	\caption{\textbf{(left)} Two flexible realization of the generically rigid octahedron, that can be continuously deformed into each other. \textbf{(right)} \textcolor{NiceBlue}{\textit{Hypostatic}} sphere packing, which is rigid, even though the number of contacts between spheres is lower than its nontrivial degrees of freedom (cf. \cite{almostrigidity}).}
	\label{fig:framework-introductory-pic}
\end{figure}

Nowadays, the applications of geometric constraint systems are diverse. Kempe's Universality Theorem \cite{kempesuniversality} shows that any bounded subset of an algebraic curve in $\mathbb{R}^2$ can be traced by an appropriate planar linkage. The possibility to reverse engineer a structure with only rigid bars that follows a prescribed mechanism makes such linkages a particularly compelling model for various natural phenomena. This principle can be used in the design of robot arms with targeted mobility \cite{10.1115/1.4031717}. When adding facet planarity constraints to simple frameworks, certain classes of polytopes are flexible, even though random realizations of convex polytopes are rigid \cite{himmelmannschulzewinter2025rigiditypolytopesedgelength}. Similar to the paradoxically flexible Bricard's octahedron, there are geometric constraint systems which are paradoxically rigid, meaning that they have fewer equations than degrees of freedom (e.g.\ Figure \ref{fig:framework-introductory-pic}(r.)). 

Moreover, geometric constraint systems can be used to study the behavior of granular material \cite{PhysRevLett.94.198001} or filament packings \cite{robustgeometricmodellingcylinderpackings}. Associated techniques can be used to study how viruses self-assemble \cite{viral_selfassembly}.
Accordingly, flexible configurations of geometric objects are studied by artists \cite{penrose-tiling-art}, have applications to quasicrystalline structures \cite{PhysRevX.9.021054}, and have the ability to generate exciting quasistatic properties such as \struc{auxeticity} \cite{geometricauxetics}, having implications for filtration and impact protection. 

For all of these reasons, the ability to explicitly compute continuous motions of geometric constraint systems is highly relevant. Related symbolic computational approaches are available and are actively developed (e.g.\ \cite{HPSchroecker,NACcolorings,pyrigi}). These approaches often rely on exploiting special properties of the system such as symmetry or only work on small examples. Since the underlying problem is NP-hard, the curse of dimensionality quickly becomes an issue. Consequently, numerical approximation strategies are often favored for larger systems. In this context, a specialized, certified curve-tracking approach for polynomial systems has been developed (cf. \cite{burr2025certifiedalgebraiccurveprojections}). Beyond that, local techniques for traversing geometric configuration spaces are available (e.g. \cite{HeatonHimmelmann,almostrigidity,XU2024112939}). 

However, such numerical approximation strategies are usually either not applicable to arbitrary polynomial systems or have other inherent drawbacks.
For instance, such methods often do not guarantee that a continuous motion cannot be explained by numerical inaccuracies in solving the underlying polynomial system and they do not guarantee that approximated motions do not jump between connected components.

\paragraph{Contribution} In this article, we combine the established strategies of Riemannian optimization and homotopy continuation to construct a method that tracks certified solutions along continuous curves. Riemannian geometry is concerned with investigating metric structures on smooth manifolds $\mathcal{M}$. This lets us compute angles, distances and in particular, closest connections between two points on a smooth constraint set known as \struc{geodesics}. We approximate such geodesics with retractions, which can be computed as the projection to the closest point: Starting from $p\in\mathcal{M}$ and a tangent vector $v\in T_p\mathcal{M}$, we take a linear tangent step $p+v$ that is orthogonally projected back to $\mathcal{M}$ (cf. \cite{AbsMal2012}). This approach provides a means to move on manifolds.

Finding the closest point on a manifold leads to a nonlinear optimization problem. By tracking the closest point on $\mathcal{M}$ from $p$ to $p+v$ using a homotopy continuation predictor-corrector scheme, we address both of the previous issues with numerical algorithms: First, the solutions of the polynomial constraint system can be certified using Smale's $\alpha$-theory (cf. \cite{smale-approximate-zero}), so if we find a solution we are sure that it is an actual solution and does not simply come from numerical inaccuracies. Second, the existing literature comes with heuristics and adaptive step size controls implemented to detect and subsequently avoid path-jumping (cf. \cite{robust_numerical_path_tracking, mixedprecisionpathtracking}).

The realization space of a geometric constraint system consists of all points that satisfy the geometric constraints. It is cut out by polynomials, implying that it is smooth almost everywhere. Consequently, we can locally interpret the realization space as a smooth manifold. Based on this observation, we implement the Julia package \juliainline{DeformationPaths.jl} \cite{deformationpaths} for approximating continuous motions by constructing the metric projection using homotopy continuation.

\begin{center}
	\vspace*{2.5mm}
		\begin{minipage}{0.05\linewidth}
			~
			\end{minipage}
	\begin{minipage}{0.21\linewidth}
		\includegraphics[width=1\linewidth]{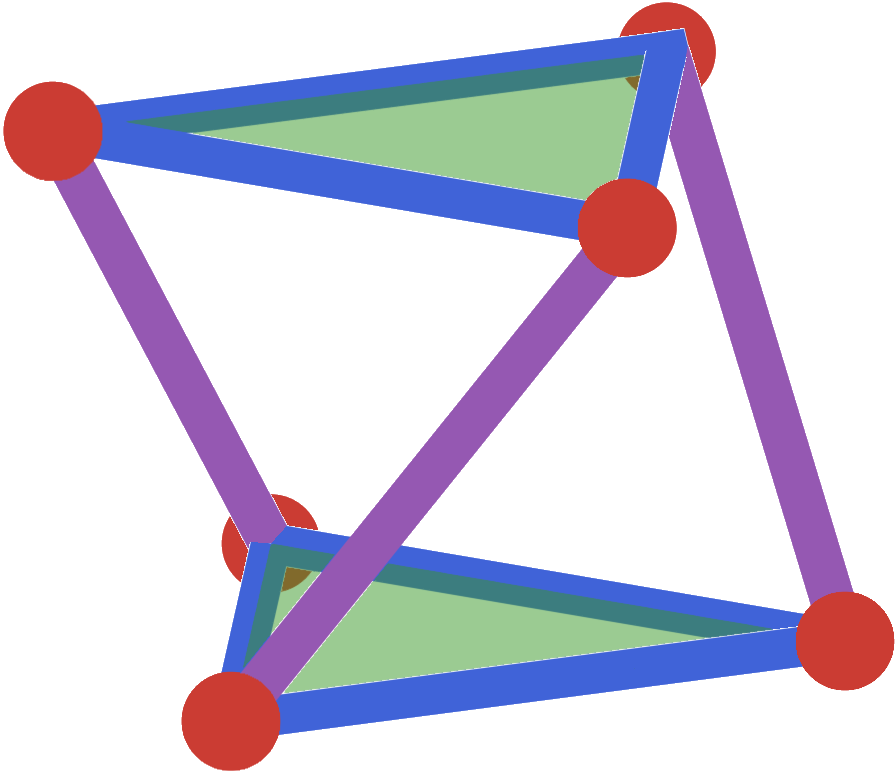}
	\end{minipage}\hspace*{2mm}
		\begin{minipage}{0.4\linewidth}
	 \LARGE\href{https://matthiashimmelmann.github.io/DeformationPaths.jl/}{\color{JuliaGray}\textsc{DeformationPaths.jl}}
	\end{minipage}
		\vspace*{2.5mm}
\end{center}

 The package name is derived from the term \struc{deformation path}, which -- particularly in applied contexts -- is used synonymously with continuous motions, although it is less common in theoretical work. In \cite{HeatonHimmelmann}, it has already been shown that this approach works well over smooth manifolds; however, over general algebraic constraint sets, it comes with several issues. In particular, homotopy continuation struggles with singular start points and over- or under-constrained systems. Therefore, in this article we develop additional techniques for handling these problems. 

Moreover, it is a priori not clear how to compare tangent vectors between different tangent spaces, so we introduce a vector transport in this setting and find unit-speed parametrizations via introducing variable step sizes to the Euclidean distance retraction.
The result is a general-purpose method, which works on arbitrary constraint sets defined by quadratic polynomials. We also show that this class of geometric constraint systems is relatively general and contains many popular and well-studied examples.

\paragraph{Outline} We begin by introducing the notion of a geometric constraint system in Section \ref{section:geometric-constraint-systems} and establish conditions for their rigidity and flexibility. Moreover, we introduce the notion of second-order rigidity, which is used as a sufficient condition for a system's rigidity. In the Appendix Section \ref{appendix:infrig}, we additionally prove several statements about these systems' rigidity properties by generalizing classical results.

In Section \ref{section:riemannian-geometry-numerical-AG}, we construct the Euclidean distance retraction on a general Riemannian manifold, leading to the implicit closest point problem in $\mathbb{R}^d$. Starting in a known solution on the manifold, we compute the Euclidean distance retraction using homotopy continuation, which is a predictor-corrector scheme for solving systems of polynomial equations. 

Section \ref{section:novel-approaches-approximating-deformation-paths} introduces several heuristics and techniques for tackling the issues associated with this approach, such as singularities and over- or under-constrained systems. In this section, we also develop the algorithmic details and briefly present the Julia package \juliainline{DeformationPaths.jl} which implements the Euclidean distance retraction alongside these additional techniques. In the Appendix Section \ref{section:examples-of-gcs} we provide details on this software package and how to use it on several classes of examples.  Finally, Section \ref{section:outlook} gives an outlook and highlights future resarch directions.

\section{Geometric Constraint Systems}
\label{section:geometric-constraint-systems}

Geometric constraint system are used to describe configurations of points, lines, polygons, hyperplanes or other geometric objects that must satisfy a certain set of prescribed relations. Such constraints typically encode distances, angles, incidences, tangencies, or planarity. In other words, geometric constraint system are defined on geometric primitives such as points, lines or rigid bodies and are specified by geometric relationships, such as distances, angles or incidences (cf. \cite{handbookgeometricconstraintsystems}). In this work, we only focus on the class of geometric constraint systems, which are defined on a set of points. This restriction makes it possible to represent geometric constraint systems as a simple and effective data type, while still allowing us to model a wide range of more general geometric constraint systems. As an example, polyhedral rigid bodies can be encoded by a framework on a complate graph on its vertices. Geometric constraint systems thus entail a wide range of objects and are studied in the mathematical areas of computer-aided design (CAD), rigidity theory, computational geometry and kinematics. Some of the key questions in this area are deciding whether a given geometric object is realizable in real space and, if so, whether it is rigid, or admits continuous motions. 

These objects find application in structural engineering and architecture, molecular biology and chemistry, materials modeling, formation control and robotics, sensor networks, and computer graphics. Even though geometric constraint systems are widely used, their definition is slightly vague. In this article, we introduce the following (specialized) definition of geometric constraint systems which will let us prove several theoretical results. 
\newpage
\begin{definition}
	\label{def:constraint-system-in-Rd}
	Given a list of vertices $V=[n]$, a map $p:V\rightarrow \mathbb{R}^d$ and a polynomial constraint set $g:\mathbb{R}^{dn}\rightarrow \mathbb{R}^m$ consisting of quadratic polynomials which are invariant under Euclidean isometries of $\mathbb{R}^d$, a \struc{$d$-dimensional geometric constraint system} $\mathcal{F}$ is a triple $(V,g,p)$. We call $p$ a \struc{realization in $\mathbb{R}^d$} if it satisfies the polynomial equations: $g\big(p(1),\, \dots ,\, p(n)\big)=0$. 
\end{definition}

Established geometric objects that fit this definition are bar-joint frameworks, sphere packings and polytopes. In Section \ref{section:examples-of-gcs}, many more examples are discussed and how they can be instantiated in the Julia package \juliainline{DeformationPaths.jl}, which has been developed for and alongside this article.

Note that this definition of geometric constraint systems slightly differs from the one provided by Sitharam, St. John and Sidman \cite{handbookgeometricconstraintsystems}. In our context, the specialization to quadratic polynomials and the generalization to polynomials that are invariant under Euclidean isometries will prove useful.
Accordingly, from our Definition \ref{def:constraint-system-in-Rd}, it stands out that $g$ is required to be invariant under Euclidean isometries. In other words, for any Euclidean transformation (translations, rotation, reflections) given by an orthogonal matrix $Q\in \mathbb{R}^{d\times d}$ and a vector $\tau\in \mathbb{R}^d$, it holds that
\begin{eqnarray*}
	g\left(Q\cdot p(1)+\tau,\, \dots ,\, Q\cdot p(n)+\tau\right)~=~g( p(1),\, \dots ,\, p(n)).
\end{eqnarray*}

If all possible deformations of a $d$-dimensional geometric constraint system $\mathcal{F}$ infinitesimally arise from Euclidean transformations, we call $\mathcal{F}$ \struc{rigid}.  
Definition \ref{def:constraint-system-in-Rd} has the immediate implication that investigating a geometric constraint system's rigidity is the same as checking whether the realization $p: V\rightarrow \mathbb{R}^d$ is a real-isolated solution of the polynomial system given by $g$ modulo the Euclidean transformations. In other words, there is no continuous curve through $p$ in $\mathbb{R}^{dn}$ that does not correspond to an isometry of $\mathbb{R}^d$. This expression slightly abuses notation by replacing the fixed vector of realizations $\left(p(1),\,\dots ,\, p(n)\right)\in\mathbb{R}^{dn}$ with ``$p$''. In the following, we distinguish the placement $p(i)$ of a vertex $i$ from a vector of coordinates $\Vector{p}_i=(p_{i1},\,\dots,\, p_{id})$, which we treat as the geometric constraint system's variables. 

It is generally not clear whether for a fixed system of (geometric) polynomial constraints $g$ there exist any realizations. Consider, for instance, a triangle with edge lengths $1,2,4$. The triangle inequality requires that $1+2\geq 4$, which is incorrect, so there cannot exist a realization with these edge lengths.

This example relates rigidity to the feasibility problem from real algebraic geometry. Our setting is considerably simpler, since in the definition of a geometric constraint system we assume that we are given a realization. Alternatively, we could also drop that assumption and compute a realization from the incomplete data $(V,g)$, which would make the problem we are facing significantly harder, since the feasibility problem is PSPACE-complete \cite{Canny:CSD-88-439}. 

On the one hand, Definition \ref{def:constraint-system-in-Rd} describes a class of geometric constraint systems and lets us characterize their rigidity. On the other hand, proving that a solution is real-isolated is complicated in general. For this reason, textbooks usually take a different approach to rigidity theory by considering stronger conditions (cf. \cite{handbookgeometricconstraintsystems,handbookdiscreteandcomputationalgeometryschulzewhiteley, combinatorialrigidity}). Consequently, the notion of congruent realizations is introduced in the following. 
\newpage
\begin{definition}
	\label{def:congruence-of-frameworks} 
	Consider two realizations $p,q:V\rightarrow \mathbb{R}^d$ of geometric constraint systems with the same vertex set $V$ and polynomial map $g$. They are called \struc{equivalent} if the polynomial equations are satisfied for both:
	\[g\big(p\big)~=~ g\big(q\big).\]
	The realizations $p$ and $q$ are called \struc{congruent}, if additionally all pairwise distances agree:
	\[\lvert\lvert p(i)-p(j)\lvert\lvert \,=\, \lvert\lvert q(i)-q(j)\lvert\lvert ~~~~\text{for all } \{i,j\}\in {[n]\choose 2}.\]
\end{definition}

Note that this definition is non-standard. If we assume that all polynomials in $g$ arise in the same way, then the typical approach is to investigate the corresponding $g$-stabiliser to determine congruent realizations (cf. \cite{g-rigidity}). However, our setting is more general. In Section \ref{section:examples-of-gcs}, we discuss how to adjust more general geometric constraint systems to our setting, whose space of congruent realizations is potentially larger than Definition \ref{def:congruence-of-frameworks} suggests. The definition of congruence that we give here is inspired by bar-joint frameworks on the complete graph and enables us to generalize several classical results with relative ease.

In $\mathbb{R}^d$, there are infinitely many equivalent realizations and the collection of all these realizations is called the \struc{realization space} and is given by $g^{-1}(0)$. Many are arise from Euclidean isometries composed of rotations and translations; they are characterized by congruent realizations and are called \struc{trivial}. Factoring out the Euclidean motions from the framework's realization space, one obtains its \struc{configuration space}. This notion paves the way for the definition of infinitesimal rigidity appearing in Schulze and Whiteley \cite{handbookdiscreteandcomputationalgeometryschulzewhiteley}, a property that is especially relevant in structural engineering.

\begin{definition}
	\label{def:infinitessimal-rigidity}
	Given a geometric constraint system $\mathcal{F}=(V,g,p)$ in $\mathbb{R}^d$ on a finite list of vertices, consider the polynomial system $g\big(p_{11},\,\dots,\, p_{1d},\, p_{21},\,\dots ,\, p_{nd}\big)=0$
	with variables \mbox{$\Vector{p} = \big(p_{jk}\,:\, j\in \{1,\dots, n\},\, k\in [d]\big)$} representing the realization $p:V\rightarrow \mathbb{R}^d$. We write $\Vector{p}_i$ for the variable coordinates of the vertex $i\in V$. The corresponding \struc{rigidity matrix} with $m$ rows and $dn$ columns can be described in terms of the system's Jacobian:
	\[R_\mathcal{F}(p) = \left(\frac{\partial g_i}{\partial p_{jk}} (p)\right)_{i\in \{1,\dots,m\},\,(j,k) \in [n]\times [d]} \in \mathbb{R}^{m\times d n }.\]
	Then, an \struc{infinitesimal flex} $\Vector{\dot{p}}$ is an element of $\ker{R_\mathcal{F}(p)}$ and an \struc{equilibrium stress} is an element of the cokernel of $R_\mathcal{F}(p)$. 
	We call an infinitesimal flex \struc{trivial} if it extends to an ambient isometry. In the case of $\mathbb{R}^d$, this means that for all vertices $i\in V$ we have $\Vector{\dot{p}}_i=S\cdot \Vector{p}_i+\tau$ for a skew-symmetric matrix $S\in \mathbb{R}^{d\times d}$ (a rotation) and a vector $\tau\in \mathbb{R}^d$ (a translation). If each infinitesimal flex is trivial, we call the framework \struc{infinitesimally rigid}; otherwise, it is \struc{infinitesimally flexible}.
\end{definition}

In particular, the infinitesimal Euclidean transformations lie in the kernel of the rigidity matrix by construction (see Lemma \ref{lem:trivial-flexes-lie-in-kernel}). Since we are not only interested in the infinitesimal flexibility of geometric constraint systems, but actually want to find deformations of the geometric object we are working with, we introduce the notion of a continuous motion.

\begin{definition}
	\label{def:continuous-motion}
	Let $\mathcal{F}=(V,g,p)$ be a $d$-dimensional geometric constraint system. A \struc{continuous motion} is a continuous map $\alpha: [0,1]\rightarrow \mathbb{R}^{dn}$ such that
	\begin{itemize}
		\item $\alpha(0)=p$;
		\item $\alpha(t)$ is a realization for all $t\in [0,1]$;
		\item $\alpha$ is right differentiable in $0$;
		\item $(V,g,p)$ and $(V,g,\alpha(t))$ are equivalent for every $t\in [0,1]$.
	\end{itemize}
	A motion is called \struc{trivial} if $p$ and $\alpha(t)$ are congruent for every $t\in [0,1]$. It is called \struc{nontrivial} otherwise.
	If there exists a neighborhood around $p$ in $\mathbb{R}^{dn}$ in which all equivalent realizations are congruent to $p$, then $\mathcal{F}$ is called \struc{rigid}. Otherwise, we call the geometric constraint system \struc{flexible}.
\end{definition}

The assumption that a continuous motion $\alpha$ is right-differentiable is, in fact, superfluous: In a sufficiently small neighborhood of $p$, all non-congruent points $q\in g^{-1}(0)$ can be joined to $p$ by analytic arcs according to Wallace \cite[Lem.\ 18.3]{algebraicapproximationcurves}. We use this fact in Lemma \ref{lem:existence-continuous-motion} to show that every flexible geometric constraint system possesses a nontrivial continuous motion, which is also analytic. Nevertheless, we retain the assumption in Definition \ref{def:continuous-motion} for the sake of convenience.

Let us now study the relationship between the rigidity and infinitesimal rigidity of geometric constraint systems. It is a classical result from rigidity theory that infinitesimally rigid objects are also rigid. This fact has already been proven for various geometric constraint systems such as bar-joint frameworks (see \cite{Asimow1978TheRO}) or point-hyperplane frameworks (see \cite{EJNSTW}). It turns out that the implicitation is also true in the more general setting of geometric constraint systems from Definition \ref{def:constraint-system-in-Rd}, as we demonstrate in the following theorem.

\begin{theorem}
	\label{thm:inf-implies-rig}
	Infinitesimal rigidity implies rigidity.
	\begin{proof}
		Let $\mathcal{F}=(V,g,p)$ denote a $d$-dimensional geometric constraint system with realization $p:V\rightarrow \mathbb{R}^d$ and assume that $\mathcal{F}$ is infinitesimally rigid. Given any continuous motion \mbox{$\alpha:[0,1]\rightarrow \mathbb{R}^{dn}$}, it is right differentiable in $0$ by definition. Applying the multivariate chain rule, we find that
		\[0 = g(\alpha(t))~~~\Rightarrow~~~ 0 = \frac{\partial}{\partial t^+} g(\alpha(t)) \big\lvert_{t=0} =  \left(R_\mathcal{F}(\alpha(t))\cdot \frac{\partial}{\partial t^+}\alpha(t)\right)\big\lvert_{t=0} = R_\mathcal{F}(p)\cdot \frac{\partial}{\partial t^+}\alpha(0),\]
		so $\frac{\partial}{\partial t^+}\alpha(0)$ lies in the kernel of $R_\mathcal{F}(p)$. By Lemma \ref{lem:rank-rigiditiy-matrix}, there exists an open neighborhood $U(p)$ of $p$ such that $U(p) \cap g^{-1}(0)$ defines a smooth $(\ell_p+1)(2d-\ell_p)$-dimensional manifold for $\ell_p=\dim(\mathrm{aff}(p(1),\dots,p(n)))$. The Euclidean motions on the point set $\{p(1),\dots,p(n)\}$ also define a smooth manifold of dimension $(\ell_p+1)(2d-\ell_p)/2$ and there exists a neighborhood of this manifold which is contained in $U(p)\cap g^{-1}(0)$ (cf. Lemma \ref{lem:rank-rigiditiy-matrix}).
		This means all continuous motions $\alpha$ ending sufficiently close to $p$ are necessarily trivial, implying that $\mathcal{F}$ is rigid.
	\end{proof}
\end{theorem}

\begin{example}
	\label{ex:infrig_rig}
	The converse of Theorem \ref{thm:inf-implies-rig} is not necessarily true. Consider the complete graph $K_3=\left(\{1,2,3\},\,\{12,23,13\}\right)$ with the collinear realization $p:V\rightarrow \mathbb{R}^2$ so that
	\[1\mapsto(0,0),~2\mapsto (1,0) \text{ and }3\mapsto(2,0).\]
	The corresponding bar-joint framework $\mathcal{K}_3=(K_3,p)$ has the rigidity matrix\setcounter{MaxMatrixCols}{20}
	\begin{eqnarray*}
		R_{\mathcal{K}_3}(p)=
		\begin{pmatrix}
			-1 & 0 & 1 & 0 & 0 & 0\\
			-2 & 0 & 0 & 0 & 2 & 0\\
			0 & 0 & -1 & 0 & 1 & 0
		\end{pmatrix},
	\end{eqnarray*}
	which has rank $2$, so $\mathcal{K}_3$ is not infinitesimally rigid by Lemma \ref{lem:rank-rigiditiy-matrix}. Nevertheless, we can factor out the trivial translations by pinning the first vertex to the origin and the trivial rotations by fixing the direction of the edge $\{1,2\}$ to the $x$-axis. Doing so completely pins the edge in place and thus fixes the position of the vertices $1$ and $2$. If we assume that vertex $3$ has the variable position $(x,y)$, then this leaves us with the two edge-length equations
	\begin{eqnarray*}
		x^2+y^2=4\\
		(x-1)^2+y^2=1
	\end{eqnarray*}
	When solving this polynomial system over $\mathbb{R}^2$, we find only one solution $(x,y)=(2,0)$, which demonstrates that $\mathcal{K}_3$ is rigid, despite not being infinitesimally rigid.
\end{example}

Together, Theorem \ref{thm:inf-implies-rig} and Example \ref{ex:infrig_rig} show that examining the kernel of the rigidity matrix provides a simple sufficient condition for the rigidity of a geometric constraint system; however, this condition is not necessary.

\begin{remark}
	\label{remark:quadraticpolynomials}
	Up until this point, the assumption that the constraint system only consists of quadratic polynomials from Definition \ref{def:constraint-system-in-Rd} has not been used. In fact, Theorem \ref{thm:inf-implies-rig} and the theoretical results from Appendix Section \ref{appendix:infrig} hold without it. Yet, the assumption ensures that the rigidity matrix only consists of linear forms, which facilitates the computations in Section \ref{section:novel-approaches-approximating-deformation-paths} and makes the usage of cotangent vectors in the next section (here called ``equilibrium stresses'') significantly more meaningful. It is unclear how to extend the associated results for higher-order rigidity to arbitrary polynomials.
\end{remark}

\subsection{Higher-Order Rigidity}
\label{section:higher-order-rigidity}
Even though the main goal of this article is to compute continuous motions of geometric constraint systems, it is worthwhile to introduce the notion of higher-order rigidity. Doing so will allow us to more accurately describe the local geometry of the configuration space than the simple first-order analysis of the infinitesimal flexes.

The first-order analysis of geometric constraint systems is necessarily incomplete. Even for simple bar-joint frameworks given by embedded graphs equipped with Euclidean edge-length constraints, the converse of Theorem \ref{thm:inf-implies-rig} is not true in general, as can be seen in Example \ref{ex:infrig_rig} and Figure \ref{fig:inf-rigid-prestress-stable-flexible}.

\begin{figure}[h!]
	\centering
	\raisebox{3.2em}{
		\includegraphics[width=0.315\textwidth]{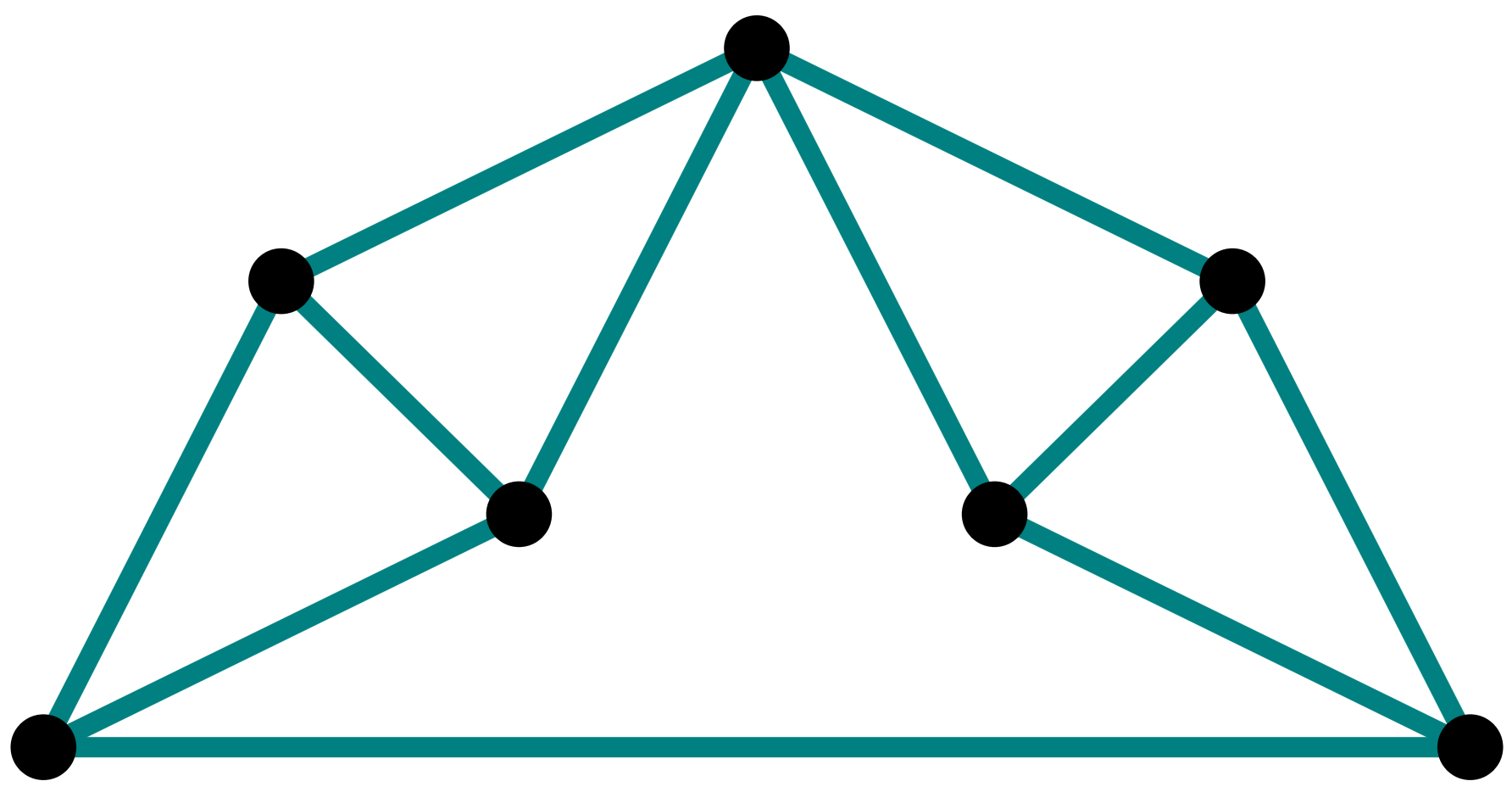}
	}
	\includegraphics[width=0.285\textwidth]{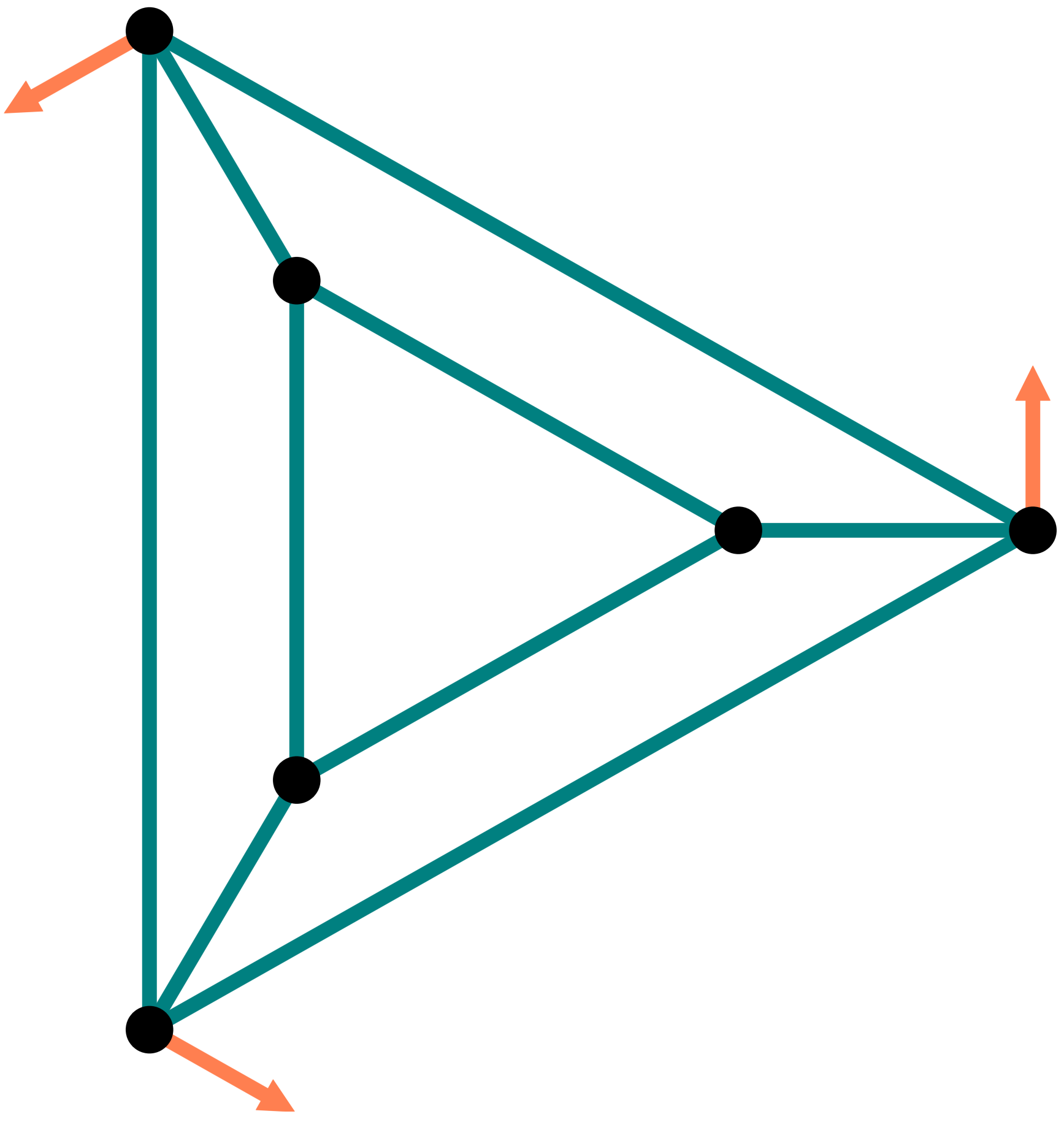}
	\hspace*{4mm}
	\raisebox{3em}{
		\includegraphics[width=0.315\textwidth]{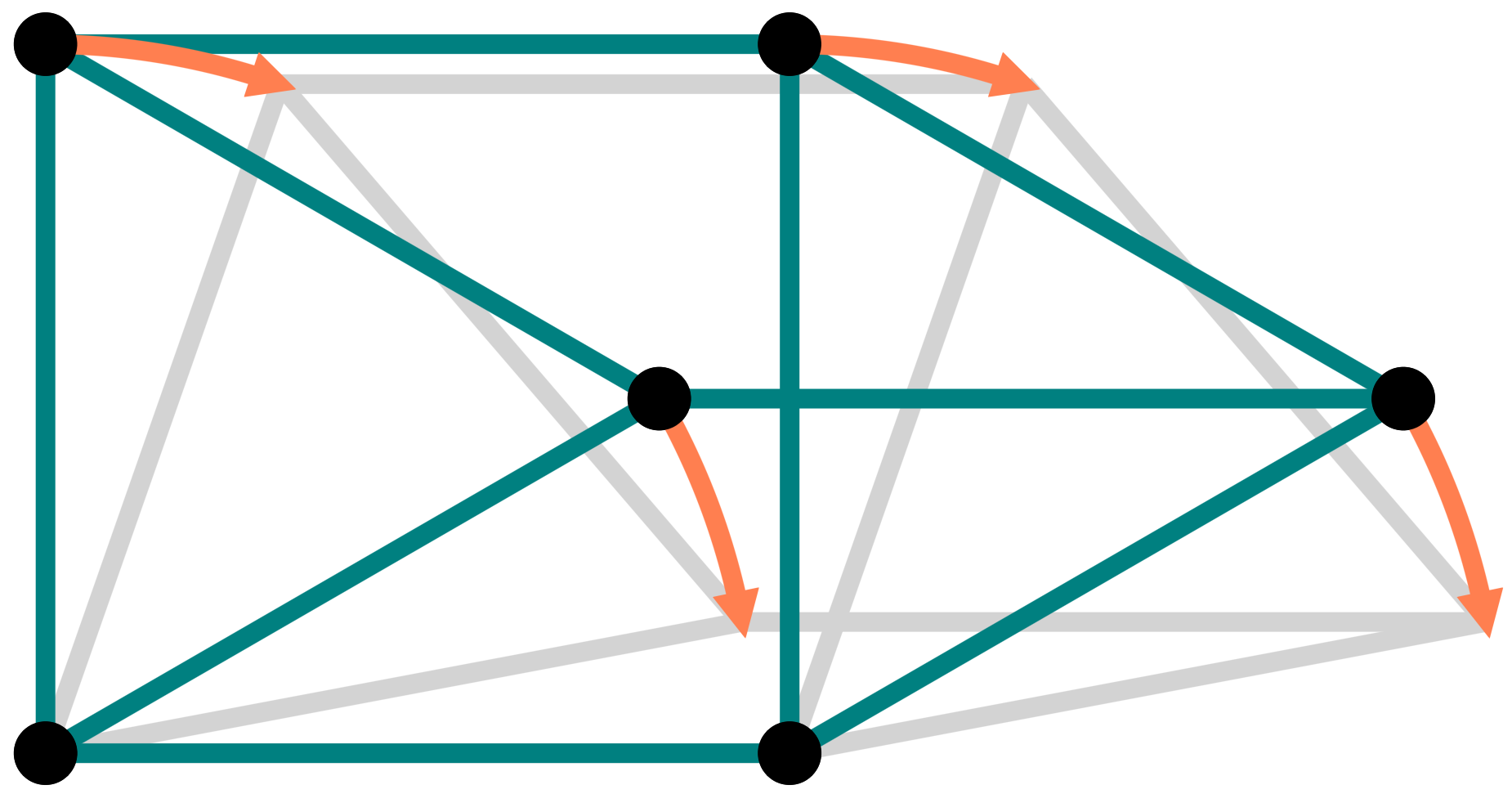}
	}
	\caption{\textbf{(left)} An infinitesimally rigid framework, \textbf{(center)} a locally rigid framework that is not infinitesimally rigid and \textbf{(right)} a flexible framework. The bars of these frameworks are depicted in teal. The coral arrows indicate an infinitesimal flex and the gray shadow depicts a non-congruent realization.}
	\label{fig:inf-rigid-prestress-stable-flexible}
\end{figure}

Robert Connelly was able to extend the linear notion of infinitesimal rigidity to the weaker concept of higher-order rigidity in 1980 by using equilibrium stresses \cite{secondorderrigiditycabled,higherorderrigidity}. Still, this concept only gained traction recently, mainly because relevant examples for which the first-order analysis is insufficient were lacking. Such geometric constraint systems have since been added in the form of classical bar-joint frameworks \cite{secondorderrigiditycabled} polytopes with edge length and coplanarity constraints \cite{himmelmannschulzewinter2025rigiditypolytopesedgelength} and coned 1-skeleta of convex polytopes \cite{energiesonpolytopes}. This leads us to the definition of second-order rigidity (cf. \cite{prestressstabilitypolyhedra}).

\begin{definition}
	\label{def:stress-and-second-order-rigidity}
	Given a geometric constraint system $\mathcal{F}=(V,g,p)$, we call the pair $(\Vector{\dot{p}},\,\Vector{\ddot{p}})$ a \struc{second-order flex} if 
	\begin{align*}
		R_\mathcal{F}(p)\cdot \Vector{\dot{p}}=0 &~~~~~\text{and}\\
		R_\mathcal{F}(\Vector{\dot{p}})\cdot \Vector{\dot{p}} +R_\mathcal{F}(p)\cdot \Vector{\ddot{p}} =0&~
	\end{align*}
	In other words, $\Vector{\dot{p}}$ is an infinitesimal flex of $\mathcal{F}$. A second-order flex is \struc{nontrivial} if $\Vector{\dot{p}}$ is a nontrivial infinitesimal flex. If there is no nontrivial second-order flex, we call $\mathcal{F}$ \struc{second-order rigid}.
\end{definition}

If there exists a nontrivial second-order flex $(\dot{p},\ddot{p})$, then for any equilibrium stress $\omega$ it automatically holds that

\[
	0 ~=~ \omega \cdot \left( R_\mathcal{F}(p)\cdot\Vector{\ddot{p}} + R_\mathcal{F}(\Vector{\dot{p}})\cdot\Vector{\dot{p}}\right) ~=~ \omega \cdot R_\mathcal{F}(p)\cdot\Vector{\ddot{p}} + \omega \cdot R_\mathcal{F}(\Vector{\dot{p}})\cdot\Vector{\dot{p}}
	~=~ \omega \cdot R_\mathcal{F}(\Vector{\dot{p}})\cdot\Vector{\dot{p}}
\]
since $\omega \in \text{coker}(R_\mathcal{F}(p))$ by definition. Conversely, if for a given infinitesimal flex $\Vector{\dot{p}}$ there exists an equilibrium stress $\omega$ such that $\omega^\top \cdot R_\mathcal{F}(\Vector{\dot{p}})\cdot \Vector{\dot{p}}\neq 0$, then there cannot exist a $\Vector{\ddot{p}}$ such that $(\Vector{\dot{p}},\,\Vector{\ddot{p}})$ is a second-order flex. We call such an $\omega$ a \struc{blocking stress} for $\Vector{\dot{p}}$ and we call $\Vector{\dot{p}}$ a \struc{blocked flex}. This approach will be used in Section \ref{section:novel-approaches-approximating-deformation-paths} to rule out certain nontrivial infinitesimal flexes.

\begin{remark}
	Notably, the proof that second-order rigidity implies rigidity is omitted here for the sake of brevity. This implication has been shown to hold for various geometric constraint systems such as bar-joint frameworks \cite{secondorderrigiditycabled}, disk packings \cite{stickydiskrigidity} and polytopes \cite{secondorderpolytoperigid}. The proof typically requires that the underlying constraint system only consists of quadratic polynomials, which rationalizes the presence of this assumption in Definition \ref{def:constraint-system-in-Rd}. It is unclear, whether it is possible to extend this result to arbitrary geometric constraint systems. 
	
	A different interpretation of the condition for second-order flexes is given by the second-order Taylor expansion of the expression $g(\alpha(t))$ in $t=0$ in terms of the variable $t$ for a nontrivial motion $\alpha(t)$. The condition for infinitesimal flexes (cf. Definition \ref{def:infinitessimal-rigidity}) can analogously be derived as a first-order Taylor expansion. For quadratic polynomials, the derivative of $R_\mathcal{\mathcal{F}}(\alpha(t))$ takes a particularly convenient form, as the rigidity matrix only consists of linear forms. Nevertheless, the hope that the second-order condition can be extended to more general classes of geometric constraint system as well is not unreasonable, provided that we can control the behavior of the rigidity matrix. However, the theoretical derivation is left for future research and in this article we only use Definition \ref{def:stress-and-second-order-rigidity} as a heuristic instead of a hard theoretical fact.
\end{remark}

\section{Riemannian Geometry and Numerical Algebraic Geometry}
\label{section:riemannian-geometry-numerical-AG}
In this section, we introduce the general ideas from Riemannian optimization that lead to the construction of the Euclidean distance retraction, which lets us approximate continuous motions. In this particular section, we closely follow \cite{HeatonHimmelmann}, where the Euclidean distance retraction is constructed using homotopy continuation. Contrary to the approach taken in that article, we intend to apply these techniques to approximate continuous motions instead of solving nonlinear optimization problems. As we will later see, there are several distinguishing factors between these two settings, that provide advantages as well as pitfalls. For instance, particular types of singularities that arise in geometric constraint systems need to be treated with caution (cf. Section \ref{section:singularities}), it becomes necessary to move along the constraint set with fixed speed (cf. Section \ref{section:vector-transport}) and blocking stresses can help in computing feasible tangent vectors (cf Section \ref{section:blocked-flexes}). All of these concepts are highly relevant in this article, yet they have not been discussed in \cite{HeatonHimmelmann}. In addition, we briefly describe the functionality as used in our Julia package \juliainline{DeformationPaths.jl}.

\subsection{Euclidean Distance Retraction}
\label{section:Euclidean-distance-retraction}
 As per Definition \ref{def:constraint-system-in-Rd}, geometric constraint systems are given by polynomial constraints which cut out a real algebraic variety. Hence, they may contain singularities. A variety's collection of regular points forms a manifold and the set of singularities again defines an algebraic variety \cite[p.\ 26f.]{nonlinearalgebraandapplications}. According to Hartshorne \cite[Thm.\ 5.3]{hartshorne}, the singular locus is a proper closed subset. Thus, for irreducible varieties $V$ it holds that $\dim \text{Sing} (V)<\dim V$.

Hence we can assume that -- at least locally -- the configuration space of a geometric constraint system is a smooth manifold. To compute angles and distances and to define the concept of a geodesic, we give the realization space a metric structure. A \struc{Riemannian manifold} is a pair of a smooth manifold $\mathcal{M}$ and a \struc{Riemannian metric} $g_p: T_p\mathcal{M}\times T_p\mathcal{M}\rightarrow \mathbb{R}$ that varies smoothly and assigns to each $p\in \mathcal{M}$ a positive-definite symmetric bilinear form. Since the configuration space of geometric constraint systems are naturally embedded in $\mathbb{R}^{dn}$, they inherit the Euclidean \struc{pullback metric} as a submanifold \cite[Lem.\ 2.11]{lee2018}. 

Ultimately, our goal is to compute (mostly) smooth trajectories in the configuration space of geometric constraint systems.
To be physically meaningful, these paths should follow closest connections between two configurations, i.e., \struc{geodesics}. Otherwise, additional energy would be required, which is a priori not available in our system. This relates the problem of computing continuous motions to finding geodesics of an embedded submanifold.

Even for submanifolds of Euclidean space that are defined by a single sparse polynomial, exactly computing geodesics is a nontrivial task.
In most cases, the corresponding second-order differential equation cannot be solved analytically. Approximating geodesics numerically is also undesirable, since this differential equation is often stiff and even sophisticated numerical approximation schemes may diverge. Consequently, we relax our goal to constructing a second-order approximation for geodesics.

\begin{definition}\label{def:retraction}
Let $\mathcal{M} \subset \mathbb{R}^{dn}$ be a manifold and $p \in \mathcal{M}$ a point. A \struc{(first-order) retraction} at $p$ is a smooth map $R_p:T_p \mathcal{M} \rightarrow \mathcal{M}$ from the tangent space $T_p\mathcal{M}$ to the manifold $\mathcal{M}$ satisfying
\[R_p(0) = p~\text{ and }~d(R_p)(0):T_p \mathcal{M} \to T_p \mathcal{M} \text{ is the identity map}.\]
A \struc{(local) retraction at p} is a retraction defined in some neighborhood of $0 \in T_p \mathcal{M}$.
A retraction is called \struc{second-order} if, in addition, 
\begin{equation*}
    \frac{d^2}{dt^2} R_p(tv) \big |_{t=0} \in N_p \mathcal{M}
\end{equation*}
for all $\Vector{v} \in T_p \mathcal{M}$ and the Euclidean normal space $N_p\mathcal{M}=\{\Vector{n}\in \mathbb{R}^{dn}\,:\,\langle \Vector{n},\, \Vector{v}\rangle=0~\forall \Vector{v}\in T_p\mathcal{M}\}$.
\end{definition}

According to Absil and Malick \cite{AbsMal2012}, finding the \struc{closest point} with respect to Euclidean distance defines a local, second-order retraction on any submanifold of Euclidean space, which we call the \struc{Euclidean distance retraction}.

\begin{theorem}[cf. \cite{AbsMal2012}]\label{thm:MAIN:euclidean-distance-is-a-retraction}
	Let $\mathcal{M} \subset \R^{dn}$ be a smooth manifold or real algebraic variety. For any smooth point $p \in \mathcal{M}$, define the relation $R_p \subset T_p \mathcal{M} \times \mathcal{M}$ by
	\begin{equation*}
		R_p = \{ (v, u) \in \R^{dn} \times \R^{dn} \,\, : \,\, u \in \text{arg}\min_{y \in \mathcal{M}} ||p + v - y|| \}.
	\end{equation*}
	There exists a neighborhood $U$ of $0$ in $T_p \mathcal{M}$ such that $R_p$ defines a local, second-order retraction. 
\end{theorem}

The Euclidean distance retraction $R_p$ finds application in Riemannian optimization schemes over general, implicitly defined manifolds (cf. \cite{AbsMal2012, boumal2020intromanifolds,HeatonHimmelmann}). Two exemplary Euclidean distance retractions on a torus of revolution are depicted in Figure \ref{fig:retractionontorus}.

\begin{figure}[h!]
	\centering
	\includegraphics[width=0.51\linewidth]{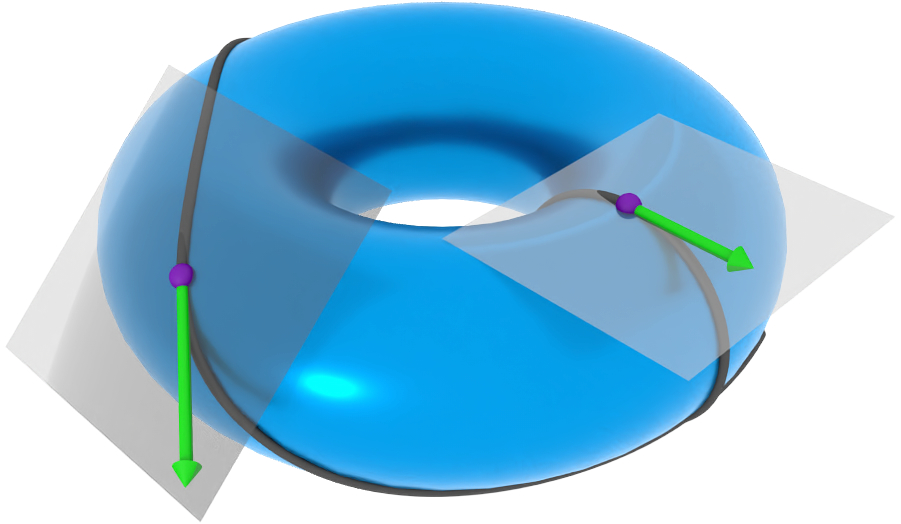}
	\caption{Two Euclidean distance retractions on the torus.}
	\label{fig:retractionontorus}
\end{figure}

The closest point problem of a parametric $u\in \mathbb{R}^{dn}$ to the manifold $\mathcal{M}=g^{-1}(0)\subseteq \mathbb{R}^{dn}$ with respect to Euclidean distance gives rise to a constrained optimization problem. Using the variables $(x,\,\lambda)\in \R^{dn}\times \R^m$ for the number of equations $m$, we can formulate it as the Lagrange multiplier problem

\[\mathcal{L}(x,\lambda; u) = \frac{1}{2}\sum_{i=1}^{dn} (x_i-u_i)^2 + \lambda^T\cdot g(x).\]
The zeros of $\mathcal{L}$'s gradient describe the critical points of the underlying optimization problem \cite[Thm. 12.1]{Nocedal2006}, which is encoded by the polynomial system $0=\nabla_{x,\lambda}\mathcal{L}(x,\lambda;u)$ generated by differentiating $\mathcal{L}$ with respect to $x$ and $\lambda$.

Apart from special cases such as matrix manifolds \cite{AbsMal2012}, the Euclidean distance retraction does not admit an explicit formula. Thus, it is desirable to develop \mbox{techniques} for numerically approximating the Euclidean distance retraction on arbitrary, implicitly defined manifolds, which typically arise in the context of geometric constraint systems. Contrary to geodesics, where numerical approximation techniques may fail due to the stiffness of the underlying second-order differential equations, the specific setting of our problem gives us access to a technique called \struc{homotopy continuation}, which allows us to track parametrized homotopies across $\mathcal{M}$. The connection between these two areas has already been established in \cite{HeatonHimmelmann}.

\subsection{Homotopy Continuation}
\label{section:Homotopy-Continuation}

In this section, we introduce a path-tracking method called \struc{homotopy continuation}. By following the solutions $x(t)$ of a parametrized polynomial system $G(x(t);u(t))=0$ with $t\in [0,1]$ from the parameter $u(1)$, where the solutions $x(1)$ are known, to $u(0)$. This procedure yields the previously unknown solution $x(0)$. Tracking the homotopy $G(x,u(t)):\mathbb{R}^{dn}\times [0,1]\rightarrow \mathbb{R}^{dn}$ is reliant on a combination of a predictor step (e.g.\ Euler's method) and a corrector step (e.g.\ Newton's method). In Figure \ref{fig:predictor-corrector}, a schematic representation of the homotopy continuation predictor-corrector algorithm is depicted.

\begin{figure}[h!]
	\centering
	\includegraphics[width=0.6\linewidth]{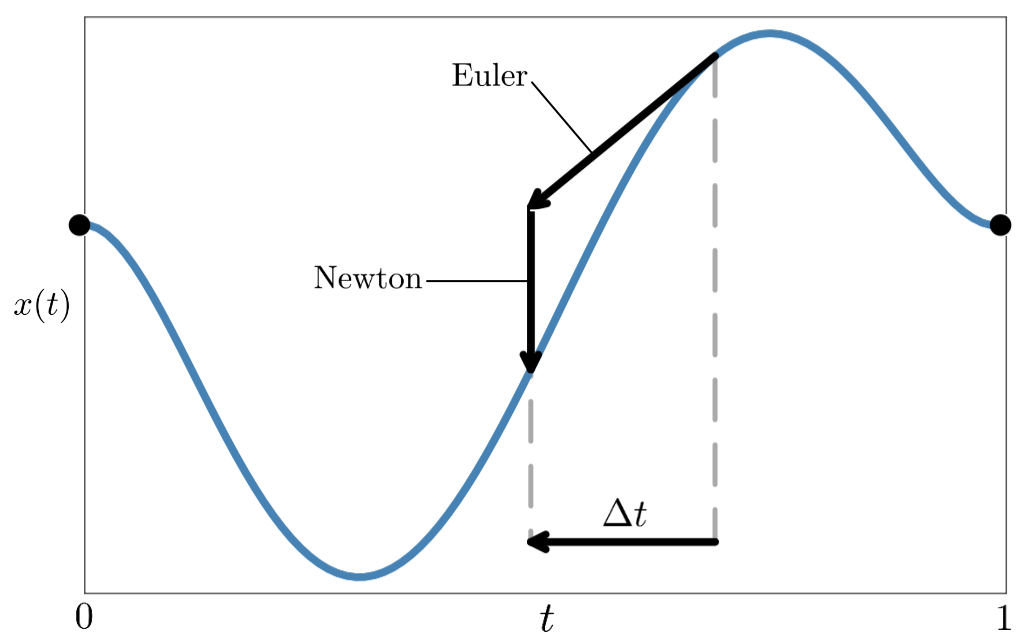}
	\caption{A pictogram of the predictor-corrector scheme at the heart of homotopy continuation.}
	\label{fig:predictor-corrector}
\end{figure}

Primarily, this technique is applied to square polynomial systems that possess 
$s$ isolated complex solutions. Tracking all solutions simultaneously over complex projective space comes with extensive theoretical guarantees, such as smooth solution paths \cite[Thm.\ 7.1.1]{sommesewampler} and that all isolated nonsingular solutions are found if $G(x;u(1))=0$ has sufficiently many solutions \cite[Thm.\ 8.3.1]{sommesewampler}.

In order to solve the implicit closest point problem given by the Lagrange multiplier system \mbox{$\nabla _{x,\lambda}\mathcal{L}(x,\lambda; u)=0$} through the means of homotopy continuation, we make the crucial observation that taking $u(1) = p$ produces a system with known solution $x = p$, where $\lambda$ can be chosen arbitrarily from the cokernel of the Jacobian $\nabla_x g(p)$. Consequently, we can start at the solution $x(t)=p$, discretize the interval $[0,1]$ using an adaptive step size control (e.g.\ \cite{mixedprecisionpathtracking}) and then apply a predictor-corrector scheme to track the solution to the parameter $u(0)=p+v$, yielding a system of equations whose solutions include $(R_p(v),\lambda^*)$ for some $\lambda^*$ that we discard and the Euclidean distance retraction $R_p$ (see Section \ref{section:Euclidean-distance-retraction}).

The expression $\nabla _{x,\lambda}\mathcal{L}(x(t),\lambda(t);u(t))=0$ necessarily remains satisfied throughout the homotopy. For this reason, we can predict the next point using Euler's method and subsequently correct it back the homotopy via Newton's method (cf. \cite{davidenkoode}). This approach is reliable on the invertability of the Hessian $\nabla^2 _{x,\lambda}\mathcal{L}(x(t),\lambda(t);u(t))$ corresponding to $\mathcal{L}$.

While this assertion does not need to hold for arbitrary parameter paths $u(t)$,
it is possible to show that the predictor-corrector path-tracking algorithm applied to the closest point system $\nabla_{x,\lambda}\mathcal{L}=0$ converges to the Euclidean distance retraction, provided that the parameter path $u(t)$ stays sufficiently close to the manifold $\mathcal{M}$:

\begin{theorem}[cf.\ \cite{HeatonHimmelmann}]\label{thm:convergence-result}
	Let $g:\R^{dn} \to \R^m$ be three times continuously differentiable. Let $\mathcal{M} = g^{-1}(0) \subset \R^{dn}$ be a smooth manifold such that the Jacobian $dg_x$ has full rank for every $x \in \mathcal{M}$, and let $u(t)$ be a smooth path in $\R^{dn}$ staying within the tubular neighborhood of $\mathcal{M}$ for which there is a unique closest point. For each $t$, let $x(t)$ be the closest point to $u(t)$. If the path $u(t)$ satisfies
	\begin{equation*}\label{eqn:hypersurface-inequality}
		\mathrm{dist}(u(t), \, x(t)) < \frac{1}{|S(t)|}\sum_{i\in S(t)} \frac{|\nabla g_i(x(t))|}{|d^2g_i(x(t))|}
	\end{equation*}
	for all $t\in [0,1]$ where $S(t) = \{ i : |d^2g_i(x(t))| \neq 0 \}$,
	then there exists a partition $0=t_0<t_1<\cdots<t_N < t_{N+1}=1$ such that each of the $N+1$ applications of the predictor-corrector path-tracking algorithm consisting of Newton's and Euler's method converges and outputs the correct point $R_p(v)$.
\end{theorem}

In \cite{HeatonHimmelmann}, three different types of homotopy continuation approaches are tested for performance and robustness: A single application of Newton's method to correct $\nabla \mathcal{L}_{x,\lambda}(p,0;p+v)$ to $\mathcal{M}$, multiple iterations of Newton's method on a discretization of $[0,1]$ and a combination of Euler's method and Newton's method on a discretization of $[0,1]$. While in the cases where Newton's method converges, the direct application of Newton's method is undoubtably faster than the other two approaches, the predictor-corrector scheme adds an additional reliability that is worth the decreased computational efficiency in our setting. In particular, that article demonstrates that it is not difficult to construct non-degenerate examples for which Newton's method diverges or converges to the wrong point when applied directly to the closest point homotopy without discretizing $[0,1]$ or using a predictor scheme. 

There are various implementations of this homotopy continuation algorithm available (cf. \cite{Bertini, Hom4PSArticle, PHCpack}). In this paper, we rely on \juliainline{HomotopyContinuation.jl} \cite{HC.jl}, which has been written in the high-performance dynamic programming language \juliainline{Julia}.

In particular, that package has implemented heuristics reliant on an adaptive step size control and local geometric information to prevent path-jumping  \cite{robust_numerical_path_tracking, mixedprecisionpathtracking} and comes with certification techniques to check whether the constraints are indeed satisfied using Smale's $\alpha$-theory \cite{smale-approximate-zero}. This is particularly useful in our setting, as we need to ensure that a given numerical solution corresponds to an actual solution and because discontinuous jumping between connected components constitutes non-physical behavior that we want to avoid. In addition, there is a novel approach for certified curve-tracking (cf. \cite{burr2025certifiedalgebraiccurveprojections}) that extends Smale's $\alpha$-theory to one-dimensional solution sets, making it highly relevant for our setting.

\section{Novel Approaches for Approximating Continuous Motions}
\label{section:novel-approaches-approximating-deformation-paths}

We now apply the construction for smooth manifolds from Section~\ref{section:riemannian-geometry-numerical-AG} to compute continuous motions of geometric constraint systems (see Definitions~\ref{def:constraint-system-in-Rd} and~\ref{def:continuous-motion}). Since the realization space is rarely a smooth manifold, this construction must be extended to handle singularities. This section addresses this and related issues by developing and applying several heuristics and theoretical tools.

Assume that we begin in a flexible realization $\Vector{p}\in g^{-1}(0)$ with a nontrivial infinitesimal flex $\Vector{\dot{p}}\in T_{\Vector{p}}g^{-1}(0)$ that extends to a continuous motion. In that case, the Euclidean distance retraction $R$ from Section \ref{section:Euclidean-distance-retraction} can be used to approximate and discretize the curve \mbox{$t\mapsto R_{\Vector{p}}(t\cdot \Vector{\dot{p}})$} for $t\in [0,\alpha]$ with some $\alpha>0$ that completely lies in the configuration space $g^{-1}(0)$ via homotopy continuation (cf.\ Section \ref{section:Homotopy-Continuation}). From the new point $R_{\Vector{p}}(\alpha\cdot\Vector{\dot{p}})$, we iteratively apply this procedure again. 
This approach has already proven itself as a reliable strategy for approximating continuous motions (e.g.\ in \cite{pyrigi,HeatonHimmelmann,robustgeometricmodellingcylinderpackings,secondorderpolytoperigid}). However, there are several inherent challenges that come from this strategy:
\begin{itemize}
	\item Newton's method typically assumes that $g$ is a square systems of equations whose Jacobian has full rank,
	\item if we want the continuous motion to look smooth, the Euclidean distance retraction should roughly have the same length in each step and should match the previous Euclidean distance retraction step free of cracks and kinks,
	\item algebraic sets may contain singularities, so we need to have a means to escape them and
	\item we need to ensure that the flex $\Vector{\dot{p}}$ extends to a continuous motion.
\end{itemize}

In this section, we systematically investigate these challenges and formulate corresponding mathematical and algorithmic approaches.

In the existing literature, there are several alternative approaches for numerically approximating continuous motions (e.g.\ \cite{almostrigidity,XU2024112939}) that also rely on predictor corrector schemes based on Newton's method. While these strategies have means to address the divergence of Newton's method, they have the same issues when dealing with singularities or when attempting to iteratively smoothly extend the continuous motion. This makes the broad investigation from this section worthwhile.

To help navigate this section, we first formulate Algorithm \ref{alg:approx-deformation-paths} for approximating continuous motions based on the Euclidean distance retraction in pseudocode. This algorithm is implemented in the Julia package \juliainline{DeformationPaths.jl}. Concrete implementation details and functionalities are discussed in Section \ref{section:DeformationPaths.jl} and Appendix Section \ref{section:examples-of-gcs}. 

\renewcommand{\thealgocf}{ACM}
\begin{algorithm}[h]
	\SetAlgoLined
	\SetKwInOut{Input}{Inputs}
	\SetKwInOut{Output}{Output}
	\DontPrintSemicolon
	\Input{Constraints $g:\mathbb{R}^{dn}\rightarrow \mathbb{R}^m$ and realization space $\mathcal{M}=g^{-1}(0)\subseteq \mathbb{R}^{dn}$,\\flexible realization $\Vector{p}\in \mathcal{M}$,\\initial step size $\alpha>0$,\\number of steps $N\in \mathbb{N}$,\\implicit description of the Euclidean distance retraction $R:T\mathcal{M}\rightarrow \mathcal{M}$.}
	\Output{Discretization of a nontrivial continuous motion $\Vector{x}:[0,1]\rightarrow g^{-1}(0)$.}
	\;
	$\Vector{x}_0\,\leftarrow\, \Vector{p}$\;
	$\Vector{v}_0\, \leftarrow\,$ \juliainline{compute_initial_nontrivial_flex}$(g,\Vector{p})$\quad\quad\quad\quad\quad\quad\quad\quad\quad\quad\quad~\juliainline{//} Section \ref{section:blocked-flexes}\;
	\juliainline{curve_lengths} $\,\leftarrow\, $\juliainline{[]}\;
	\For{$k=1,\,\dots,\,N$}{
		\uIf{\normalfont\juliainline{is_singularity}$(g,\Vector{x}_{k-1})$}{
			$\Vector{x}_{k}\,\leftarrow\,$\juliainline{resolve_singularity}$(g,\Vector{x}_{k-1},\Vector{v}_{k-1})$\quad\quad\quad\quad\quad\quad\quad\quad\quad\quad~~\juliainline{//} Section \ref{section:singularities}\;
		} \Else {
			$\Vector{y} \,\leftarrow\, R_{\Vector{x}_{k-1}}(\alpha \cdot\Vector{v}_{k-1})$\quad\quad\quad\quad\quad\quad\quad\quad\quad\quad\quad\quad\quad\quad\quad\quad\quad~ \juliainline{//} Sections \ref{section:riemannian-geometry-numerical-AG} and \ref{section:gauss-newton-method}\;
			\juliainline{push!}\juliainline{(}\juliainline{curve_lengths, compute_length_of_curve}$(\Vector{x}_{k-1},\Vector{y},R)$\juliainline{)}~\quad \juliainline{//} Section \ref{section:vector-transport}\;
			\uIf {\normalfont\juliainline{L[-1]}$=0$}{
				$\Vector{x}_{k}\,\leftarrow\,$\juliainline{resolve_singularity}$(g,\Vector{x}_{k-1},\Vector{v}_{k-1})$\quad\quad\quad\quad\quad\quad\quad\quad\quad\,\juliainline{//} Section \ref{section:singularities}\;
			} \Else {
				$\alpha' \,\leftarrow \, \alpha \,\cdot$ \juliainline{L[1]/L[-1]}\quad\quad\quad\quad\quad\quad\quad\quad\quad\quad\quad\quad\quad\quad\quad\quad\quad\quad~~~\juliainline{//} Section \ref{section:vector-transport}\;
				$\Vector{x}_{k} \, \leftarrow \, R_{\Vector{x}_{k-1}}(\alpha' \cdot\Vector{v}_{k-1})$\quad\quad\quad\quad\quad\quad\quad\quad\quad\quad\quad\quad\quad\quad\quad~\juliainline{//} Sections \ref{section:riemannian-geometry-numerical-AG} and \ref{section:gauss-newton-method}\;
				$\Vector{v}_k \,\leftarrow \, $\juliainline{compute_next_flex}$(g,\Vector{x}_k,\Vector{x}_{k-1},\Vector{v}_{k-1})$\quad\quad\quad\quad~~~\juliainline{//} Sections \ref{section:vector-transport} and \ref{section:blocked-flexes}\;
			}
		}
	}
	\Return{\Vector{x}}\;
	\caption{Approximation of Continuous Motion}
	\label{alg:approx-deformation-paths}
\end{algorithm}

\subsection{Gauss-Newton Method and Randomization}
\label{section:gauss-newton-method}

Newton's method has been designed for square system of equations whose Jacobian has full rank. Hence, outside of regular points and in overdetermined geometric constraint sytems Newton's method may encounter issues or diverge altogether. In the following, we consider two strategies to address these problems.

\paragraph{Damped Gauss-Newton Method}
Classically, Newton's method is derived using a Taylor approximation of first order. For that, let us consider a polynomial system $g:\mathbb{R}^{dn}\rightarrow \mathbb{R}^m$. Starting from a point $\Vector{x}\in \mathbb{R}^{dn}$, we try to find an \struc{approximate zero} of $g$, i.e., a point $\Vector{z}\in \mathbb{R}^{dn}$ such that $g(\Vector{z})=0$. To find it, we compute the first-order Taylor approximation around $\Vector{x}$ as
\[g(\Vector{x}+\Vector{\Delta x})\,=\, g(\Vector{x}) + \Vector{\nabla} g( \Vector{x})^\top \cdot \Vector{\Delta x}+o(\Vector{\Delta x})\]
for the \struc{Jacobian} 
\[\Vector{\nabla} g( \Vector{x})^\top = \left(
	\nabla g_1( \Vector{x}) \hdots \nabla g_m( \Vector{x})\right)^\top\]
of $g$ at $\Vector{x}$ and $\Vector{\Delta x}\in \mathbb{R}^{dn}$ with sufficiently small norm. We now assume that $\Vector{x}+\Vector{\Delta x}$ is already an approximate zero of $g$, so we set the left term to $0$. The crucial observation is now that contrary to Newton's method $\Vector{\nabla} g( \Vector{x})^\top$ is not necessarily invertible. The \struc{Gauss-Newton method} \cite{gauss-newton-method}  addresses this problem by instead computing
\[\left(\Vector{\nabla} g(\Vector{x})\cdot \Vector{\nabla} g(\Vector{x})^\top \right)\cdot \Vector{\Delta x} \,=\, -\Vector{\nabla} g(\Vector{x})\cdot g(\Vector{x}),\]
which is also known as the linear system's \struc{normal form} \cite[p.\ 126]{numericallinearalgebra}.
It provides an implicit expression for $\Vector{\Delta x}$. If $\Vector{\nabla} g(\Vector{x})$ has at least rank $m$, then $\Vector{\nabla} g(\Vector{x})\cdot \Vector{\nabla} g(\Vector{x})^\top$ is invertible. In this case, we can explicitly write 
\[\Vector{\Delta x} \,=\,-\left(\Vector{\nabla} g(\Vector{x})\cdot \Vector{\nabla} g(\Vector{x})^\top\right)^{-1}\cdot \Vector{\nabla} g(\Vector{x})\cdot g(\Vector{x})\,=\,-\left(\Vector{\nabla} g(\Vector{x})^\top\right)^\dagger\cdot g(\Vector{x})\]
for the Moore-Penrose pseudoinverse $\left(\Vector{\nabla} g(\Vector{x})^\top\right)^\dagger$ \cite{pseudoinverse}. It provides a minimum norm least squares solution of the linear system \cite[Prop.\ 7.2.1]{numericallinearalgebra}. If a linear system $\Vector{\nabla} g(\Vector{x})^\top\cdot\Vector{\Delta x}=-g(\Vector{x})$ has a solution, this approach finds it: The least squares solution $\Vector{\Delta x}$ with 
\[\lvert\lvert \Vector{\nabla} g(\Vector{x})^\top\cdot\Vector{\Delta x}+g(\Vector{x})\lvert\lvert^2\,=0\] 
is equivalent to $\Vector{\nabla} g(\Vector{x})^\top\cdot\Vector{\Delta x}=-g(\Vector{x})$. Note that the pseudoinverse is also defined in the case where $\Vector{\nabla} g(\Vector{x})$ does not have full rank, so we can use this technique in either case. The Gauss-Newton method comes with similar convergence guarantees as the classical Newton's method \cite{gaussnewtonconvergenceanalysis}. 

However, the residuals do not necessarily decrease in every iteration. For that reason we damp the Gauss-Newton step by adding a $dn\times dn$ diagonal matrix $\lambda\cdot I_{dn}$ with $\lambda >0$ in each iteration that is determined via a backtracking line search using Armijo conditions (cf.\ \cite{armijolinesearch}). In summary, when setting $\Vector{x}_0=\Vector{x}$ the damped Gauss-Newton update of the $k$-th point $\Vector{x}_k$ is given by
\[\Vector{x}_{k+1}\,=\, \Vector{x}_{k}-\alpha_k\cdot \left(\Vector{\nabla} g(\Vector{x}_k)^\top\right)^\dagger\cdot g(\Vector{x}_k)\]
for an $\alpha_k$ that satisfies the \struc{Armijo condition} 
\[||g(\Vector{x}_{k+1})||\leq ||g(\Vector{x}_{k})||+\frac{1}{2}\cdot\alpha_k\cdot \big\langle \Vector{\nabla} g(\Vector{x}_k) \cdot\frac{g(\Vector{x}_k)}{||g(\Vector{x}_{k})||} ,~\Vector{\Delta x}_k\big\rangle \]
for the Euclidean norm $\lvert\lvert \bullet \lvert\lvert$. In practice, $\alpha_k$ is obtained by starting from $\alpha_0=1$ and then halving the value of the damping factor until the condition is satisfied. Doing so guarantees that the Gauss-Newton step sufficiently decreases the norm of the system of polynomial equations evaluated at the new point. The damping stabilizes the method when the Jacobian is ill-conditioned and ensures that the method does not diverge. This makes the Gauss-Newton method highly versatile and suitable for the approximation of continuous motions. 

\paragraph{Randomization}
The existence of stresses in geometric constraint system often makes it difficult to accurately approximate continuous motions. Since equilibrium stresses represent dependencies between the polynomial constraints, they hinder the convergence to solutions with arbitrarily low norm as two dependent constraints may be conflicting.

As an example for why stresses may pose issues, let us consider a (second-order) rigid and two flexible bar-joint frameworks. These frameworks have in common that they all have one nontrivial infinitesimal flex and one equilibrium stress; despite that, their behavior is quite different. The 3-prism in Figure \ref{fig:prestress-stable-or-no}(l.) is generically rigid, yet flexible. In Figure \ref{fig:prestress-stable-or-no}(c.), a vertex has been added to an edge of the otherwise infinitesimally rigid complete graph $K_4$. The resulting framework is second-order rigid. Alternatively, we can add a vertex and an edge to $K_4$, which renders it flexible (cf. Figure \ref{fig:prestress-stable-or-no}(r.)). Still, it has an overbraced component, producing a stress.

\begin{figure}[h!]
	\centering
	\raisebox{1.75em}{
		\includegraphics[width=0.315\textwidth]{Images/quadrilateral_flexible.png}
	}
	\hspace*{3mm}
	\includegraphics[width=0.245\textwidth]{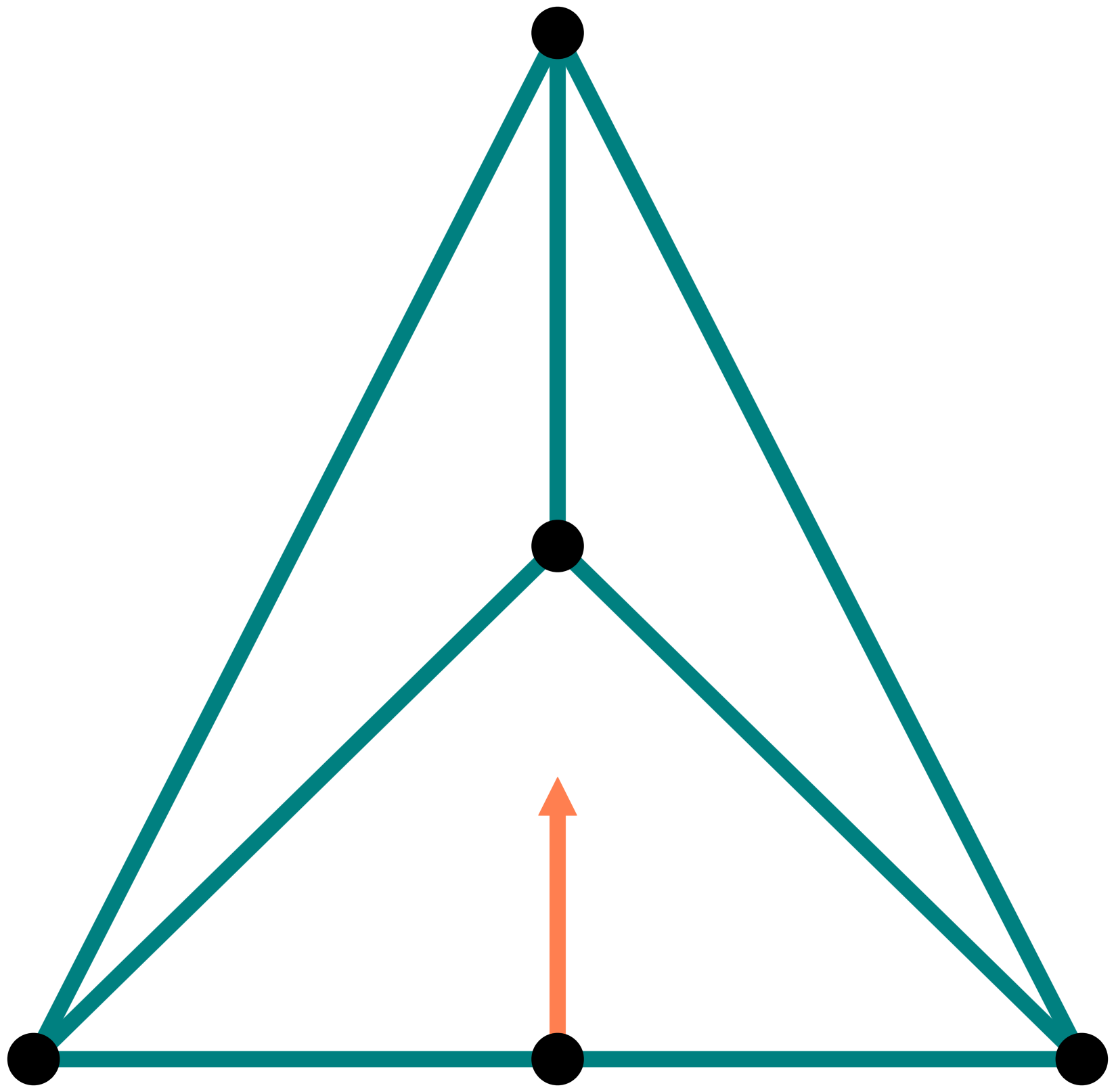}
	\hspace*{7mm}
	\raisebox{1.6em}{
		\includegraphics[width=0.305\textwidth]{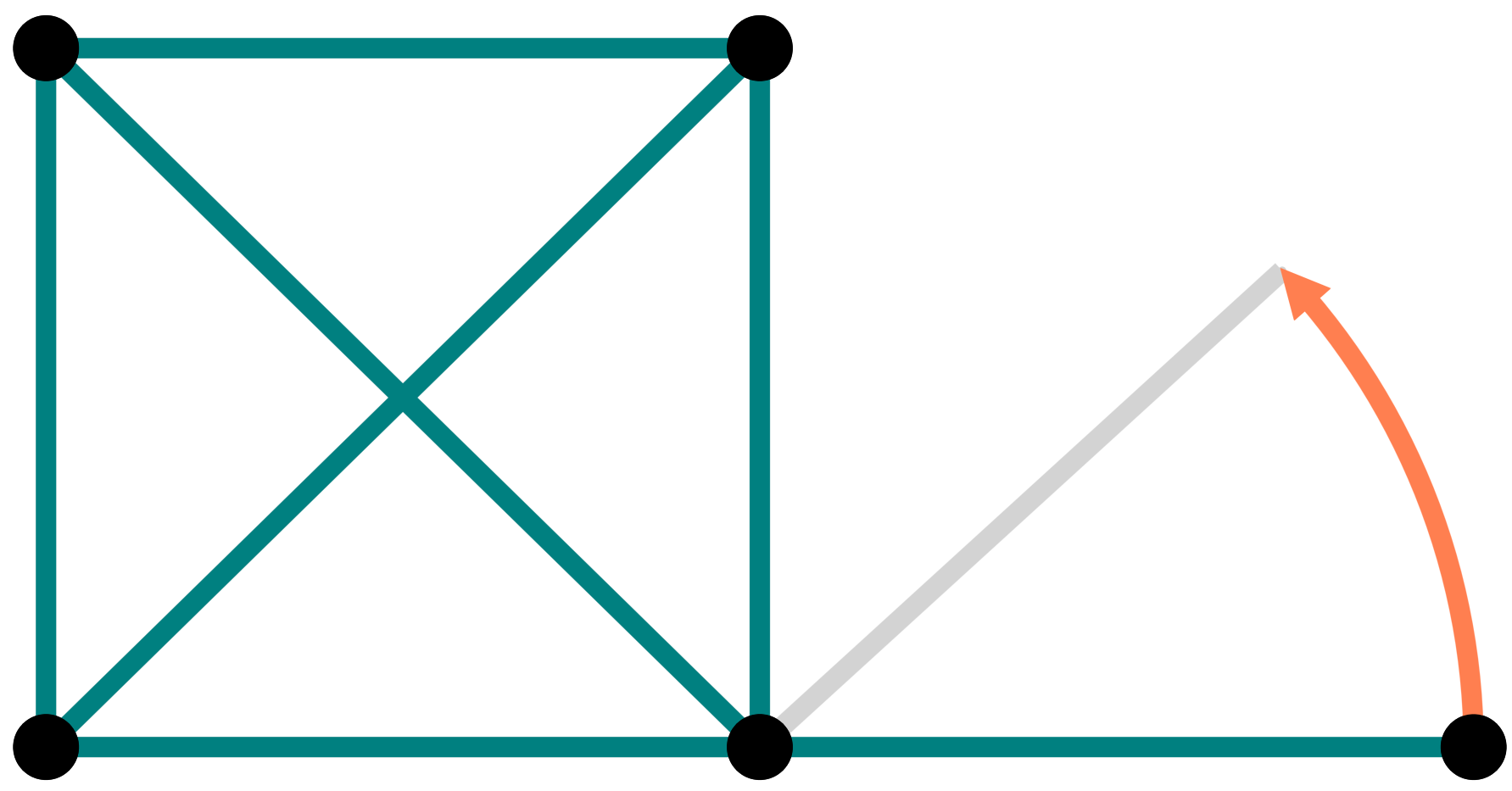}
	}
	\caption{Three bar-joint frameworks with the same number of nontrivial infinitesimal flexes and equilibrium stresses. Two of them are flexible (left and right) and a corresponding motion has been hinted at in gray.}
	\label{fig:prestress-stable-or-no}
\end{figure}

While the local algebraic information from all of these examples is identical, they require careful and individual treatment to produce continuous motions: whereas any small perturbation of the flexible 3-prism in Figure \ref{fig:prestress-stable-or-no}(l.) and the rigid framework in Figure \ref{fig:prestress-stable-or-no}(c.) produces infinitesimally rigid realizations, the flexible framework portrayed in Figure \ref{fig:prestress-stable-or-no}(r.) is flexible, even under arbitrary perturbations. The central question in this context is: How can we remove the redundancies given by the equilibrium stresses so that Newton's method reliably converges again, while preserving the local geometry?

In numerical algebraic geometry, such algebraic dependencies are addressed using a technique called \struc{randomization} (cf. \cite[§13.5]{sommesewampler}), which consists of multiplying the corresponding polynomial system $g(\Vector{x})=\Vector{0}$ with a matrix $\Lambda\in \mathbb{C}^{(m-k)\times m}$ filled with random entries for the number of equilibrium stresses $k$, which is necessarily smaller than the number of equations $m$. Subsequently, we study the solutions of $\Lambda\cdot g(\Vector{x})=\Vector{0}$. The properties of this new system of equations are captured by the following theorem:

\begin{theorem}[cf.\ {\cite[Thm.\ 13.5.1]{sommesewampler}}]
	\label{thm:poly-sys-randomization}
	Let $g(\Vector{x})=\Vector{0}$ be a system of $m$ polynomials on $\mathbb{C}^N$. Assume that $\mathcal{X}\subseteq \mathbb{C}^N$ is an irreducible affine algebraic set. Then, there exists a Zariski open set $U$ of $(m-k)\times m$ matrices $\Lambda \in \mathbb{C}^{(m-k)\times m}$ for $0\leq k<m$ such that it holds for $\Lambda \in U$:
	\begin{enumerate}
		\item if $\dim \mathcal{X} > N-(m-k)$, then $\mathcal{X}$ is an irreducible component of the algebraic set $\mathcal{V}(g)$ if and only if it is an irreducible component of the algebraic set $\mathcal{V}(\Lambda \cdot g)$;
		\item if $\dim \mathcal{X} = N-(m-k)$, then $\mathcal{X}$ is an irreducible component of the algebraic set $\mathcal{V}(g)$ implies that $\mathcal{X}$ is also an irreducible component of the algebraic set $\mathcal{V}(\Lambda \cdot g)$;
		\item if $\mathcal{X}$ is an irreducible component of the algebraic set $\mathcal{V}(g)$, its multiplicity as a solution component of $\Lambda \cdot g(\Vector{x})=0$ is greater than or equal to its multiplicity as a solution component of $g(\Vector{x})=0$. Equality holds if either multiplicity is $1$.
	\end{enumerate}
\end{theorem}
 
In other words, when multiplying the system of equations $g$ with a matrix filled with random complex entries, then the irreducible components of the algebraic set $\mathcal{V}(g)$ do not significantly change. Assuming that $\dim \mathcal{X}<N-(m-k)$, then $\mathcal{X}$ is a proper subvariety of an irreducible component of $\mathcal{V}(g)$, since every (real) irreducible component of $\mathcal{V}(g)$ has at least dimension $N-(m-k)$.

Nevertheless, we are only interested in real solutions, so one may wonder whether this result holds for polynomial system with real coefficients defined over $\mathbb{R}^N$, as well. In fact, $U$ is Zariski open, implying that it is the complement $U=\mathbb{C}^N\setminus V$ of a proper algebraic set $V\subseteq \mathbb{C}^N$. As a proper algebraic set, $V$'s codimension is at least 1. Restricting $V$ to $\mathbb{R}^N$ still produces a (real) proper algebraic subset of codimension at least 1, since proper algebraic hypersurfaces have (at most) dimension $N-1$ (cf. \cite[Prop.\ 1.13]{hartshorne}). Therefore, we can replace the complex numbers $\mathbb{C}$ in Theorem \ref{thm:poly-sys-randomization} by the real numbers $\mathbb{R}$ to obtain a result that is relevant in our setting. In particular, the multiplication by a suitable random matrix removes the geometric constraint system's equilibrium stresses:

\begin{proposition}
	\label{prop:stress-free-randomization}
	Assume that the geometric constraints system $(V,g,p)$ with $m$ polynomial constraints has a $k$-dimensional space of equilibrium stresses at the realization $p$. Then there exists a Zariski open subset of $(m-k)\times m$ matrices $\Lambda$ so that the polynomial system $\Lambda\cdot g(\Vector{x})=0$ has the same dimension as $g(\Vector{x})=0$ and no equilibrium stresses at $p$.
	\begin{proof}
		Theorem \ref{thm:poly-sys-randomization} guarantees that any irreducible component of $\mathcal{V}(g)$ is also an irreducible component of $\mathcal{V}(\Lambda \cdot g)$. First, we realize that $p$ is also a realization of the modified geometric constraint system $\mathcal{F}^\sharp=(V,\Lambda\cdot g,p)$. The rigidity matrix $R_{\mathcal{F}^\sharp}(p)$ is related to the rigidity matrix $R_{\mathcal{F}}(p)$ via
		\[R_{\mathcal{F}^\sharp}(p) = \Lambda \cdot R_{\mathcal{F}}(p).\]
		Therefore, the infinitesimal flexes of $\mathcal{F}$ are a subset of the infinitesimal flexes of $\mathcal{F}^\sharp$.
		
		Take any equilibrium stress $\omega^\sharp$ of $\mathcal{F}^\sharp$ at $p$. In other words, \[R_{\mathcal{F}^\sharp}(p)^\top \cdot \omega^\sharp \,=\, R_{\mathcal{F}}(p)^\top \cdot \Lambda^\top \cdot \omega^\sharp \,=\, 0.\]
		Consequently, $\Lambda^\top \cdot \omega^\sharp$ is an equilibrium stress of $\mathcal{F}$ at $p$. Since $\Lambda$ can be chosen from a Zariski open set of matrices, we can vary it, implying that $\Lambda^\top \omega^\sharp$ is an equilibrium stress of $\mathcal{F}$ for any choice of $\Lambda$. We now show that $\omega^\sharp$ must be trivial. 
		
		Since the image of a linear map is a linear subspace of its codomain, we can assume that $\dim \im \Lambda^\top = m-k$. Otherwise, $\dim \im \Lambda^\top < m-k$, which according to the Rank Nullity Theorem only happens when $\Lambda^\top$ does not have full rank. However, the set of $(m-k)\times m$ matrices with rank at most $(m-k-1)$ is a proper algebraic subset of $\mathbb{R}^{(m-k) \times m}$ called the \struc{determinantal variety}. Therefore, we can choose $\Lambda$ from the Zariski open complement of this set so that it has full rank, so that the restriction of $R_{\mathcal{F}}(p)^\top$ to $\im \Lambda^\top$ is injective onto its image and hence
		\[\mathrm{rank}(R_{\mathcal{F}}(p)^\top\cdot \Lambda^\top)\,=\,\mathrm{rank}(R_{\mathcal{F}}(p)^\top)\,=\, m-k\]
		according to the Rank Nullity Theorem when interpreting $R_{\mathcal{F}}(p)^\top$ as a linear map from $\mathbb{R}^m$ to $\mathbb{R}^N$.
		Another application of the Rank Nullity Theorem shows that 
		\begin{eqnarray*}
			m-k &=& \dim \ker \left( R_{\mathcal{F}}(p)^\top\cdot\Lambda^\top\right) + \underbrace{\mathrm{rank}\left(R_{\mathcal{F}}(p)^\top\cdot\Lambda^\top\right)}_{=m-k}
					\end{eqnarray*}
		when interpreting $R_{\mathcal{F}}(p)^\top\cdot\Lambda^\top$ as a linear map from $\mathbb{R}^{m-k}$ to $\mathbb{R}^N$. In other words, the dimension of the solution space for the linear system $R_{\mathcal{F}}(p)^\top \cdot \Lambda^\top\cdot \omega^\sharp=0$ is $0$. As a homogeneous linear system, the solution space is a vector space, implying that $\omega^\sharp=0$.
	\end{proof}
\end{proposition}

To momentarily remove the stresses from our geometric constraint system in order to ensure the convergence of Newton's method on the polynomial system $g(\Vector{x})=0$ with $m$ equations, this observation suggests the following approach. We first compute the dimension $k$ of the linear space of equilibrium stresses. Subsequently, we generate a $(m-k)\times m$ matrix $\Lambda$ with random floating point entries, e.g. by drawing them from the Gaussian  $\mathcal{N}(0,1)$. This ensures that we lie outside of the corresponding discriminant according to Theorem \ref{thm:poly-sys-randomization}. We then randomize $g$ by multiplying it with $\Lambda$ to obtain a stress-free system $\Lambda \cdot g(\Vector{x})=0$ with high probability (cf. Proposition \ref{prop:stress-free-randomization}). On this system, we compute the Newton step.

\subsection{Unit-Speed Parametrization and Vector Transport}
\label{section:vector-transport}

Retractions have their origins in Riemannian optimization, where they are used to approximate geodesics. A significant disadvantage of retractions with respect to geodesics is that geodesics are locally length-minimizing and have no acceleration, implying that their speed is constant. This makes them particularly suitable for physical contexts, in which energy-minimization plays a role. We mimic this behavior by adaptively choosing the step size so that the length of subsequent curve segments stays approximately constant. 

Another disadvantage of iteratively computing retractions is that subsequent curve segments do not necessarily glue differentiably. This is major problem in geometric constraint systems, where non-differentiable behavior would produce non-physical continuous motions by the law of inertia. We solve this by employing a vector transport along the retraction curve.

\paragraph{Unit-Speed Parametrization} Geodesics have the property that they can be parametrized with constant speed. The homotopy continuation methods that we employ to find the Euclidean distance retraction work on fine discretizations of the Euclidean distance retraction curve between two points. Hence, we can approximate the curve's length by summing over the lengths of the individual line segments. In order for the continuous motion to be smooth and not to change speed rapidly, all predictor-corrector steps should have roughly the same length. We accomplish this by using an adaptive step size control. Instead of fixing a global step size $\alpha>0$, we first record the length of the initial Euclidean distance retraction $L_0>0$ starting from a point $\Vector{x}_0 = \Vector{p}\in \mathbb{R}^N$ with step size $\alpha_0=\alpha>0$ and tangent direction $\Vector{v}_0=\Vector{v}\in T_{\Vector{x}_0}g^{-1}(0)$. In subsequent steps $k>0$, we first look ahead by computing the Euclidean distance retraction curve $c(t)=R_{\Vector{x}_k}(t\cdot\alpha_{k-1}\cdot \Vector{v}_k)$ for $t\in [0,1]$ with the previous step size $\alpha_{k-1}$ using homotopy continuation. Subsequently, we approximate the length of the curve $c$ using the previously described approach and denote it by $L_k$. We correct the step size so that the length of the curve is roughly equal to $L_0$ by setting \[\alpha_k=\frac{L_0}{L_k}\cdot \alpha_{k-1}\]
and then recomputing the Euclidean distance retraction as $R_{\Vector{x}_k}(\alpha_{k}\cdot \Vector{v}_k)$.

In singularities, this approach does not necessarily work. For that reason, we do not rescale the length of the curve when detecting a singularity. 

\paragraph{Vector Transport}
Assume that our iteration with $k>0$ is currently at the point $\Vector{x}_k$ with previous tangent vector $\Vector{v}_{k-1}\in T_{\Vector{x}_{k-1}}g^{-1}(0)$. Our goal is to choose the subsequent tangent vector $\Vector{v}_k\in T_{\Vector{x}_k}g^{-1}(0)$ so that the subsequent step $\Vector{v}_k$ lies in the linear span of the tangent $c_{k-1}'(1)$ in $t=1$ of the Euclidean distance retraction curve $c_{k-1}(t)=R_{\Vector{x}_{k-1}}(\alpha_{k-1}\cdot \Vector{v}_{k-1})$ with $t\in [0,1]$. However, we do not typically have access to the derivative of $c_{k-1}$ and choosing a sufficiently fine discretization of $c_{k-1}$ with homotopy continuation would be costly. Therefore, we rely on a parallel transport along $c_{k-1}$ \cite[§10.3]{boumal2020intromanifolds}. Along retractions, the parallel transport takes a particularly convenient form in terms of the retractions differential (cf. \cite[Prop.\ 10.62]{boumal2020intromanifolds}).

To obtain a derivative-free version of vector transport, we consider \cite[Prop.\ 10.64]{boumal2020intromanifolds} instead. This result states that the orthogonal projection of $\Vector{v}_{k-1}$ onto the tangent space $T_{\Vector{x}_{k-1}}g^{-1}(0)$ is a \struc{transporter}, which is an approximation of the parallel transport sufficiently close to $\Vector{x}_{k-1}$. Consider the matrices $T_{k-1}$ and $T_{k}$ whose columns consist of orthonormal bases of $T_{\Vector{x}_{k-1}}g^{-1}(0)$ and $T_{\Vector{x}_{k}}g^{-1}(0)$, respectively. First, we express $\Vector{v}_{k-1}$ in terms of the local coordinates by computing $T_{k-1}^\top\cdot \Vector{v}_{k-1}$. Subsequently, we solve the associated \struc{Procrustes problem} (see \cite{procrustes}) to find an orthogonal matrix $\mathcal{O}_{(k-1)\rightarrow k}$ that minimizes the least squares problem $||\mathcal{O}_{(k-1)\rightarrow k}\cdot T_{k-1}^\top-T_k^\top||$. In other words, $\mathcal{O}_{(k-1)\rightarrow k}$ is the orthogonal transformation that best aligns the local frame $T_{k-1}$ with the local frame given by $T_{k}$. We apply this orthogonal transformation to $T_{k-1}^\top\cdot \Vector{v}_{k-1}$ in order to transform the vector $\Vector{v}_{k-1}$ to the local frame given by $T_k$. Finally, we multiply the resulting intrinsic representation to $T_{x_{k}}g^{-1}(0)$ by multiplying it with $T_{k}$. In summary, we obtain
\[\Vector{v}_k\,=\,T_k\cdot \mathcal{O}_{(k-1)\rightarrow k}\cdot T_{k-1}^\top\cdot \Vector{v}_{k-1}.\]
We can use the singular value decomposition to compute the orthogonal transformation $\mathcal{O}_{(k-1)\rightarrow k}$. This construction explicitly describes a vector transport of $\Vector{v}_{k-1}$ from $T_{x_{k-1}}g^{-1}(0)$ to $T_{x_{k}}g^{-1}(0)$. If the step size is sufficiently short, then $\Vector{v}_k$ is a suitable approximation for the derivative of the parametrized Euclidean distance retraction curve at $\Vector{x}_k$.

\subsection{Treatment of Singularities}
\label{section:singularities}
Singularities pose a significant challenge for continuous optimization algorithms. Retractions are no exception and in Riemannian optimization, a central assumption is that the underlying constraint set is a smooth manifold. Nevertheless, algebraic sets $\mathcal{X}$ have the property that the singularities form a proper algebraic subset of codimension at least 1 (cf. \cite[Thm. 5.3]{hartshorne}). Therefore, the \struc{singular locus} $\mathrm{Sing}(\mathcal{X})$ is a set of measure $0$ in $\mathcal{X}$ and generic points in $\mathcal{X}$ are non-singular or \struc{regular}. Since the complex numbers $\mathbb{C}$ can be interpreted as a 2-dimensional $\mathbb{R}$ vector space, over $\mathbb{C}$ the singular locus therefore has \emph{real} codimension at least two, meaning that the path-connected components of $\mathcal{X}$ are path-connected in $\mathcal{X}\setminus \mathrm{Sing}(\mathcal{X})$, as well, and we can simply avoid the singularities.

However, over $\mathbb{R}$ it may be unavoidable that the Euclidean distance retraction gets stuck in a singularity. For illustration purposes, we consider the cuspidal and nodal cubic, which are depicted in Figure \ref{fig:cusp-and-node}.

\begin{figure}[h!]
	\centering
	\includegraphics[width=0.65\linewidth]{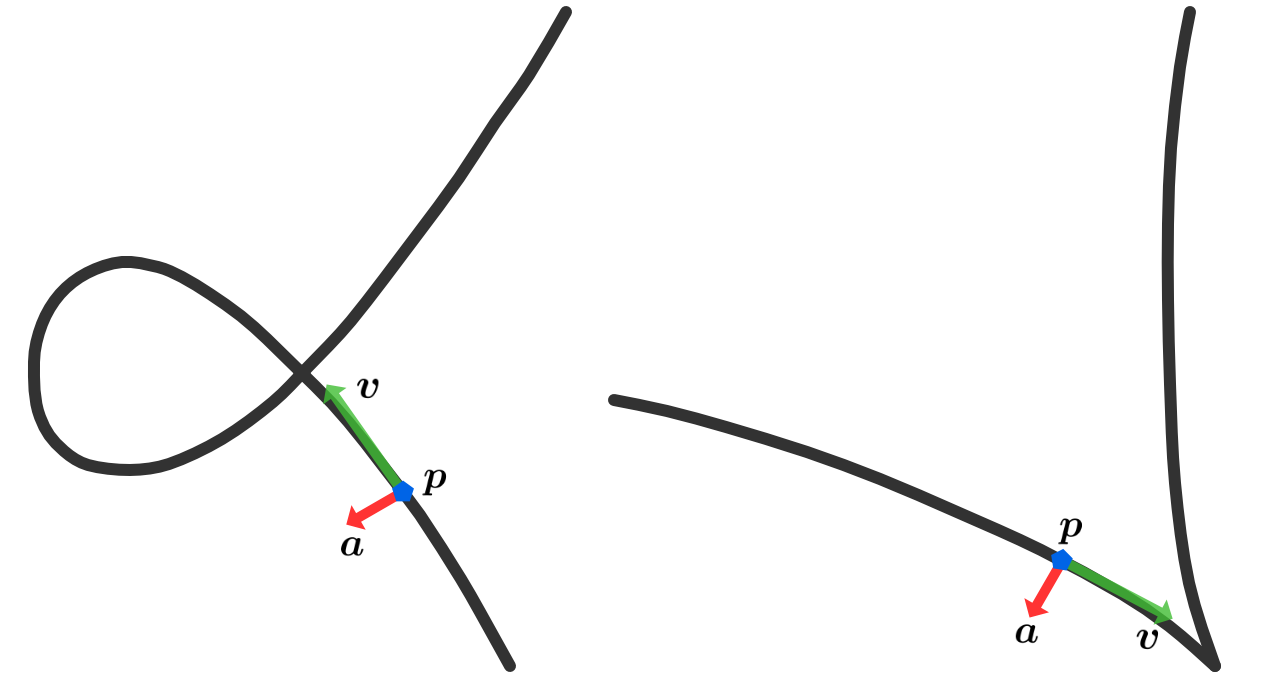}
	\caption{Depictions of the \textbf{(left)} nodal cubic and \textbf{(right)} cuspidal cubic. We marked a tangent vector $\Vector{v}$ (green) and the corresponding acceleration $\Vector{a}$ (red) in a regular point $\Vector{p}$ on the curves.}
	\label{fig:cusp-and-node}
\end{figure}
In order to traverse a singularity, we take inspiration in Gortler, Holmes-Cerfon and Theran \cite{almostrigidity}. In this article, the authors compute the \struc{acceleration} $\Vector{a}$ along a curve $t\mapsto c(t)$ parametrizing a continuous motion by solving a certain linear system in the curve's tangent vector $\dot{c}(0)$.

In the context of geometric constraint systems $\mathcal{F}=(V,g,p)$, these tangent vectors arise from (locally) smooth curves $t\mapsto c(t)\in g^{-1}(0)$ through regular points $c(0)=\Vector{p}$ which satisfy the polynomial constraints $g(c(t))=0$. When differentiating the former constraints with respect to $t$ by applying the chain rule and evaluating the result at $t=0$, we obtain
\[\frac{\partial}{\partial t}g(c(t))\big\lvert_{t=0}\,=\, R_{\mathcal{F}}(c(0))\cdot \dot{c}(0) \,=\, R_{\mathcal{F}}(\Vector{p})\cdot \dot{c}(0) \,=\, 0.\]
Therefore, the possible tangents $\dot{c}(0)\in T_{\Vector{p}}g^{-1}(0)$ are exactly given by infinitesimal flexes $\Vector{v}$ of $\mathcal{F}$ at $\Vector{p}$, providing a different interpretation of infinitesimal flexes.

Before computing the second derivative of $g(c(t))$ in $t=0$, let us recall that $g$ is a system of quadratic polynomials, implying that its Jacobian $R_{\mathcal{F}}(\Vector{p})$ -- which we refer to as the rigidity matrix here -- solely consists of linear forms. Consequently, the second derivative has the form
\[\frac{\partial^2}{\partial t^2}g(c(t))\big\lvert_{t=0}\, = \, R_\mathcal{F}(c(0))\cdot \ddot{c}(0)+R_\mathcal{F}(\dot{c}(0))\cdot \dot{c}(0)\, = \, 0.\]
If we now set $\Vector{v}=\dot{c}(0)$ and $\Vector{a}=\ddot{c}(0)$, then we can rewrite this equation as the linear system
\begin{eqnarray}
	\label{eq:acceleration-lin-sys}
	R_{\mathcal{F}}(\Vector{p})\cdot \Vector{a} \, =\, -R_{\mathcal{F}}(\Vector{v})\cdot \Vector{v}
\end{eqnarray}
for computing an acceleration $\Vector{a}$. Still, there is some level of ambiguity in this linear system, as it is underdetermined: While the right side of the equation is completely determined by the tangent $\Vector{v}$, the rigidity matrix $R_{\mathcal{F}}(\Vector{p})$ does not have full rank, since we assumed that $\mathcal{F}$ is flexible.

If no solution $(\Vector{v},\Vector{a})$ this linear system with a nontrivial flex $\Vector{v}$ exists, then $\mathcal{F}$ is second-order rigid (cf. Definition \ref{def:stress-and-second-order-rigidity}), relating the acceleration to second-order flexes. Assuming that $\mathcal{F}$ is flexible in $\Vector{p}$, our approach of following retraction curves in order to compute has the additional benefit that $\Vector{a}\in N_{\Vector{p}}g^{-1}(0)$ for these types of curves (cf. Definition \ref{def:retraction}). This lets us resolve the ambiguity in Equation (\ref{eq:acceleration-lin-sys}) from before by additionally requiring that the acceleration $\Vector{a}$ is orthogonal to tangent vectors $\Vector{u}$, i.e. $\Vector{a}^\top\cdot \Vector{u}=0$ for all $\Vector{u}\in T_{\Vector{p}}g^{-1}(0)$.

In singularities, the dimension of the tangent space increases. For these cases, the predicted next tangent vector (cf. Section \ref{section:vector-transport}) may not accuractely describe the curves' behavior. 

Singularities can be detected by continually checking whether the rank of the rigidity matrix drops. If that happens, then we drastically reduce the step size and take several more Euclidean distance retraction steps to ensure that we are sufficiently close to the singularity. If we already computed at least one retraction curve ($k>0$), we consider the previous acceleration $\Vector{a}_{k-1}$ alongside the predicted tangent $\Vector{v}_{k}$. The reason why we consider the previous acceleration is that the linear system \eqref{eq:acceleration-lin-sys} does not necessarily have a solution in singularities. 

\begin{example}
	A cusp singularity can be achieved in the configuration space by using only bar-joint frameworks (cf. \cite{higherorderrigidity}). For that reason, we consider the planar cuspidal cubic depicted in Figure \ref{fig:cusp-and-node}(r.) given by the implicit equation $h(x,y)=y^2-x^3=0$. The Jacobian of this system is given by $\nabla h(x,y)^\top=(-3x^2,2y)$. The cusp singularity is located in the origin, where the $\mathrm{rank}(\nabla h(0,0)^\top)=0$. When following any parametrization of a smooth curve $t\mapsto\gamma(t)$ for $t\in[0,1]$ on $h^{-1}(0)$ so that $\gamma(1)=(0,0)$, the  normalized velocity $\dot{\gamma}(1)=\Vector{v}$ will lie in the affine cone $\mathbb{R}_{\geq 0}\cdot (-1,0)$. From the linear system \eqref{eq:acceleration-lin-sys}, we obtain the equation
	\[0\,=\, (0,0)\cdot \Vector{a} \,=\, \nabla h(0,0)^\top\cdot \Vector{a} \overset{\exists t\geq 0}{=} \nabla h(-t,0)^\top\cdot \begin{pmatrix} -t\\ 0 \end{pmatrix} = (-3t^2,0)\cdot \begin{pmatrix} -t\\ 0 \end{pmatrix}\,=\,3t^3\]
	for $\Vector{a}$ and $t\geq 0$, which only has a solution for $t=0$, in which case $\Vector{v}= (0,0)^\top$.
	
	Conversely, in the regular point $\Vector{p}=(1,1)$ with tangent $\Vector{v}=(-2,-3)^\top$ pointing towards the origin, we find that $(-3,2)\cdot\Vector{a}=42$ by inserting these values into \eqref{eq:acceleration-lin-sys}. Solving this linear system yields the acceleration $\Vector{a}^\top=\frac{1}{13}\cdot (-126,\, 84)$.
\end{example}

By computing the acceleration in a previous, regular point $\Vector{p}_{k-1}$, we ensure that the linear system \eqref{eq:acceleration-lin-sys} has a solution. 
 
To escape the singularity $\Vector{p}_k$, we first try to slightly perturb the linear step $\Vector{p}+\alpha\Vector{v}$ using the acceleration to compute the quadratic predictor $\Vector{p}_k+\alpha\Vector{v}_k+\frac{1}{2}\alpha^2 \Vector{a}_{k-1}$ with step size $\alpha>0$ inspired by Newtonian mechanics and computed by approximating a second-order Taylor expansion of $c(t)$ at the singularity $\Vector{p}_k$. This lets us differentiably traverse tacnodes, nodal singularities (cf. Figure \ref{fig:cusp-and-node}(l.)), triple points and $n$-fold points. 

Still, even among planar singularities there are more types of singularities to consider where this approach fails, namely ramphoids and cusps (cf. Figure \ref{fig:cusp-and-node}(r.)). This observation does not even take into account the types of singularities that can occur in higher dimensions. Regardless, by constructing a heuristic for singularities of the cusp type, we hope to be able to tackle a variety of different singularity types in higher dimensions. 

Assume that Newton's method does not converge to a point that is significantly different from $\Vector{p}_k$ when applied to the quadratic step $\Vector{p}_{k+1}\approx\Vector{p}_k+\alpha\Vector{v}_k+\frac{1}{2}\alpha^2 \Vector{a}_{k-1}$. In particular, this is the case when $\Vector{p}_k$ is a cusp, since the basin of attraction for Newton's metod corresponding to the cusp is locally given by a half space. Instead, we compute the point $\Vector{p}_{k+1}$ by taking a step in the directions $\Vector{p}_{k-1}+\alpha \Vector{a}_{k-1}$ and $\Vector{p}_{k+1}\approx\Vector{p}_{k-1}+\alpha \Vector{a}_{k-1}$. Since $\Vector{a}_{k-1}$ lies in $N_{\Vector{p}_{k-1}}g^{-1}(0)$ by construction, applying Newton's method to both of these linear steps for appropriate step sizes $\alpha >0$ should converge to both, $\Vector{p}_{k-1}$ and produce the subsequent point $\Vector{p}_{k+1}$ which lies on a different branch of $g^{-1}(0)$ than $\Vector{p}_{k-1}$. In total, this approach produces the sequence of points $\Vector{p}_{k-1}$, the singularity $\Vector{p}_k$ and $\Vector{p}_{k+1}$.

\subsection{Removing Blocked Flexes}
\label{section:blocked-flexes}

A major difficulty in computing continuous motions is to find initial infinitesimal flexes that extend to nontrivial continuous motions. Otherwise, if we only follow trivial or blocked flexes in the sense of Section \ref{section:higher-order-rigidity}, we may get stuck in the initial position.

We can compute a basis of the space of nontrivial infinitesimal flexes of a geometric constraint system $\mathcal{F}=(V,g,p)$ by first computing its infinitesimal flexes via the kernel of the rigidity matrix $R_\mathcal{F}(p)$. According to Section \ref{section:geometric-constraint-systems}, its trivial infinitesimal flexes are given by the infinitesimal translations and rotations applied to the realization $p$. We factor out the trivial flexes by first extending the basis to a basis of the space of the infinitesimal flexes and then only picking the basis vectors starting from position $r+1$.

To also remove the blocked flexes, we generate a basis $(\Vector{\dot{p}}_1,\,\dots,\,\Vector{\dot{p}}_r)$ of nontrivial infinitesimal flexes and a basis $(\omega_1,\,\dots,\,\omega_s)$ of equilibrium stresses and parametrize the basis using variables $\lambda_1,\,\dots,\,\lambda_r$ and $\mu_1,\,\dots,\,\mu_s$, respectively. Hence, we can express any nontrivial infinitesimal flex as $\Vector{\dot{p}}=\sum_{i=1}^r\lambda_i\cdot \Vector{\dot{p}}_i$ and any equilibrium stress as $\omega=\sum_{i=1}^s\mu_i\cdot \omega_j$. To find an unblocked flex, we consider the equation
\begin{eqnarray*}
	0\,= \,\omega^\top\cdot R_\mathcal{F}(\Vector{\dot{p}})\cdot\Vector{\dot{p}}&=&\sum_{i=1}^s\mu_i\cdot \omega_i^\top \cdot R_\mathcal{F}\left(\sum_{j=1}^r \lambda_j\cdot \Vector{\dot{p}}_j\right) \cdot \sum_{k=1}^r \lambda_k\cdot \Vector{\dot{p}}_k \\
&=&\sum_{i=1}^{s}\mu_i\cdot\underbrace{\sum_{j=1}^{r}\sum_{k=1}^r\lambda_j\cdot \lambda_k\cdot \left(\omega_i^\top\cdot R_\mathcal{F}\left(\Vector{\dot{p}}_j\right)\cdot\Vector{\dot{p}}_k\right)}_{=Q_i(\lambda_1,\dots,\lambda_r)}.
\end{eqnarray*}
We can move the sum outside of the rigidity matrix $R_\mathcal{F}$, as $g$ only consists of quadratic polynomials by assumption, so $R_\mathcal{F}$ only consists of linear forms. Here, $Q_i$ are quadratic polynomials in terms of the variables $\lambda_1,\dots,\lambda_r$. If we differentiate the this expression with respect to the variables $\mu_1,\,\dots,\mu_s$, we obtain the homogeneous, square polynomial system $Q_i(\Vector{\lambda})=0$ for $i=1,\dots,r$. 
 For the solution of a homogeneous polynomial, any scalar multiple is also a solution, so we normalize the set of solutions by adding the polynomial $\sum_{i=1}^r{\lambda_i}^2=1$. Doing so ensures that we find a nonzero solution. In total, we obtain the polynomial system
 \begin{align}
 	\label{eq:square-quadratic-system}
 	Q_1(\Vector{\lambda})\,=\,\dots\,=\,Q_s(\Vector{\lambda})\,=\,1-\sum_{i=1}^r{\lambda_i}^2\,=\,0
 \end{align}
 for which we collect the polynomials $Q_i$ in the vector $\Vector{Q}=(Q_1,\dots,Q_s)$. If the system of equations $(\ref{eq:square-quadratic-system})$ has a real solution $\lambda_1^*,\,\dots,\,\lambda_r^*$, then the corresponding flex $\sum_{i=1}^r\lambda_i^*\cdot \Vector{\dot{p}}_i$ cannot be blocked in the sense of Definition \ref{def:stress-and-second-order-rigidity}, so it has a chance of extending to a continuous motion. Conversely, any solution of this polynomial system can be blocked by an equilibrium stress, when choosing the parameters $\mu_i$ appropriately. Consequently, if the system of polynomial equations \eqref{eq:square-quadratic-system} has no solution, then the geometric constraint system $\mathcal{F}$ is second-order rigid. 
 
 Contrary to that observation, we are mainly interested in non-blocked flexes. For finding one, we need to solve the quadratic polynomial system \eqref{eq:square-quadratic-system}. There are several subtleties that come with this task. First of all, if the system only had isolated (complex) solutions, then we could use an approach based on Gröbner bases or homotopy continuation to find them. However, the dimension of the corresponding algebraic set is a priori unclear, so this does not work well. 
 
 \begin{example}
 	Consider the centrally symmetric bar-joint framework of the 3-prism graph from Figure \ref{fig:inf-rigid-prestress-stable-flexible}(c.) on the set of vertices $V=\{1,2,3,4,5,6\}$ with edges 
 	\[E\,=\,\{(1,2),\,(1,3),\,(1,4),\,(2,3),\,(2,5),\,(3,6),\,(4,5),\,(4,6),\,(5,6)\}\] 
 	given by the realization $p\, :\, V\rightarrow \mathbb{R}^2$ which maps
 	\begin{align*}
 		1\mapsto(1,0),\, 2\mapsto(-1/2,\sqrt{3}/2),\,3\mapsto (1/2,-\sqrt{3}/2),\,4\mapsto (2,0),\, 5\mapsto (-1,\sqrt{3}),\, 6\mapsto (-1,-\sqrt{3}).
 	\end{align*} 
 	We know that this framework is not infinitesimally rigid, as its space of nontrivial infinitesimal flexes contains the vector $\Vector{\dot{q}}=(0,0,0,0,0,0,0,2,-\sqrt{3},-1,\sqrt{3},-1)^\top$. Since it also has the equilibrium stress $\Vector{\omega}=(2,2,6,2,6,6,-1,-1,-1)^\top$, there is a chance that is is second-order rigid. These vectors can be symbolically computed using the Python package \juliainline{PyRigi} \cite{pyrigi}. From this information, we construct the polynomial system \eqref{eq:square-quadratic-system} which has the (monic) form
 	\[\lambda^2=0\quad\text{ and }\quad \lambda^2=1\]
 	for the 3-prism framework, which has no solution. By a previous observation, this implies that the 3-prism framework is second-order rigid.
 \end{example}
 Let us now denote the ideal generated by the polynomials from \eqref{eq:square-quadratic-system} by $\mathcal{I}_{\Vector{Q}}$. In order to find a single real solution, we instead use a trick from real algebraic geometry. For that, we pick a random point $\Vector{p}\in \mathbb{R}^r$ and consider the hypersphere $\sum_{i=1}^r(\lambda_i-p_i)^2=R^2$ centered around $\Vector{p}$ of variable radius $R$. We vary $R$ to find all tangencies of this hypersphere with $\mathcal{V}_{\mathbb{R}}(\mathcal{I}_{\Vector{Q}})$. This geometric construction is depicted in Figure \ref{fig:tikz-hypersphere-tangent}.
 
 \begin{figure}[h!]
\begin{tikzpicture}[scale=1.65]
	
	\def\R{0.666}                  
	\coordinate (P) at (-0.45,0.915); 
	\coordinate (C) at (-0.2,1.53);

	\draw[thick]
	(-1.5,0.3)
	.. controls (-0.5,2) and (0.0,-0.3) ..
	(1,1.15)
	.. controls (1.8,2.2) and (2.3,0.2) ..
	(3,1.4);
	
	\draw[thick] (C) circle (\R);
	
	\draw (C) -- (P) node[midway,right,yshift=-0.6mm,xshift=-0.65mm] {$R$};
	
	\fill (P) circle (1.35pt) node[below,yshift=-0.9mm,xshift=-0.9mm] {$\Vector{x}$};
	\fill (3,1.4) circle (0pt) node[above,xshift=1.35mm] {$\mathcal{V}_{\mathbb{R}}(\mathcal{I}_{\Vector{Q}})$};
	\fill (C) circle (1.35pt) node[above,yshift=0.7mm] {$\Vector{p}$};
\end{tikzpicture}
\vspace*{-10mm}
 \caption{Pictogram for the construction of a point $\Vector{x}$ on the real algebraic set $\mathcal{V}_\mathbb{R}(\mathcal{I}_{\Vector{Q}})$. The circle of radius $R$ centered in $\Vector{p}$ is tangent to $\mathcal{V}_\mathbb{R}(\mathcal{I}_{\Vector{Q}})$ at $\Vector{x}$.}
 \label{fig:tikz-hypersphere-tangent}
\end{figure}
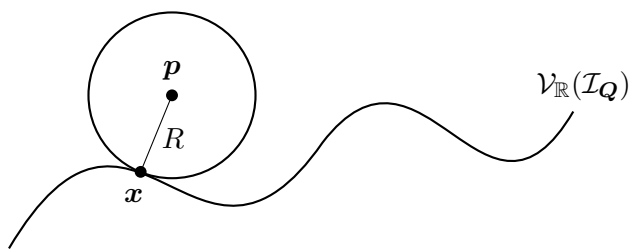

This geometric problem has an unneccessary redundancy: The radius of the hypersphere $R$ does not add information. For this reason, we reformulate the problem of finding a point $\Vector{x}\in \mathcal{V}_\mathbb{R}(\mathcal{I}_{\Vector{Q}})$ in terms of algebraic geometry as the \struc{Euclidean distance degree} problem (cf. \cite{euclideandistancedegree}). For a point $\Vector{p}$, we collect all critical points $\Vector{\lambda}\in \mathcal{V}(\mathcal{I}_{\Vector{Q}})$ of the Euclidean distance $||\Vector{p}-\Vector{\lambda}||$. This approach simplifies the geometric construction from Figure \ref{fig:tikz-hypersphere-tangent} slightly and allows us to exploit existing theoretical results, producing the polynomial ideal 
  \begin{align*}
 	\label{eq:square-quadratic-system-3}
 	\mathcal{I}^{\Vector{p}}_{\Vector{Q}}\,=\,\mathcal{I}_{\Vector{Q}}+\left\langle (c+1)\times (c+1)\text{-minors of }\begin{pmatrix}
 		\Vector{\lambda}-\Vector{p}\\
 		\Vector{\lambda}\\
 		\Vector{\nabla}\Vector{Q}^\top
 	\end{pmatrix}\right\rangle
 \end{align*}
 with $c=\mathrm{codim}(\mathcal{V}(\mathcal{I}_{\Vector{Q}}))$. The dimension of $\mathcal{V}(\mathcal{I}_{\Vector{Q}})$ can be found by evaluating the Jacobian $\Vector{\nabla}\Vector{Q}^\top$ at a generic point in $\mathbb{C}^r$ (cf. \cite[p.\ 239f.]{sommesewampler}). If $c=r$, then the latter ideal simplifies to $\langle 0\rangle$, as in this case the Jacobian $\Vector{\nabla}\Vector{Q}$ has full rank and the algebraic set $\mathcal{V}(\mathcal{I}_{\Vector{Q}})$ only consists of isolated points. 
 
For computing the Euclidean distance degree, we additionally need saturate the ideal $\mathcal{I}_{\Vector{Q}}^{\Vector{p}}$ with the singular points of $\mathcal{V}(\mathcal{I}_{\Vector{Q}})$. We saturate $\mathcal{I}^{\Vector{p}}_{\Vector{Q}}$ with the singular ideal $\mathcal{I}_{\text{sing}}$ to remove the contributions of the singular locus, which may otherwise contribute positive-dimensional solution sets. The singular points of $\mathcal{I}_{\Vector{Q}}$ are described by the ideal
\[
\mathcal{I}_{\text{sing}}\,=\,\mathcal{I}_{\Vector{Q}}+\left\langle c\times c\text{-minors of }
	\begin{pmatrix}
		\Vector{\lambda}\\
		\Vector{\nabla}\Vector{Q}^\top
	\end{pmatrix}
\right\rangle .
\]
This saturation produces $\mathcal{I}^{\Vector{p},\infty}_{\Vector{Q}}=~\mathcal{I}^{\Vector{p}}_{\Vector{Q}}\,:\,{\mathcal{I}_{\text{sing}}}^\infty$. For a generic point $\Vector{p}\in \mathbb{R}^r$, the algebraic set $\mathcal{V}(\mathcal{I}^{\Vector{p}, \infty}_{\Vector{Q}})$ contains only isolated points \cite{euclideandistancedegree}. If $\mathcal{V}_\mathbb{R}(\mathcal{I}_{\Vector{Q}})\neq \emptyset$ and $\Vector{p}\in \mathbb{R}^r$ is generic, then $\mathcal{V}_\mathbb{R}(\mathcal{I}^{\Vector{p}, \infty}_{\Vector{Q}})$ contains at least one point for all real connected components of $\mathcal{V}_\mathbb{R}(\mathcal{I}_{\Vector{Q}})$. Therefore, the algebraic set $\mathcal{V}(\mathcal{I}^{\Vector{p},\infty}_{\Vector{Q}})$ consists solely of complex points if and only if $\mathcal{V}_\mathbb{R}(\mathcal{I}_{\Vector{Q}})$ is empty. 
 
 Consequently, we use homotopy continuation to find \emph{all} isolated complex solutions in $\mathcal{V}(	\mathcal{I}^{\Vector{p}}_{\Vector{Q}})$ (cf. \cite{sommesewampler}). Note that we are again considering the ideal $\mathcal{I}^{\Vector{p}}_{\Vector{Q}}$ instead of the saturated ideal $\mathcal{I}^{\Vector{p},\infty}_{\Vector{Q}}$, since homotopy continuation only finds isolated solutions and the singularities are viable solutions of the system of equations corresponding to $\mathcal{I}_{\Vector{Q}}^{\Vector{p}}$. Using interval arithmetic, we can then distinguish and certify whether a given solution is real or complex (cf. \cite{realeuclideandistancedegree}). By filtering out all non-real isolated solutions from $\mathcal{V}(	\mathcal{I}^{\Vector{p}}_{\Vector{Q}})$, we can produce a real solution of $\mathcal{V}_\mathbb{R}(\mathcal{I}_{\Vector{Q}})$, provided that this algebraic set is non-empty, or alternatively certify the emptiness of $\mathcal{V}_\mathbb{R}(\mathcal{I}_{\Vector{Q}})$.
 
 To recap, we first find the nontrivial infinitesimal flexes of $\mathcal{F}$ by factoring out the infinitesimal translations and rotations from $\mathcal{F}$'s infinitesimal flexes. From these flexes, we construct a non-blocked flex -- if existent -- by computing a point $(\lambda_1^*,\,\dots,\,\lambda_r^*)$ on the zero-dimensional algebraic set $\mathcal{V}_{\mathbb{R}}(	\mathcal{I}^{\Vector{p}}_{\Vector{Q}})$ using the Euclidean distance degree in combination with homotopy continuation, which produces a real, infinitesimal flex $\sum_{i=1}^r\lambda_i^*\cdot \Vector{\dot{p}_i}$.

\subsection{The Software Package DeformationPaths.jl}
\label{section:DeformationPaths.jl}

Together with the heuristics and algorithmic techniques discussed in Sections \ref{section:gauss-newton-method}--\ref{section:blocked-flexes}, Algorithm \ref{alg:approx-deformation-paths} is implemented within the Julia package \juliainline{DeformationPaths.jl}. In applications, the term \struc{deformation paths} is often used synonymous with ``continuous motions''. The package comes with various explicit geometric constraint systems that are implemented as separate classes with intuitive instantiation, ranging from bar-joint frameworks and polytopes to sphere packings and simplicial complexes. The package not only lets the user approximate continuous motions for these geometric constraint systems, but also allows checking for the infinitesimal, second-order and continuous rigidity. It comes with a multiple options for transformations, visualization and animations of continuous motions. In the Appendix Section \ref{section:examples-of-gcs}, the usage of \juliainline{DeformationPaths.jl} is explained in far greater detail. In particular, that section introduces the explicit geometric constraint systems that are implemented in the package and demonstrate how to interact with the software.
As an example of how to interact with the package \juliainline{DeformationPaths.jl}, let us consider the 3-prism instantiated using the code

\begin{julia}
F = Framework(
	[(1,2), (1,3), (2,3), (4,5), (5,6), (4,6), (1,4), (2,5), (3,6)],
	[0 0 sqrt(3)/2 1 1 1+sqrt(3)/2; 0 1 0.5 0 1 0.5]; 
	pinned_vertices=[1,4]
)
plot(F, "3-Prism_1"; edge_color=teal, flex_color=coral, plot_flexes=true)

D = DeformationPath(F, [-1], 30; step_size=0.01)
F2 = Framework(F.bars, D.motion_matrices[end]; pinned_vertices=[1,4])
plot(F2, "3-Prism_2"; edge_color=teal, flex_color=coral, plot_flexes=true)

D2 = DeformationPath(F2, [-1], 30; step_size=0.01)
F3 = Framework(F.bars, D2.motion_matrices[end]; pinned_vertices=[1,4]) 
plot(F3, "3-Prism_3"; edge_color=teal, flex_color=coral, plot_flexes=true)
\end{julia}

\noindent This code plots the bar-joint frameworks \juliainline{F}, \juliainline{F2} and \juliainline{F3} obtained from a continuous motion of \juliainline{F} with teal edges and coral infinitesimal flexes. It saves the resulting frame as a \juliainline{png} file called \juliainline{3-Prism_i.png} for the counter $i\in \{1,2,3\}$. The frameworks come with the optional argument \juliainline{pinned_vertices=[1,4]}, which appends linear constraints to the geometric constraint system that fix the coordinates of the vertices $1$ and $4$. The two continuous motions \juliainline{D} and \juliainline{D2} are computed with the fixed step size \juliainline{0.01} and the total number of steps \juliainline{30}. The vector \juliainline{flex_mult=[-1]} indicates the selection of the infinitesimal flex: first, the nontrivial infinitesimal flexes are computed and stored in \juliainline{flex_matrix}; subsequently, the infinitesimal flex for the \juliainline{plot} method is computed as \juliainline{flex_matrix*flex_mult}. The images resulting from these lines of code are depicted in Figure \ref{fig:3-prism-pictures}.

\begin{figure}[h!]
	\centering
	\includegraphics[width=0.328\textwidth]{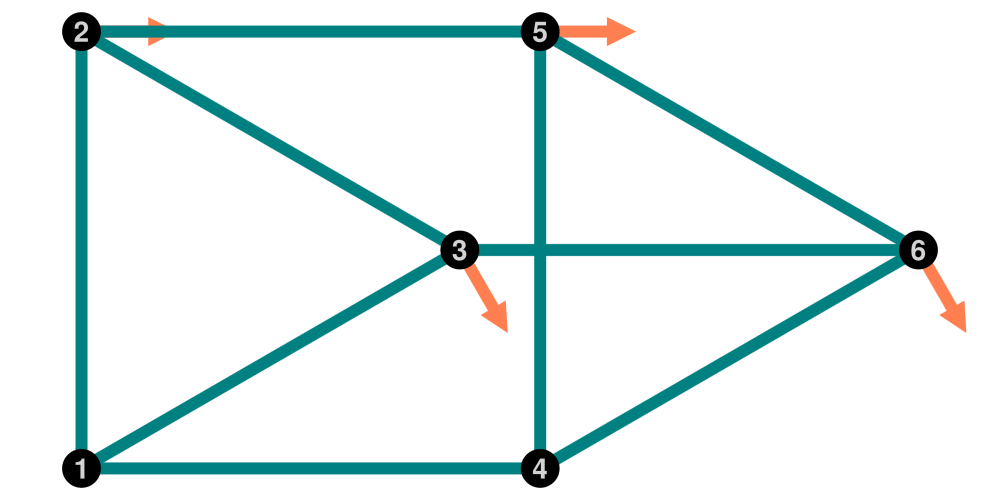}
	\includegraphics[width=0.328\textwidth]{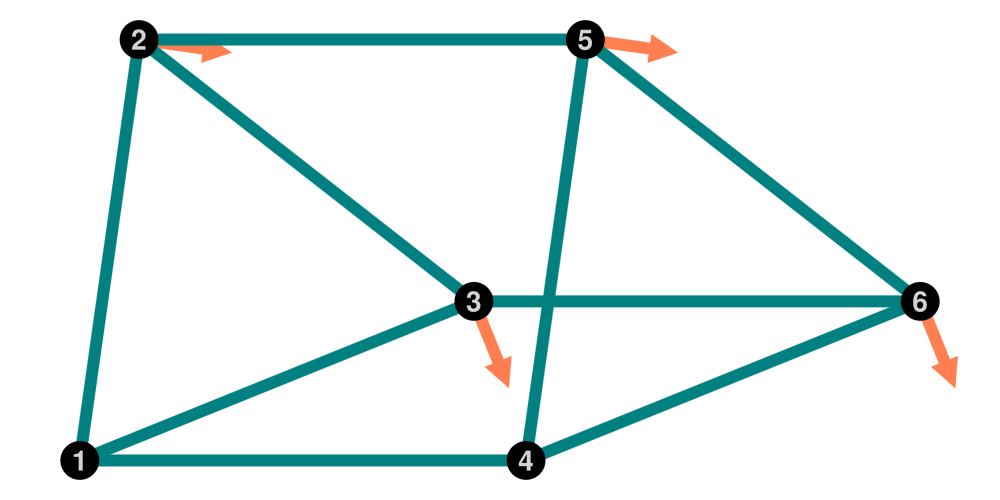}
	\includegraphics[width=0.328\textwidth]{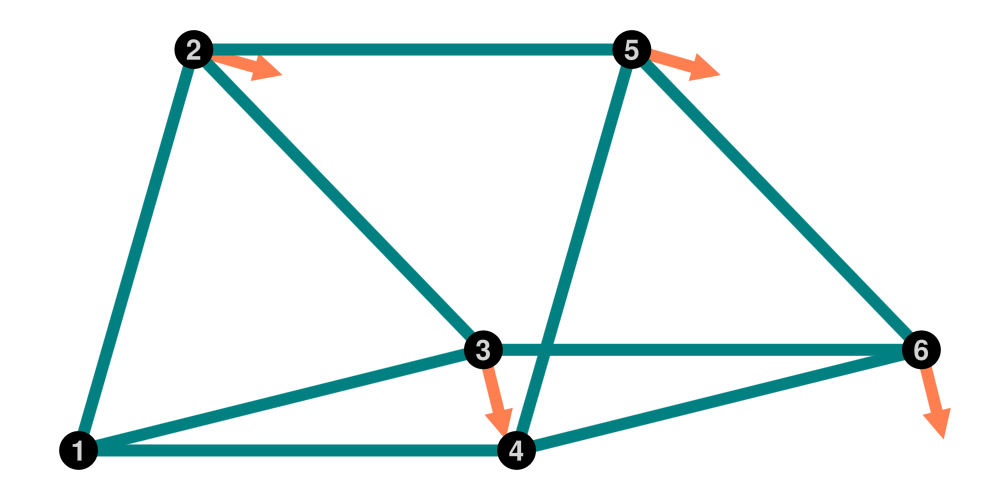}
	\caption{Three frames from a continuous motion of the 3-prism bar-joint framework.}
	\label{fig:3-prism-pictures}
\end{figure}

Beyond still frames, \juliainline{DeformationPaths.jl} has the option to create animations of continuous motions using the generic \juliainline{animate} method, which takes the same keywords as the \juliainline{plot} routine. For more details, you can consult the online documentation\footnote{\url{https://matthiashimmelmann.github.io/DeformationPaths.jl/}} of the package.

\section{Outlook}
\label{section:outlook}
In this article we demonstrate that the combination of Riemannian optimization and numerical algebraic geometry enables us to approximate physically-meaningful continuous motions of a wide range of geometric constraint systems. We show that these techniques even work when the corresponding configuration space contains singularities, in the presence of algebraic dependencies between the polynomial constraints or when not all nontrivial infinitesimal flexes extend to continuous motions. By implementing the corresponding algorithms in the Julia package \juliainline{DeformationPaths.jl}, we illustrate that the corresponding algorithms are effective and can be applied to various geometric constraint systems.

Even though the techniques we develop in this article let us reliably and effectively approximate continuous motions of relatively general geometric constraint systems, the overall goal is quite ambitious. Consequently, several questions remain open and fall outside the scope of this article. In what follows, we collect several open problems that we hope will inspire future research.
\begin{itemize}
	\item Does the second-order rigidity of a geometric constraint system imply its rigidity?
	\item The addition of further and more general geometric constraint systems to the Julia package in the form of classes would increase the package's relevance and the amount of its practical applications. In particular, adding spherical packings with different radii or packings of ellipsoidal particles to the package would make for exciting additions with applications in soft matter physics.
	\item In the same context, more general types of hypergraphs should be implemented. For instance, $k$-dimensional volume-constrained simplicial complexes in $\mathbb{R}^d$ for $k<d$ or $d$-dimensional polyhedral complexes in $\mathbb{R}^d$ are highly interesting and would generalize the possible geometric objects significantly. Still, constructing the geometric constraints is nontrivial in these cases and requires a careful analysis.
	\item The algorithmic environment can be adapted to find (paradoxically) flexible circle packings on the sphere, or, more generally, flexible geometric constraint systems that have more constraints than degrees of freedom. A comprehensive empirical study thus has the potential to reveal novel flesible structures.
\end{itemize}

\section*{Acknowledgments}

I am grateful for the financial support by the National Science Foundation under Grant No.\ DMS-1929284 while I was in residence at the Institute for Computational and Experimental Research in Mathematics in Providence, RI, during the Geometry of Materials, Packings and Rigid Frameworks semester program.

\section*{Data Availability Statement}

The software package \texttt{DeformationPaths.jl} supporting the findings of this article is openly and permanently available in Zenodo \cite{deformationpaths}.

\printbibliography

\appendix
\markboth{APPENDIX}{APPENDIX}
\section{Infinitesimal Rigidity of Geometric Constraint Systems}
\label{appendix:infrig}
	By the assumption that the polynomials in a geometric constraint systems are invariant under Euclidean isometries made in Definition \ref{def:constraint-system-in-Rd}, the kernel of the rigidity matrix $R_\mathcal{F}(p)$ usually contains the trivial infinitesimally flexes via group actions of the Euclidean group applied to the point set $\{p(1),\dots,p(n)\}$. Yet, if there are not sufficiently many vertices, the kernel of the rigidity matrix may be smaller than expected. In these cases, the kernel still contain the Lie algebra of Euclidean transformations given by
	\begin{itemize}
		\item translations $(e_i)_{j\in [n]}$ for $i\in [d]$ and the standard basis $(e_1,\dots, e_d)$ of $\mathbb{R}^d$ and
		\item infinitesimal rotations $(Q\cdot p(j))_{j\in [n]}$ for any $d\times d$ skew-symmetric matrix $Q$.
	\end{itemize}
	
	There are exactly $(\ell_p+1)(2d-\ell_p)/2$ Euclidean transformations for $\ell_p=\dim(\mathrm{aff}(p(1),\dots,p(n)))$ (cf. \cite{Asimow1978TheRO}). These observations are summarized in the following two lemmas. 
	
	\begin{lemma}
		\label{lem:trivial-flexes-lie-in-kernel}
		For any geometric constraint system $\mathcal{F}=(V,g,p)$, the trivial infinitesimal flexes of $\mathcal{F}$ lie in the kernel of $R_\mathcal{F}(p)$.
		\begin{proof}
			By Definition \ref{def:constraint-system-in-Rd}, $g$ is invariant under Euclidean transformations, implying that 
			\[g(\left(F(t)\cdot p(i)+t\cdot \tau\right)_{i\in[n]})=0\] 
			for a smooth curve of rotation matrices $F(t)\in \mathbb{R}^{d\times d}$ with $F(0)=I$, translations $\tau\in \mathbb{R}^d$ and $t\in (-\varepsilon,\varepsilon)$ for some $\varepsilon>0$. Skew-symmetric matrices represent infinitesimal rotations, since the Lie group of orthogonal matrices $\mathcal{O}(d)$ has the skew-symmetric matrices as its Lie algebra \cite[Prop.\ 3.24]{liealgebras}, which represents the tangent space at the identity matrix.
			
			As a consequence, $F'(t)$ is skew-symmetric for all $t$ and for any skew-symmetric matrix there exists a curve $F$ so that $F'(0)$ is that matrix. For the differentiable curve $\alpha(t)=F(t)\cdot p(i)+t\cdot \tau$ we then find that
			\[0~=~\frac{\partial}{\partial t}g(\alpha(t)) \lvert_{t=0} ~=~R_\mathcal{F}(p)\cdot (F'(0)\cdot p(i)+\tau),\]
			implying that $(F'(0)\cdot p(i)+\tau)$ lies in the kernel of the rigidity matrix $R_\mathcal{F}(p)$. This proves the claim. 
		\end{proof}
	\end{lemma}
	
	\begin{example}
		Lemma \ref{lem:trivial-flexes-lie-in-kernel} is not in conflict with the previous observation that the number of trivial infinitesimal flexes depends on the number of vertices $n$. Consider, for instance, a line segment $\overline{p(1)p(2)}$ as a bar-joint framework 
		\[\mathcal{F}=(\{1,2\},\,(||\Vector{p}_1-\Vector{p}_2||-||p(1)-p(2)||),\,i \mapsto p(i))\] 
		in $\mathbb{R}^3$. The group of three translations in $\mathbb{R}^3$ completely lies in the kernel of the corresponding rigidity matrix. As for the rotations, we may without loss of generality fix the position of $\Vector{p}_1$ as the center of rotation. The position of $\Vector{p}_2$ is then constrained to a sphere of radius $||p(2)-p(1)||$. If $p(2)\neq p(1)$, this sphere can be generated by two angles as a 2-dimensional manifold, meaning that the position of $\Vector{p}_2$ has two degrees of freedom. In total, there are 5 elements in the kernel of $R_\mathcal{F}(p)$, which are all the trivial infinitesimal flexes of $\mathcal{F}$.
	\end{example}
	
	The fact that the polynomial map $g$ describing a geometric constraint system's realization space purely consists of polynomials puts this problem in the realm of algebraic geometry, making related techniques applicable. Since the infinitesimal flexes are further given by tangent vectors associated with the constraint set $g^{-1}(0)$, the infinitesimal flexes lie in the tangent space at the geometric constraint system's realization in $\mathbb{R}^{dn}$. Moreover, Definition \ref{def:infinitessimal-rigidity} highlights the necessity for the requirement that $g$ consists only of quadratic polynomials. In general, the cokernel of a polynomial system's Jacobian encodes linear constraints on the image of the polynomial map and can therefore reduce the amount of flexes that can extend to continuous paths. While it is, in priciple, possible to perform a rigidity analysis with any type of constraints, the linearized analysis here heavily relies on the fact that the underlying constraints are quadratic. 
	
	We can make the connection of infinitesimally rigid geometric constraint systems with the rank of their Jacobian even more explicit, as demonstrated in the following generalization of a classical theorem by Asimow-Roth \cite{Asimow1978TheRO}.

	\begin{lemma}
		\label{lem:rank-rigiditiy-matrix}
		The rank of the rigidity matrix associated with a $d$-dimensional geometric constraint system $\mathcal{F}=(V,g,p)$ is bounded above by 
		\[\mathrm{rank}(R_\mathcal{F}(p)) \leq nd-(\ell_p+1)(2d-\ell_p)/2\] 
		for $\ell_p=\dim(\mathrm{aff}(p(1),\dots,p(n)))$. Any infinitesimally rigid realization is a regular point of the constraint set $g^{-1}(0)$.
		Moreover, $\mathcal{F}$ is infinitesimally rigid if and only if the rank of the rigidity matrix is $nd-(\ell_p+1)(2d-\ell_p)/2$.
		\begin{proof}
			By a previous observation, $\mathcal{F}$'s trivial infinitesimal flexes are generated by translations, rotations and reflections. 
			All of these infinitesimal flexes also
			lie in the kernel of the rigidity matrix $R_\mathcal{F}(p)$, since the polynomials in $g$ are invariant under Euclidean isometries by definition. Then, we saturate $g$ with the bar-equations $||p(i)-p(j)||^2-||\Vector{p}_{i}-\Vector{p}_{j}||^2=0$ for all $i,j\in V$. By Asimow-Roth \cite{Asimow1978TheRO}, the kernel of the latter system's Jacobian exactly contains the trivial motions. Since this operation can only increase the rank of the Jacobian and since the Euclidean isometries on the point set $\{p(1),\dots,p(n)\}$ define a smooth manifold of dimension $(\ell_p+1)(2d-\ell_p)/2$, the rank of the rigidity matrix is bounded above by $dn-(\ell_p+1)(2d-\ell_p)/2$ and this number can be attained. 
			
			Now assume that $\mathcal{F}$ is infinitesimally rigid, meaning that the kernel of $R_\mathcal{F}(p)$ exactly consists of trivial infinitesimal flexes. The fact that $g$ purely consists of polynomials implies that there exists an open neighborhood $U(p)$ of $p$ such that $U(p) \,\cap\, g^{-1}(0)$ defines a smooth $\left((\ell_p+1)(2d-\ell_p)/2\right)$-dimensional manifold, since the rank of the Jacobian is maximal and would only drop in a proper Zariski-closed subset of codimension at least 1 (cf. \cite[Thm. 5.3]{hartshorne}).
			
			Since the Euclidean motions on the point set $(p(1),\dots,p(n))$ also define a smooth manifold of dimension $\left((\ell_p+1)(2d-\ell_p)/2\right)$ which is contained in $g^{-1}(0)$ and since $g$ is invariant under Euclidean transformations, this implies that the corank of infinitesimally rigid realizations is necessarily $\left((\ell_p+1)(2d-\ell_p)/2\right)$. 
			
			Conversely, assume that the corank of the rigidity matrix is $\left((\ell_p+1)(2d-\ell_p)/2\right)$. As a consequence, $p$ is a regular point of the constraint set $g^{-1}(0)$ (cf. \cite[§1.5]{hartshorne}). Since the manifold of Euclidean isometries is $\left((\ell_p+1)(2d-\ell_p)/2\right)$-dimensional and is contained in $g^{-1}(0)$, this implies that the kernel of $R_\mathcal{F}(p)$ only consists of trivial infinitesimal flexes, so $\mathcal{F}$ is infinitesimally rigid.
		\end{proof}
	\end{lemma}
	
	This lemma gives us an idea how large the kernel of the rigidity matrix is for infinitesimally rigid systems. Yet, infinitesimal rigidity is only a sufficient condition. For geometric constraint systems, their continuous rigidity (see Definition \ref{def:continuous-motion}) and flexibility are related in the following way:
	\begin{lemma}
		\label{lem:existence-continuous-motion}
		Let $\mathcal{F}=(V,g,p)$ be a geometric constraint system. Then the following are equivalent:
		\begin{enumerate}
			\item $\mathcal{F}$ is not rigid.
			\item There exists a non-trivial continuous motion $\alpha:[0,1]\rightarrow \mathbb{R}^{dn}$ in $g^{-1}(0)$ so that $\alpha(0)=p$, which is also analytic.
		\end{enumerate}
		\begin{proof}
			Assume that $\mathcal{F}$ is flexible. By Definition \ref{def:continuous-motion}, this implies that arbitrarily close to $p$ there are equivalent realizations that are not congruent to $p$. The constraint set $g^{-1}(0)$ defines an algebraic set with $p\in g^{-1}(0)$ that contains the space of trivial continuous motions as a proper subset. The latter set is a algebraic subset of $g^{-1}(0)$, since it can be written as the intersection of $g^{-1}(0)$ with the constraint set corresponding to the bar-joint framework on the complete graph with realization $p$, exactly describing the realizations which are congruent to $p$. Thus, we can apply \cite[Lem.\ 18.3]{algebraicapproximationcurves}, which ensures the existence of an analytic path in $g^{-1}(0)$ through $p$, which clearly also is continuous.
			
			Conversely, assume that there exists a non-trivial continuous motion $\alpha:[0,1]\rightarrow \mathbb{R}^{dn}$ with $\alpha(0)=p$. In other words, all realizations $\alpha(t)$ are equivalent, but not congruent to $p$ for $t>0$. Since $\alpha$ is continuous, this implies that arbitrarily close to $p$ there are equivalent, but not congruent, realizations to $p$, which implies that $\mathcal{F}$ is not rigid.
		\end{proof}
	\end{lemma}
	In other words, for flexible geometric constraint systems there always exist nontrivial continuous motions which can be differentiated arbitrarily often.

\section{Examples of Geometric Constraint Systems}
\label{section:examples-of-gcs}

The theoretical aspects of this article are accompanied by the development of the Julia package \juliainline{DeformationPaths.jl} \cite{deformationpaths}, which allows the user to create a wide range of geometric constraint systems, check their rigidity properties, compute approximate continuous motions and visualize the results.

Notably, the images from this section have been generated using the visualization capabilites of \juliainline{DeformationPaths.jl}. Besides having the capacity to plot configurations for all of the geometric constraint systems that are featured in this section, the package comes with tools to animate continuous motions for them. Some exemplary animations can be found in the online documentation\footnote{\url{https://matthiashimmelmann.github.io/DeformationPaths.jl/}}.

In this section, we highlight several popular classes of geometric constraint systems and how objects from these classes can be instantiated in the Julia package \juliainline{DeformationPaths.jl}. In particular, we are going to discuss, where applicable, how to modify the underlying geometric constraint systems so that their intrinsic trivial infinitesimal flexes are identical to the ones that we describe in Section \ref{section:geometric-constraint-systems}.

\subsection{Framework Materials} 

\,\struc{Framework materials} describe a class of geometric constraint systems whose primary constraints come from embedded graphs with distance constraints on their edges. These structures may come with additional constraints such as restricting them to a surface or prescribing angles between edges.

\paragraph{Bar-Joint Frameworks}
A \struc{bar-joint framework} is a classical geometric constraint system, which is given as a tuple $\mathcal{F}=(G,p)$ consisting of a graph $G=(V,E)$ and a realization $p:V\rightarrow \mathbb{R}^d$. Its polynomial constraint set consists of the edge length equations
\[||\Vector{p}_u-\Vector{p}_v||\,=\,||p(u)-p(v)||\quad\quad\text{ for all }uv\in E.\]
In addition, we can pin vertices by adding linear constraints of the form $\Vector{p}_u = p(u)$. Although this gives to a relatively simple geometric constraint system, the applications of frameworks are diverse. They are used in civil engineering to study the stiffness of structures, in granular materials to understand the jammedness of packings, find applications in quasicrystalline structures and are used in kinematics or in the design of robots for creating prescribed mechanism curves. 

\begin{figure}[h!]
	\begin{minipage}{0.59\linewidth}
		\includegraphics[width=\linewidth]{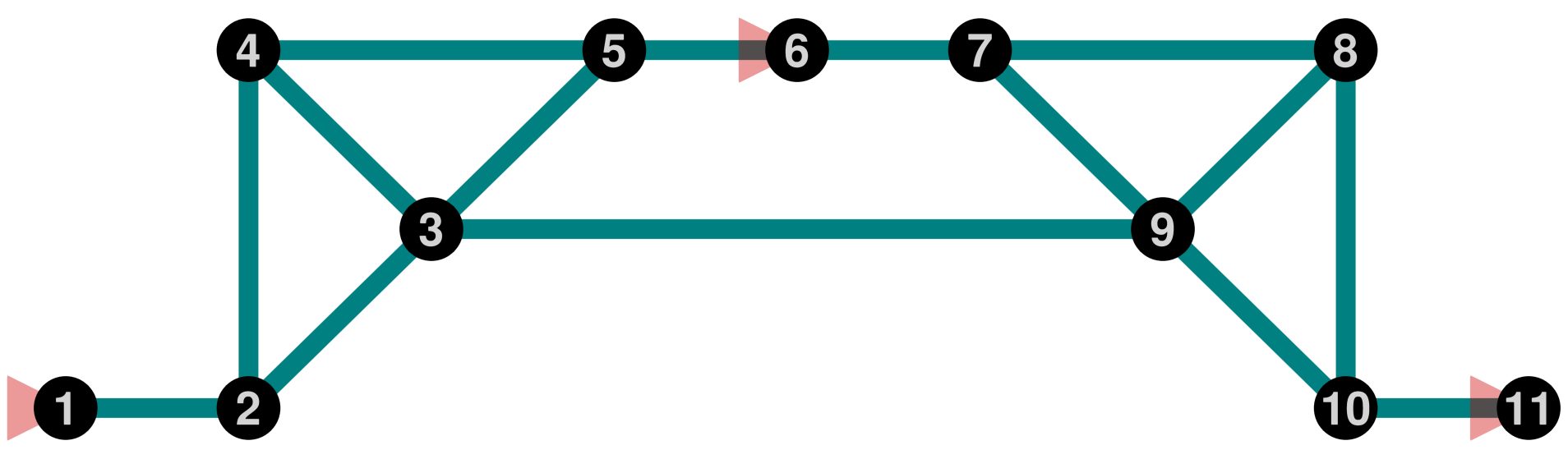}
	\end{minipage}\hfill
	\begin{minipage}{0.355\linewidth}
		\includegraphics[width=\linewidth]{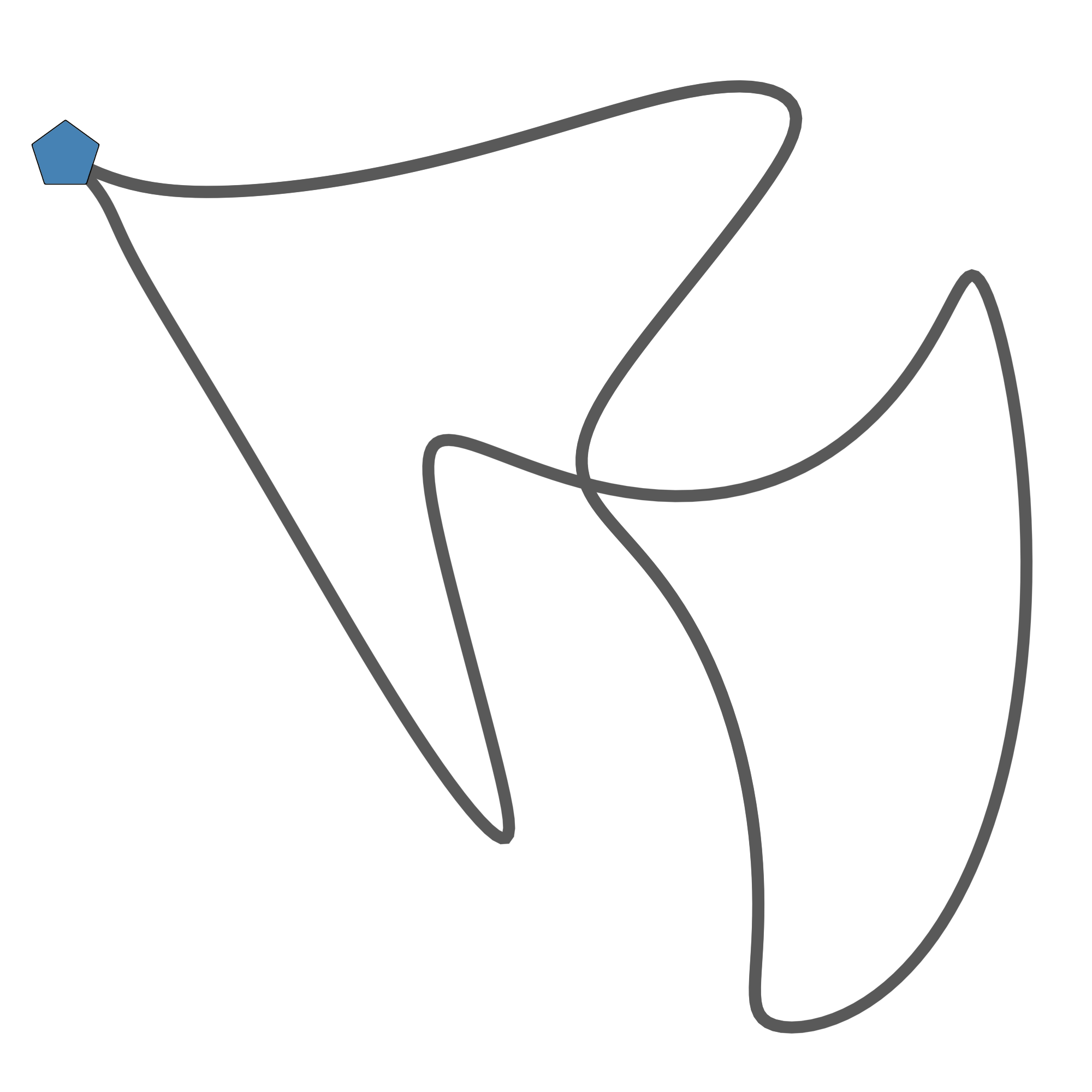}
	\end{minipage}
	
	\caption{\textbf{(left)} A realization of the Double Watt mechanism and \textbf{(right)} a random projection of its deformation space. The deformation space contains a cusp singularity (blue pentagon).}
	\label{fig:double-watt}
\end{figure}

A particularly fascinating framework is given by the Double Watt mechanism, which is depicted in Figure \ref{fig:double-watt}(l.). This framework has the curious property that it is third-order rigid, yet flexible \cite{higherorderrigidity} (a discussion of higher-order rigidity is omitted for the sake of brevity). As a result, the realisation space contains a singularity, which can be seen in Figure \ref{fig:double-watt}(r.). In the package \juliainline{DeformationPaths.jl}, a corresponding object can be instantiated using the following, simplified code:

\begin{julia}
F = Framework([[1,2], [2,3], ..., [9,10], [10,11]], 
	[0 1 2 1 3 4 5 7 6 7 8; 0 0 1 2 2 2 2 2 1 0 0]; pinned_vertices=[1,6,11]
)
D = DeformationPath(F, [], 500; step_size=0.05)
\end{julia}
This creates a continuous motion of \juliainline{F}, which comes with the \juliainline{pinned_vertices=[1,5,6]}, corresponding to the addition of linear constraints to $g$ that fix the position of the vertices 1, 5 and 6. Then, the \juliainline{DeformationPath} function is given an initial step size of \juliainline{0.05} and \juliainline{500} total steps as arguments. This combination of step size and steps is sufficient to obvserve the entire orbit. The images from Figure \ref{fig:double-watt} were created using the \juliainline{plot} and \juliainline{project_deformation_random} routines, respectively.

Notably, we are able to escape and traverse the singularity in \juliainline{F}'s deformation space using the theoretical ideas developed in Section \ref{section:singularities}. By passing an empty list as the second keyword, we also let the algorithm decide and compute a nonblocked initial flex (cf. Section \ref{section:blocked-flexes}).

\paragraph{Frameworks Constrained to Surfaces}

If we additionally constrain the vertices to lie on a smooth (hyper-)surface, the frameworks' behavior may significantly change. Depending on the type of the surface, the number of trivial infinitesimal flexes of a bar-joint framework on a surface in $\mathbb{R}^3$ can take any value from 0 (for generic surfaces) to 3 (for disjoint unions of planes and concentric spheres) \cite{frameworksonsurfaces}. The condition that the vertices lie on an implicitly-defined surface $\mathcal{M}=\{(x,y,z)\,:\,f(x,y,z)=0\}$ can be imposed via
\[f({\Vector{p}}_u)\,=\,0\quad\quad\text{for all }u\in V.\]
However, there are two issues with this approach in the context of Definition \ref{def:constraint-system-in-Rd}: First, $f$ is not necessarily a quadratic polynomial and it is typically not invariant under the ambient Euclidean isometries. 

Still, the rigid motions of the hypersurface can usually be factored out by constraining a subset of the vertices to certain affine hyperplanes. The number and type of hyperplanes that are necessary is determined by the \struc{type} of the surface, i.e.\ the dimension of its group of isometries arising as isometries of $\mathbb{R}^3$ that act tangentially at every point of $g^{-1}(0)$ (cf.\ \cite{frameworksonsurfaces}). For instance, the unit sphere and the plane have type 3, the cylinder and flat torus have type $2$, the double-cone and the torus of revolution have type $1$ and general surfaces have type $0$.

The heuristics from Section \ref{section:novel-approaches-approximating-deformation-paths} dependent on second-order results are unavailable for the class \juliainline{FrameworkOnSurface}, since it seems likely that the results depend on the chosen hypersurface, which we do not want to restrict at this moment. 

Constructing a (flexible) 4-bar linkage on the one-sheeted hyperboloid is possible using the following commands:
\begin{julia}
F = FrameworkOnSurface([[1,2],[2,3],[3,4],[1,4]], 
	[-sqrt(1/2) -1 0 sqrt(1/2); -sqrt(1/2) 0 1 sqrt(1/2); -1 0 0 1],
	x->x[1]^2+x[2]^2-x[3]^2-1
)
D = DeformationPath(F, [1,1], 100; step_size=0.035)
\end{julia}
You can find a frame from the \juliainline{animate} routine applied to \juliainline{D} in Figure \ref{fig:framework-hyperboloid}(l.). Since the one-sheeted hyperboloid has a 2-dimensional space of rigid motions instead of the usual 3-dimensional space of $\mathbb{R}^2$, the 4-bar linkage has a 2-dimensional space of nontrivial infinitesimal flexes instead of the typical one-dimensional space. We select an equally-weighted linear combination of this linear space's basis vectors by passing the list \juliainline{[1,1]} to the algorithm. 

\begin{figure}[h!]
	\begin{minipage}{0.33\linewidth}
		\includegraphics[width=\linewidth]{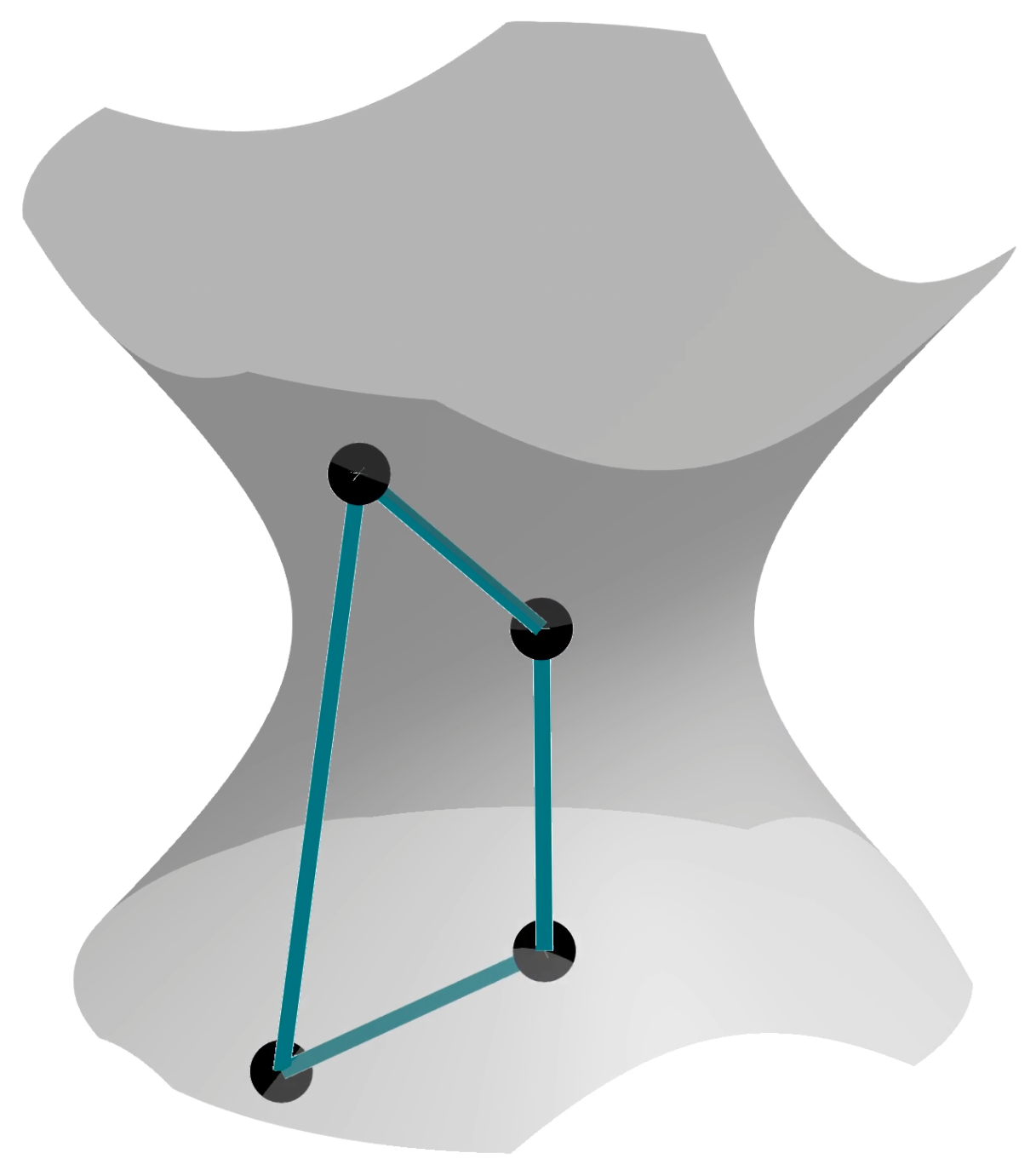}
	\end{minipage}\hspace*{8mm}
	\begin{minipage}{0.55\linewidth}
		\includegraphics[width=\linewidth]{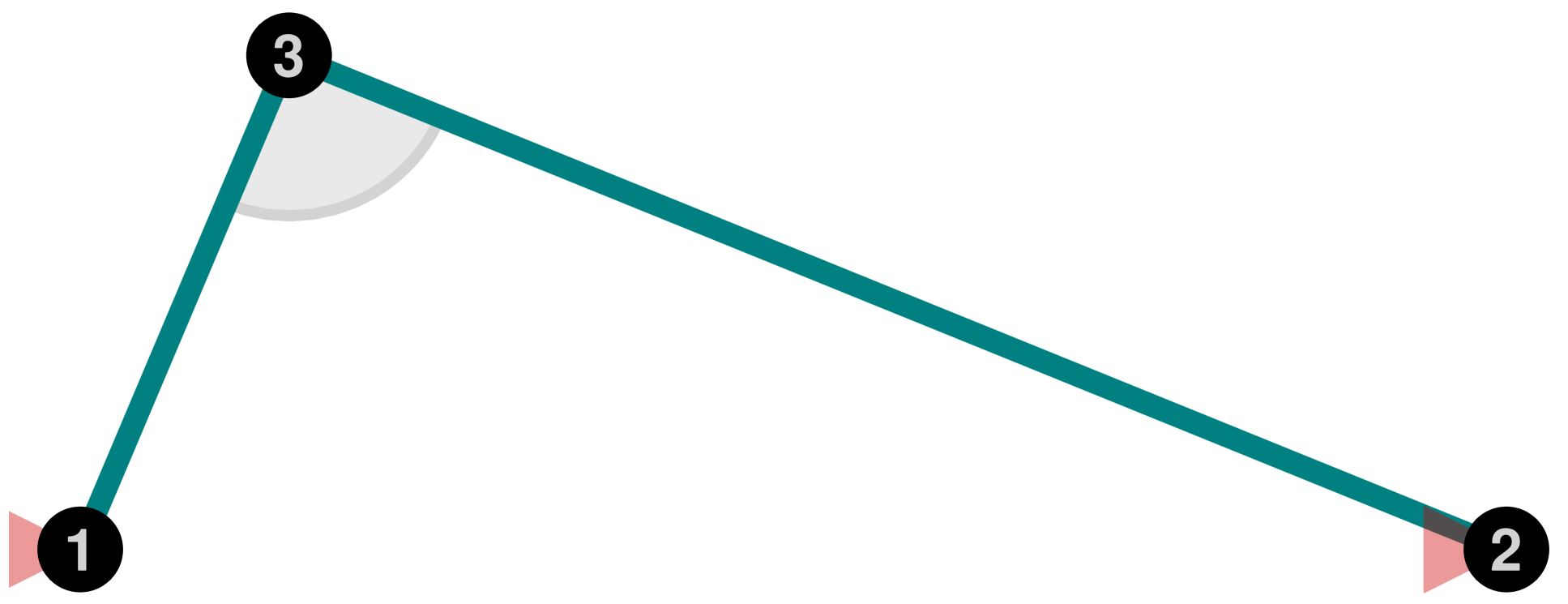}
	\end{minipage}
	
	\caption{\textbf{(left)} A 4-bar linkage constrained to the one-sheeted hyperboloid. \textbf{(right)} A 2-bar linkage with two pinned vertices reproducing Thales' Theorem.}
	\label{fig:framework-hyperboloid}
\end{figure}
\paragraph{Angular Frameworks}

Instead of distances between vertices, we can also consider angles between sets of vertices, leading to \struc{angle-constrained frameworks} (cf. \cite{dewar2024angularconstraintsplanarframeworks}). Such angular frameworks find application in formation control, where they are used to align a multi-agent networks. Similar to bar-joint frameworks, these geometric constraint systems are given by a 3-uniform hypergraphs $G=(V,E)$ with $E\subseteq {\binom{V}{3}}$ and a realization $p:V\rightarrow\mathbb{R}^d$. The underlying polynomial system is given by the angle constraints
\[ \frac{\left\langle(\Vector{p}_u-\Vector{p}_v),~(\Vector{p}_w-\Vector{p}_v)\right\rangle}{\left\lvert\left\lvert \Vector{p}_u-\Vector{p}_v\right\lvert\right\lvert\cdot \left\lvert\left\lvert \Vector{p}_w-\Vector{p}_v\right\lvert\right\lvert}~=~ \frac{\left\langle(p(u)-p(v)),~(p(w)-p(v))\right\rangle}{\left\lvert\left\lvert p(u)-p(v)\right\lvert\right\lvert\cdot \left\lvert\left\lvert p(w)-p(v)\right\lvert\right\lvert}\quad\quad\text{for all hyperedges }uvw\in E.\]
These are algebraic, but not polynomial, equations. \juliainline{HomotopyContinuation.jl} is able to handle such simple non-polynomial structures as well, so that is not an issue. Moreover, the multiplication of the realization by a positive scalar is technically also a trivial motion in this case. Hence, we typically pin two vertices in place to factor out this redundancy.

For instance, Thales' theorem can be reproduced in \juliainline{DeformationPaths.jl} via

\begin{julia}
F = AngularFramework([[1,3,2]], [-1 1 -sqrt(1/2); 0 0 sqrt(1/2)]; [1,2])
D = DeformationPath(F, [1], 250; step_size=0.025)
\end{julia}
which prescribes the angle $\angle (1,3,2)$ and pins the vertices \juliainline{[1,2]}. The resulting angular framework is depicted in Figure \ref{fig:framework-hyperboloid}(r.)

\subsection{Point-Hyperplane Frameworks}
A general \struc{point-hyperplane framework} consists of points and hyperplanes in $\mathbb{R}^d$. The possible corresponding geometric constraints are given by distances between points, between points and hyperplanes and angles between hyperplanes (cf. \cite{EJNSTW}). In the following, we are going to focus on two particular types of geometric constraint systems.

\paragraph{Polytopes with Edge Length Constraints}

\ \struc{Polytopes} provide one of the best-studied geometric objects and which has already fascinated the ancient Greeks, which assigned ethereal properties to the Platonic solids (cf. \cite[p.\ 47]{fuenfplatonischekoerper}). The \struc{combinatorial type} of a 3-dimensional polytope can be described as a tuple $(V,F)$ of vertices $V$ and (abstract) faces $F\subset \mathcal{P}(V)$ so that the following properties are satisfied:
\begin{itemize}
	\item each face $\sigma\in F$ satisfies $|\sigma|\geq 3$,
	\item the intersection $\sigma_1 \cap \sigma_2$ of two facets $\sigma_1,\sigma_2\in F$ may only contain $0$, $1$ or $2$ vertices,
	\item the polytope's edges $E$ are given by 2-element intersections of two facets and
	\item the Euler characteristic count for polytopes $|V|-|E|+|F|=2$ holds.
\end{itemize}
A realization of a 3-dimensional polytope is an embedding $(p,a):V\times F\rightarrow \mathbb{R}^{|V|\cdot 3}\times \mathbb{R}^{|F|\cdot 3}$ so that vertices in the same face $\sigma\in F$ must lie on the same affine hyperplane with facet normal $a$. 

Following Himmelmann, Schulze and Winter \cite{himmelmannschulzewinter2025rigiditypolytopesedgelength}, we additionally fix the edges on the framework $(G_E,\,p)$ on the \struc{edge graph} $G_E=(V,E)$. These properties can be modeled by the following system of equations with variables $(\Vector{p},\Vector{a})$
\begin{eqnarray}
	\label{eq:polynomial-constraints}
\nonumber 	||\Vector{p}_u-\Vector{p}_v||&=&||p(u)-p(v)||\quad~~~\,\,\text{ for all }uv\in E,\\
	0&=&{\Vector{a}_\sigma}^\top\cdot (\Vector{p}_u-\Vector{p}_v)\quad~\,\, \text{ for all }\sigma\in F \text{ and all } u,v\in \sigma \text{ and}\\
	\nonumber 1&=&||{\Vector{a}_\sigma}|| \quad\quad\quad\quad\quad\quad\, \text{ for all }\sigma\in F.
\end{eqnarray}
Note that we implicitly model the facet normals as additional vertices which are constrained to the unit sphere. We characterize the corresponding space of trivial infinitesimal flexes in the following way: Consider a complete bar-joint framework on the set of vertices $V$ with realization $p$ and a complete coned bar-joint framework on the set of facet normals indexed by $F$ with cone point in the origin and realization $a$. Doing so ensures that the realization $p$ may only get rotated and translated, while $a$ may only get deformed in accordance with $p$, but not individually. In addition, we add the facet planarity constraints from the above system for all facets $\sigma \in F$. The only possible degrees of freedom are  Clearly, this is a subvariety of the algebraic variety cut out by the equations in the constraint system (\ref{eq:polynomial-constraints}), which has the same dimension as the group of 3-dimensional Euclidean transformations.

Even the rigidity properties of the Platonic solids under these constraints are not completely understood. Although it has already been known for 200 years that the tetrahedron, octahedron and icosahedron are infinitesimally rigid due to Cauchy's Theorem \cite{cauchytheorem} and despite the cube is known to be flexible as a Zonotope, the regular dodecahedron has only been announced to be (second-order) rigid rather recently (cf.\ \cite{himmelmannschulzewinter2025rigiditypolytopesedgelength}).
In the context of this article, a regular dodecahedron may either be instantiated via the \juliainline{Polytope} class or by using the following code, which accesses the internal database of geometric constraint systems:

\begin{julia}
Dod = Dodecahedron()
\end{julia}
Generic realizations of convex polytopes are rigid (cf.\ \cite{himmelmannschulzewinter2025rigiditypolytopesedgelength}). Nevertheless, there are known instances of flexible polytopes, such as Bricard's octahedron (cf. \cite{raoulbricard}) or even the regular cube. Understanding the possible deformations of polytopes thus provides an intriguing research question.

\begin{figure}[h!]
	\begin{minipage}{0.35\linewidth}
		\includegraphics[width=\linewidth]{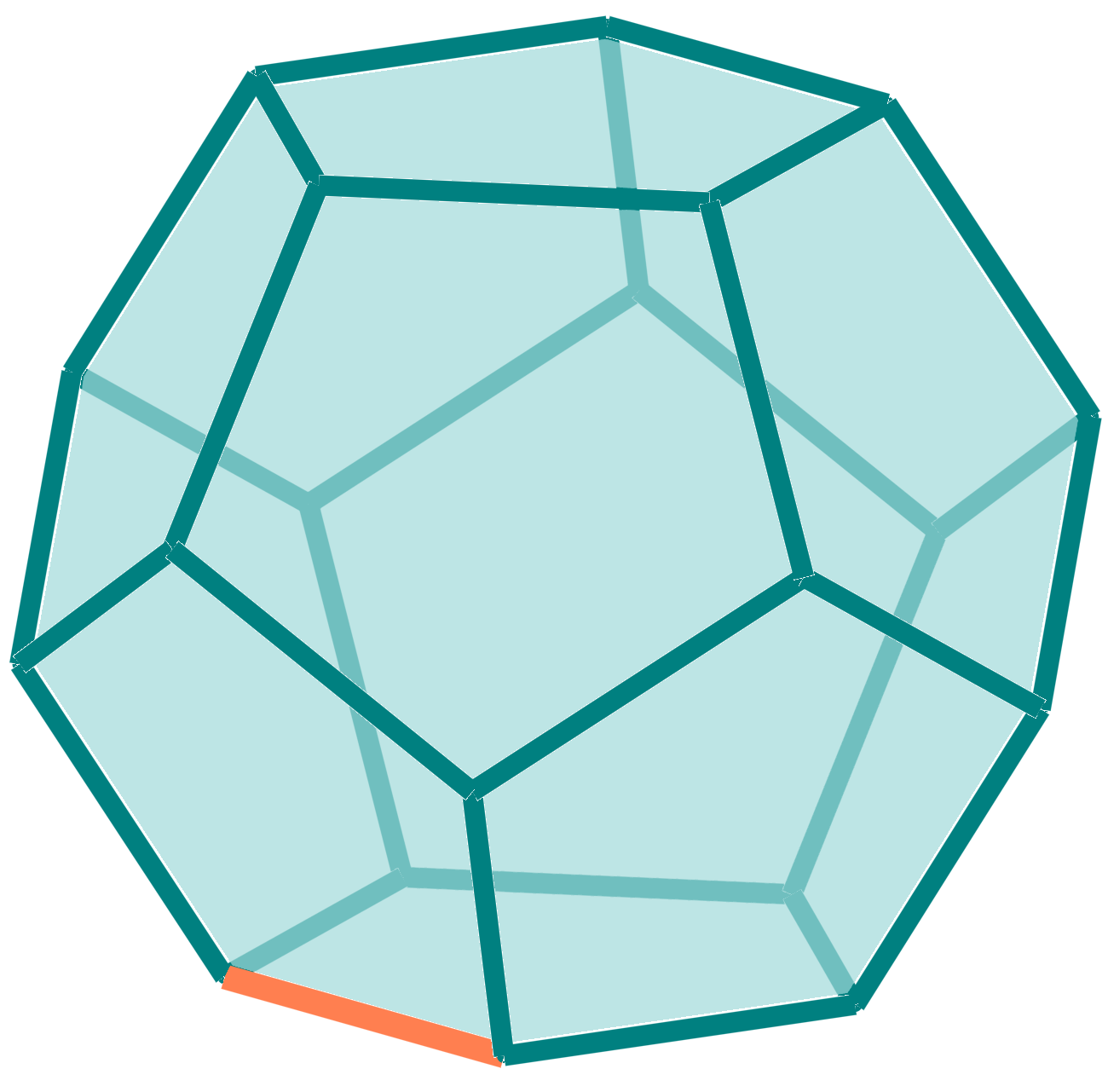}
	\end{minipage}\hspace*{14mm}
	\begin{minipage}{0.4\linewidth}
		\includegraphics[width=\linewidth]{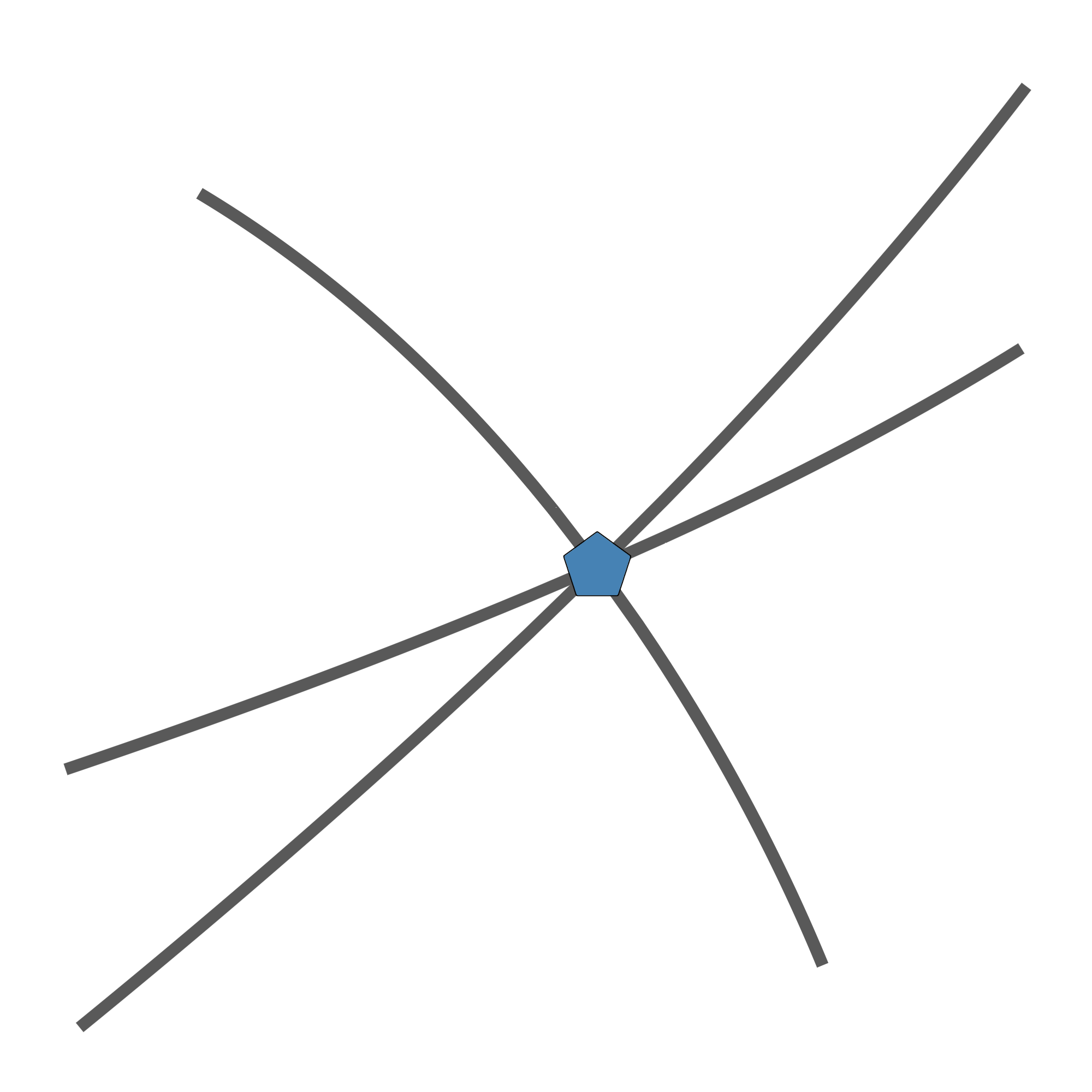}
	\end{minipage}
	\caption{\textbf{(left)} The regular dodecahedron is second-order rigid (cf. \cite{secondorderpolytoperigid}). Nevertheless, we can consider continuous motions induced by contracting edges. The Platonic solids are edge-transitive, so all edges are combinatorially equivalent. Here, we contract the length of the coral edge.  \textbf{(right)} The associated \textcolor{NiceBlue}{\textit{edge length perturbation space}} gives rise to all sufficiently close realizations corresponding to contractions and expansions of edge lengths. We randomly project the curves resulting from contracting the highlighted edge to $\mathbb{R}^2$. This space contains a singularity in the regular dodecahedron; nevertheless, it gives rise to 3 continuous motions that can be smoothly matched (cf. \cite{secondorderpolytoperigid}).}
	\label{fig:regular-dodec}
\end{figure}
Furthermore, even though the regular dodecahedron is rigid, it can be deformed by contracting one of its edges (cf. \cite{secondorderpolytoperigid}) through a parametrization in terms of $t$ of the equation corresponding to the edge $u'v'\in E$ via
\begin{eqnarray*}	||\Vector{p}_{u'}-\Vector{p}_{v'}||~=~t
\end{eqnarray*}
with start value $t_0=||p(u')-p(v')||$. By decreasing the value of $t$, this parametrizes a continuous motion of the otherwise rigid polytope with one edge of variable length.

In the Julia package \juliainline{DeformationPaths.jl}, this can be achieved by calling
\begin{julia}
	D = DeformationPath_EdgeContraction(Dod, [9, 10], 0.75)
\end{julia}
on the previously defined regular dodecahedron \juliainline{Dod}. This reduces the length of the edge \juliainline{[1,2]} of \juliainline{Dod} to $75\%$ of its original length. Naturally, the same code can be used for computing continuous motions induced by extensions of this edge. One frame from the deformation \juliainline{D} is depicted in Figure \ref{fig:tetrahedron-and-bodyhinge}

\paragraph{Body-Hinge Frameworks}
Similar to polytopes, \struc{body-hinge frameworks} in $\mathbb{R}^d$ are given by arrangement of flat $(d-1)$-dimensional facets in space. The main difference is that the facets here are rigid and are connected by \struc{hinges}, which forces two adjacent bodies to share a common $(d-2)$-dimensional facet so that the bodies can only move relative to each other by performing a rotation around this facet. For instance, in $\mathbb{R}^3$ two 2-dimensional facets can rotate around a common edge.

The combinatorial data of a body-hinge framework is given by a hypergraph $G=(V,F)$, where the facets or rigid bodies in $F$ can consist of at least $d+1$ vertices. We model a rigid body given by the convex hull of $n$ points as a complete graph $K_n$ on $n$ vertices. It is known that $K_n$ is universally rigid, meaning that it is rigid in any dimension \cite{universalrigidity}, implying that it is suitable for this purpose. The corresponding geometric constraint system thus consists of the hypergraph $G$, a realization $p:V\rightarrow \mathbb{R}^d$ and polynomial constraints of the form

\begin{eqnarray*}
	||\Vector{p}_u-\Vector{p}_v||\,=\,||p(u)-p(v)||\quad\quad\text{ for all }\sigma\in F \text{ and }uv\in \binom{\sigma}{2}.
\end{eqnarray*}
These constraints should remind the reader of the definition of a bar-and-joint framework, as they represent a special case.

\begin{figure}[h!]
	\begin{minipage}{0.375\linewidth}
		\includegraphics[width=\linewidth]{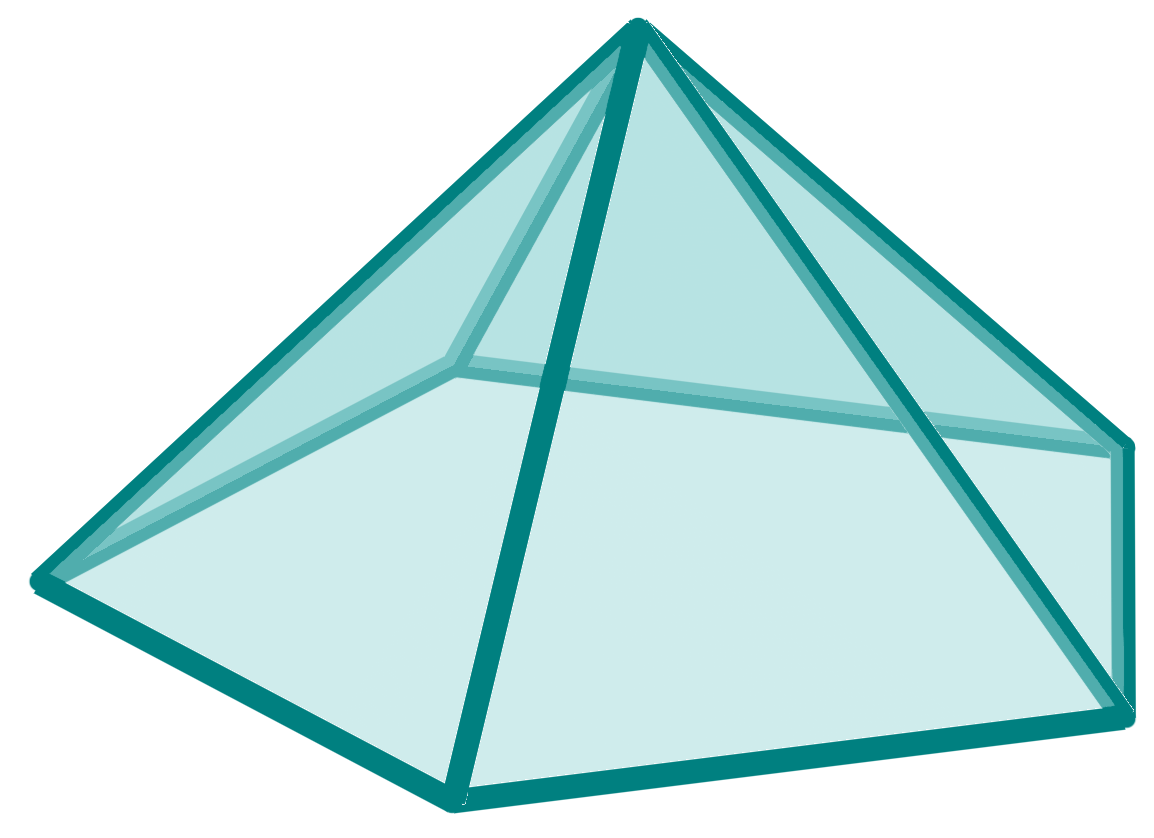}
	\end{minipage}\hspace*{21mm}
	\begin{minipage}{0.325\linewidth}\includegraphics[width=\linewidth]{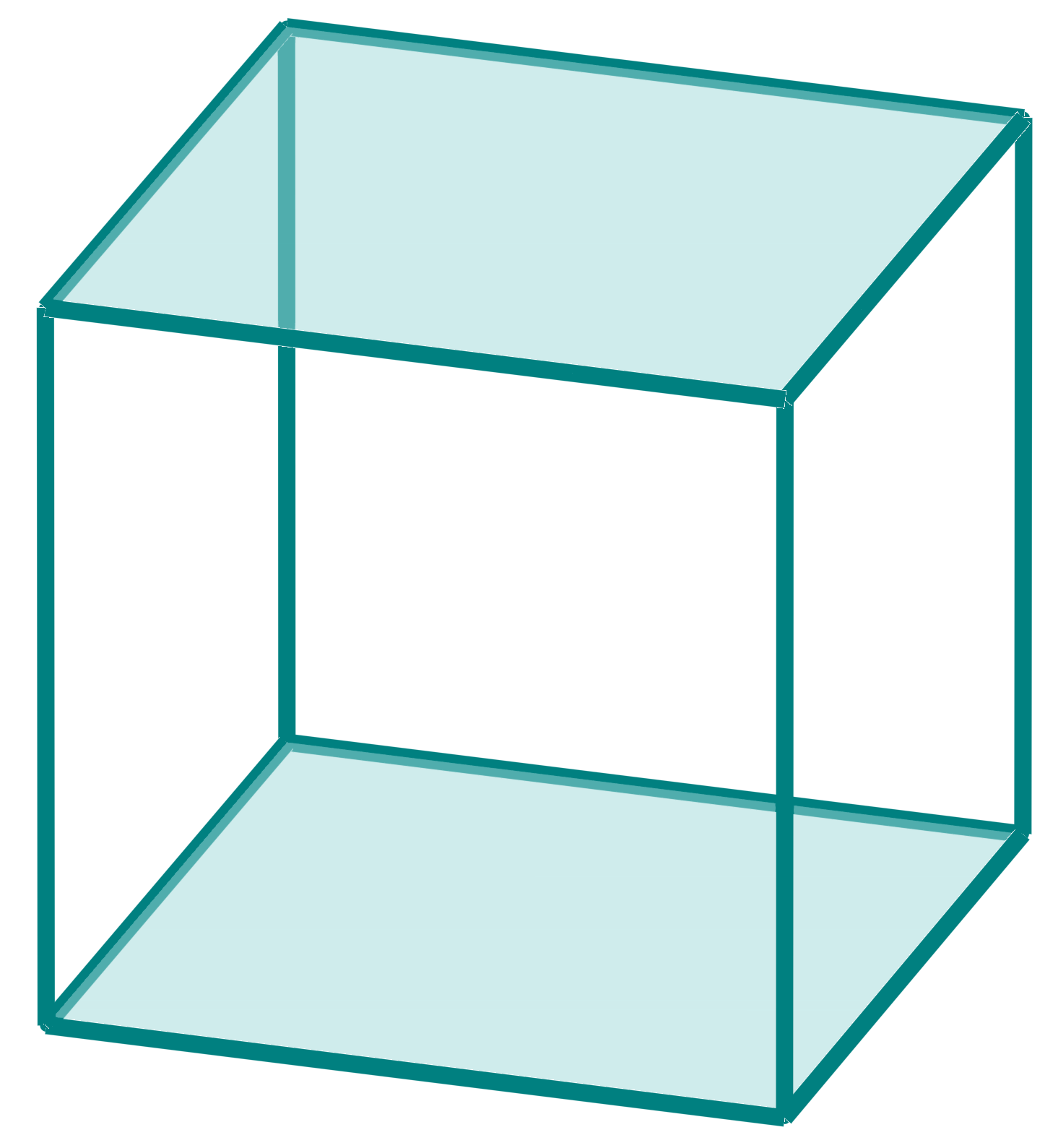}
	\end{minipage}
	\caption{\textbf{(left)} Body-hinge framework consisting of five triangular panels that are arranged around a single vertex in the shape of a pentagonal pyramid. \textbf{(right)} Body-bar framework of two panels that are connected by 4 bars in the shape of a regular cube.}
	\label{fig:tetrahedron-and-bodyhinge}
\end{figure}

In the Julia package, we can construct a body-hinge framework consisting of six triangular panels that are arranged in a pentagonal pyramid around a single vertex (cf. Figure \ref{fig:tetrahedron-and-bodyhinge}(l.)) by calling
\begin{julia}
F = BodyHinge(
	[[1,2,3],[1,3,4],[1,4,5],[1,5,6],[1,6,2]],
	[0 cos(2*pi/5) cos(4*pi/5) cos(6*pi/5) cos(8*pi/5) cos(10*pi/5); 
	0 0 sin(2*pi/5) sin(4*pi/5) sin(6*pi/5) sin(8*pi/5) sin(10*pi/5); 
	1 0 0 0 0 0;]
)
\end{julia}
\paragraph{Body-Bar Frameworks}
\ \struc{Body-bar Frameworks} are slightly more general geometric constraint systems than body-hinge frameworks: they not only allow hinge-like connections between rigid bodies, but also edges \cite{bodybarframeworks}. The combinatorial structure is given by a triple $G=(V,E,F)$ so that the edges satisfy $E\subset\binom{V}{2}$ and the facets again each contain at least 3 vertices. The corresponding geometric constraint system is thus given by the the combinatorial data $G$, a realization $p:V\rightarrow \mathbb{R}^d$ and the polynomial constraints
\begin{eqnarray*}
	||\Vector{p}_u-\Vector{p}_v||\,&=&\,||p(u)-p(v)||\quad\quad\text{ for all }\sigma\in F \text{ and }uv\in \binom{\sigma}{2} ~\text{ and }\\[-1.5mm]	||\Vector{p}_u-\Vector{p}_v||\,&=&\,||p(u)-p(v)||\quad\quad\text{ for all }uv\in E.
\end{eqnarray*}
A body-bar framework consisting of two square panels, whose vertices are connected by 4 bars can be instantiated via
\begin{julia}
F = BodyBar(
	[[1,5],[2,6],[3,7],[4,8]], [[1,2,3,4],[5,6,7,8]],
	[0 1 1 0 0 1 1 0; 0 0 1 1; 0 0 1 1; 0 0 0 0 1 1 1 1;]
)
\end{julia} 
and is depicted in Figure \ref{fig:tetrahedron-and-bodyhinge}(r.).

\subsection{Sticky Sphere Packings}
\label{section:sticky-sphere-packings}

\begin{figure}[t!]
	\begin{minipage}{0.44\linewidth}
		\includegraphics[width=\linewidth]{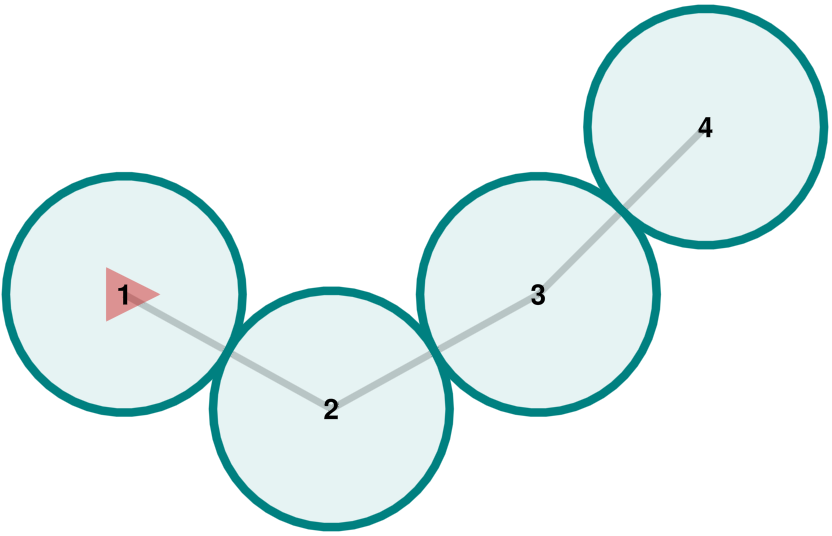}
	\end{minipage}\hspace*{17mm}
	\begin{minipage}{0.335\linewidth}
		\includegraphics[width=\linewidth]{Images/hypostaticspherepacking.png}
	\end{minipage}
	\caption{\textbf{(left)} A sticky circle packing with 4 disks that has the 4-bar linkage as a limit case. \textbf{(right)} Rigid sphere packing with fewer contacts than expected.}
	\label{fig:sphere-packings}
\end{figure}

For our purposes, a collection of identical hyperspheres of radius $r>0$ is called a \struc{hypersphere packing in $\mathbb{R}^d$} if the pairwise distances between the hyperspheres' centers is at least $2r$. Typically, we consider $d=2$, in which case we call this object a \struc{circle packing}, and $d=3$, in which case we refer to it as a \struc{sphere packing}. We call a packing of spherical objects \struc{sticky} if whenever two hyperspheres are touching -- implying that their centers have distance $2r$ -- they stay connected. In the real world, this stickiness models cohesive oder adhesive forces which can occur through electrostatics.

This problem can be modeled as a bar-joint framework. To place the spherical objects in $\mathbb{R}^d$, we consider their centers as the vertex set $V$ and embed them via the realization $p:V\rightarrow \mathbb{R}^d$. We additionally assume that for no two vertices $i\neq j\in V$ it holds that $||p(i)-p(j)||$. In other words, no two spherical objects are overlapping. The edge set is then defined by all connections that have a radius of exactly $2r$ (up to numerical imprecisions): \[E=\left\{uv\in \binom{V}{2}\,:\, ||p(u)-p(v)||=2r\right\}.\]
These edges provide the polynomial constraints for the corresponding geometric constraint system. Contrary to the other geometric constraint systems that we have covered so far, there is a peculiarity here: The constraints may change throughout a continuous motion. If during a continuous motion $\alpha:[0,1]\rightarrow\mathbb{R}^{dn}$ any of the spherical objects' centers become closer than $2r$, we momentarily stop the continuous motion computation. Subsequently, these disks ``stick together'' by creating an additional edge between these spherical objects. With this new geometric constraint system, we continue the computation of the continuous motion. Notably, the resulting continuous motion will also be a continuous motion of the original geometric constraint system, since new edges correspond to adding equations. 
To generate a circle packing with $n=4$ and $r=1$, we can run

\begin{julia}
F = SpherePacking(
	[1,1,1,1], 
	[0 7/4 7/2 7/2+sqrt(2); 0 -sqrt(15/16) 0 sqrt(2)];
	pinned_vertices=[1]
)
\end{julia}

This code pins the first vertex of the geometric constraint system in place and automatically generates the necessary edges. The specific sphere packing \juliainline{F} constructed by this code has a 2-dimensional space of nontrivial infinitesimal flexes, has the 4-bar linkage as a limit case and is depicted in Figure \ref{fig:sphere-packings}(l.).

Beyond 2-dimensional circle packings, we can also use the codebase to visualize 3-dimensional sphere packings. In Figure \ref{fig:sphere-packings}(r.), a \struc{hypostatic} sphere packing is depicted. Its coordinates are taken from Holmes-Cerfon, Theran and Gortler \cite{almostrigidity}. The term ``hypostatic'' describes a rigid sphere packing which has fewer contacts than anticipated. Assuming it was instantiated as an object called \juliainline{H3}, we can check that it is not infinitesimally rigid by calling \juliainline{is_inf_rigid(H3)}. Yet, it is rigid, which we can be verified by calling either \juliainline{is_second_order_rigid(H3)} or \juliainline{is_rigid(H3)}.

\subsection{Volume-Constrained Simplicial Complexes}
It is possible to generalize the idea of constraining the length of bars in bar-joint frameworks to constraining the volume of simplices in simplicial complexes. We start with a simple $(d+1)$-uniform hypergraph $G=(V,E)$, so that $E\subseteq {\binom{V}{d}}$. We label the vertices as $1,\dots,n$ and the hyperedges in increasing order; i.e.\ , $\sigma=i_1\dots i_{d+1}$ for $1\leq i_1\leq\dots\leq i_{d+1}\leq n$ for any  $\sigma\in E$.

$G$ can be interpreted as a pure $d$-dimensional (abstract) simplicial complex if we ignore lower-dimensional faces. Therefore, \struc{volume-constrained simplicial complexes} are given by a simple $(d+1)$-uniform hypergraph $G$, a realization $p:V\rightarrow \mathbb{R}^d$ and polynomial constraints of the form
\begin{align*}
	\det\begin{pmatrix}
		\Vector{p}_{\sigma_1} & \dots & \Vector{p}_{\sigma_{d+1}}\\
		1 & \dots & 1
	\end{pmatrix}~=~\,\det\underbrace{\begin{pmatrix}
			p(\sigma_1) & \dots & p(\sigma_{d+1})\\
			1 & \dots & 1
	\end{pmatrix}}_{\in\, \mathbb{R}^{(d+1)\times(d+1)}}\quad~~\text{ for all }\sigma\in E.
\end{align*}
These polynomials fix the volume of each $d$-simplex according to the given realization $p$ (cf. \cite{BULAVKA2025189}). 

The trivial infinitesimal flexes of a volume-constrained simplicial complex are exactly given by the infinitesimal flexes of the complete volume-constrained simplicial complex on the complete $(d+1)$-uniform hypergraph with $n$ vertices and realization $p:V\rightarrow \mathbb{R}^d$. Contrary to the case of bar-joint frameworks, where there are $d(d+1)/2$ trivial infinitesimal flexes, there are $d^2+d-1$ in this case (cf. \cite{BULAVKA2025189}), which is not lower than in the framework case. Aside from rotations, reflections and translations, the volume-preserving transformations of $\mathbb{R}^d$ additionally contain shears. In addition, the constraints are not given by quadratic polynomials when $d\geq 3$, so this geometric constraint system does not satisfy either assumption of Definition \ref{def:constraint-system-in-Rd}. 

Regardless, volume-constrained hypergraphs are included in the package. By pinning the position of the vertices of a $d$-simplex $\sigma\in E$, we can remove $d(d+1)$ degrees of freedom and consequently factor out the trivial infinitesimal flexes. According to Remark \ref{remark:quadraticpolynomials}, it is unclear whether we can use equilibrium stresses to block certain infinitesimal flexes. Therefore, the functionality derived from the second-order theory of geometric constraint systems (cf. Definition \ref{def:stress-and-second-order-rigidity}) is currently unavailable in this case.

\begin{figure}
	\centering
	\includegraphics[width=0.375\linewidth]{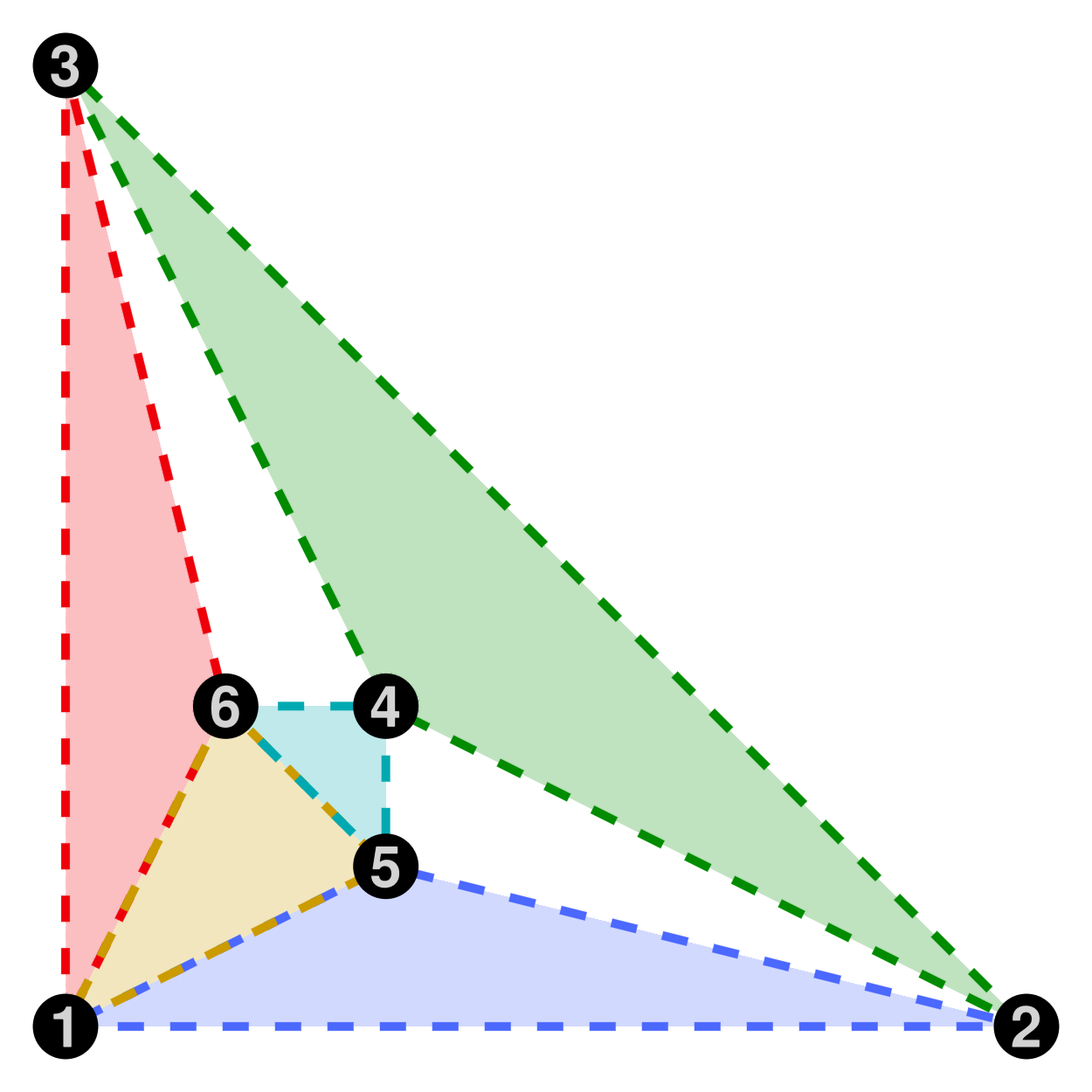}\hspace*{8mm}
	\includegraphics[width=0.415\linewidth]{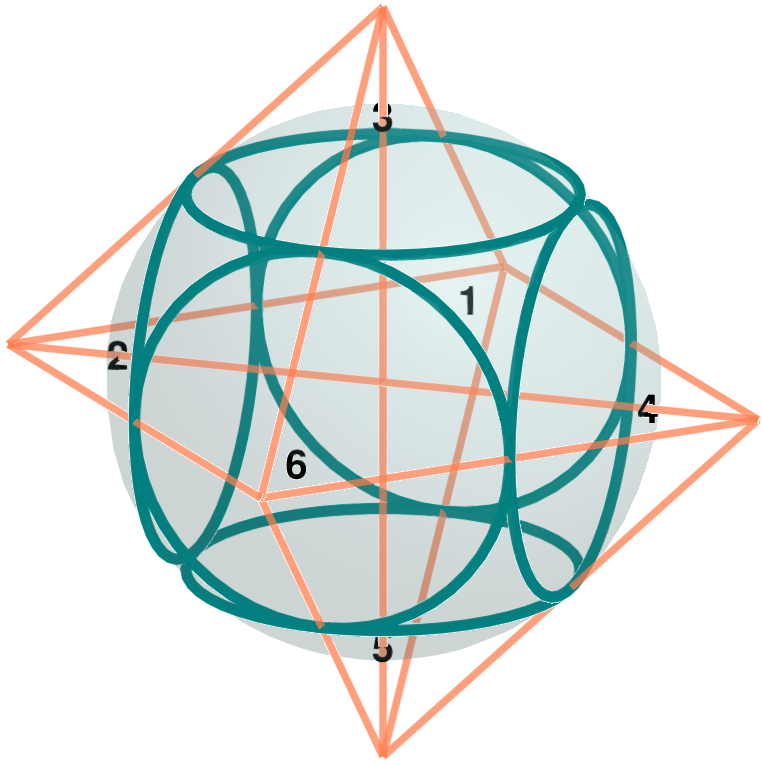}
	\caption{\textbf{(left)} When removing two faces from the octahedral hypergraph, it becomes flexible. \textbf{(right)} Flexible spherical disk packing obtained from Bricard's octahedron. The sphere packing is colored in teal, while the corresponding labeled contact graph is marked in coral.}
	\label{fig:volume-constrained-hypergraph}
\end{figure}

As an instructive example, let us consider an octahedral hypergraph in $\mathbb{R}^2$ with two facets removed. In the package \juliainline{DeformationPaths.jl}, we can construct a realization by calling
\begin{julia}
F = VolumeHypergraph(
	[[1,3,6], [1,2,5], [2,3,4], [1,5,6], [6,4,5]], 
	[0 3 0 1 1 0.5; 0 0 3 1 0.5 1];
	pinned_vertices = [4,5,6]
)
\end{julia}
which pins the interior triangle. This geometric constraint system is flexible and is depicted in Figure \ref{fig:volume-constrained-hypergraph}(l.).

\subsection{Spherical Circle Packings}

Similar to Section \ref{section:sticky-sphere-packings}, assume that we are given a circle packing. However, instead of considering a circle packing in $\mathbb{R}^d$, we constrain it to the unit sphere $\mathbb{S}^2$. Also, we consider the \struc{inversive distance} as the measurement map on $\mathbb{S}^2$ instead of the Euclidean norm. The inversive distance between two spherical caps on $\mathbb{S}^2$ with centers $\Vector{m}_1\in \mathbb{S}^2$ and $\Vector{m}_2\in \mathbb{S}^2$ and geodesic radii $r_1$ and $r_2$ is given by (cf. \cite{inversivedistance})
\[I(\Vector{m}_1,\Vector{m}_2,r_1,r_2)\, =\, \frac{-\langle \Vector{m}_1,\Vector{m}_2\rangle+\cos(r_1)\cos(r_2)}{\sin(r_1)\sin(r_2)}.\]
The above expression for the inversive distance is not algebraic. To turn it into an algebraic expression, we consider the circles to be intersections of $\mathbb{S}^2$ with affine hyperplanes instead of circles in space. Take a single circle with center $\Vector{m}$ and geodesic radius $r>0$ on $\mathbb{S}^2$. We parametrize the corresponding affine hyperplane by projective coordinates \mbox{$[a:b:c]\in \mathbb{R}\mathbb{P}^2$} and $d\in \mathbb{R}$ with the implicit equation $ax+by+cz=d$. Consequently, the circle's center lies on the ray $t\cdot (a,b,c) \in \mathbb{R}^3$ for some $t\in \mathbb{R}$. We can then normalize the parameters $a,b,c,d$ with $\sqrt{a^2+b^2+c^2}\neq 0$ so that the plane's normal has unit length. Therefore, it holds that \mbox{$(a,b,c)=\Vector{m}$}. Using the chord length formula on the unit sphere, we can compute $d^2=1-\sin(\frac{r}{2})^2=\cos(\frac{r}{2})^2$ so that $d\in [-1,1]$. According to Bowers \cite{bowers2020proofkoebeandreevthurstontheoremflow}, the inversive distance of two circles given by $\Vector{D}_1=(a_1,b_1,c_1,d_1)\in \mathbb{S}^2\times [-1,1]$ and $\Vector{D}_2=(a_2,b_2,c_2,d_2)\in \mathbb{S}^2\times [-1,1]$ satisfies 
\[I(\Vector{D}_1,\Vector{D}_2)\,=\, \frac{\langle \Vector{D}_1,\Vector{D}_2\rangle_{3,1}}{\sqrt{\big\lvert \langle \Vector{D}_1,\Vector{D}_1\rangle_{3,1}\cdot \langle \Vector{D}_2,\Vector{D}_2\rangle_{3,1} \big\lvert}}\]
with the \struc{Lorentz inner product} 
\[\langle(a_1,b_1,c_1,d_1), \, (a_2,b_2,c_2,d_2)\rangle_{3,1}\,=\, a_1a_2+b_1b_2+c_1c_2-d_1d_2.\]
Still, the value of $I(\Vector{D}_1,\Vector{D}_2)$ is undefined whenever $\langle \Vector{D}_i,\Vector{D}_i\rangle_{3,1}=0$ for $i\in \{1,2\}$. This exactly occurs when $d_i=\pm 1$, where the radius of the respective spherical cap is $0$. In this case, we set the inversive distance to $\infty$, provided that $\Vector{D}_1\neq \Vector{D}_2$. We set $I(\Vector{D}_i,\Vector{D}_i)=0$, when the circles are equal. In addition, we assume without loss of generality that $d_i\geq 0$. This chooses a positive orientation for the circle and ensures that $\langle \Vector{D}_i,\Vector{D}_i\rangle_{3,1}\geq 0$. In geometric terms, we choose the smaller of the two possible spherical caps in this way. Since we only consider the unoriented circle and not the entire spherical cap, this choice is no restriction.

If the inversive distance between two circles is equal to $\pm 1$, then the circles are tangent.
We prescribe the combinatorial data associated with this geometric constraint system as a labeled contact graph $G=(V,E,I:E\rightarrow \mathbb{R}_{>0})$ which prescribes the (absolute) inversive distances between pairs of circles. The vertices are then given by the implicit description of the affine hyperplane which cuts it out, giving rise to a realization $p:V\rightarrow \mathbb{S}^2\times \mathbb{R}\subseteq \mathbb{R}^4$. This produces the polynomial constraint system
\begin{eqnarray*}
	\frac{{\langle \Vector{p}_u,\Vector{p}_v\rangle_{3,1}}^2}{\langle \Vector{p}_u,\Vector{p}_u\rangle_{3,1} \cdot \langle \Vector{p}_v,\Vector{p}_v\rangle_{3,1}} &=&\frac{{\langle p(u),p(v)\rangle_{3,1}}^2}{\langle p(u),p(u)\rangle_{3,1} \cdot \langle p(v),p(v)\rangle_{3,1}}\quad\quad\text{ for all }uv\in E\text{ and}\\[1mm]
	1&=&{\Vector{p}_{u,1}}^2+{\Vector{p}_{u,2}}^2+{\Vector{p}_{u,3}}^2\quad\quad\quad\quad\quad\quad\text{ for all }u\in V.
\end{eqnarray*}
Analogous to the construction of non-convex simplicial polyhedra by Bricard \cite{raoulbricard}, there exists a construction of flexible sphere packings (e.g.\ \cite{inversivedistance}). We can construct the analogue of Bricard's octahedron (see Figure \ref{fig:framework-introductory-pic}(l.)) as an instance of \juliainline{DeformationPaths.jl} by calling
\begin{julia}
F = SphericalDiskPacking(
	[(1,2), (1,3), ..., (4,6), (5,6)],
	[1 ... -1; 0 ... 0; 0 ... 0; 1/sqrt(2) ... 1/sqrt(2)]; 
	pinned_vertices=[1]
)
\end{julia}
The resulting flexible configuration alongside its labeled contact graph is depicted in Figure \ref{fig:volume-constrained-hypergraph}(r.).

\end{document}